\documentclass[a4paper,12pt, twoside, openright]{book}
\usepackage[utf8]{inputenc}
\usepackage[UKenglish]{babel}
\usepackage [pdftex]{graphicx}
\usepackage [bf, small]{caption}
\usepackage{amsmath}
\usepackage{amsthm}
\usepackage{amssymb}
\usepackage {caption}
\usepackage{fancyhdr}
\usepackage{wrapfig}
\usepackage{subfigure}
\usepackage{newlfont}
\usepackage{hyperref}
\usepackage{verbatim}
\usepackage{lmodern,textcomp}
\usepackage{fp}
\usepackage{booktabs}
\usepackage{ragged2e}
\usepackage{longtable}
\usepackage{mathtools}
\usepackage{tikz}
\usepackage{leftidx}
\usepackage{mathabx}
\usepackage{color}
\usepackage{enumitem}
\usepackage[makeroom]{cancel}
\usepackage{multirow,bigdelim}
\usepackage{array}
\usepackage{hhline}
\usepackage{colortbl}
\usepackage{amsfonts}
\usepackage{mathptmx}
\usepackage{calc}
\usepackage{multicol}
\usetikzlibrary{tikzmark,calc,,arrows,shapes,decorations.pathreplacing}
\tikzset{every picture/.style={remember picture}}
\usepackage{accents}
\usepackage{tikz-cd}
\usepackage[most]{tcolorbox}
\usepackage{lmodern}
\usepackage{xparse}
\usepackage{tikz}
\usetikzlibrary{shadows,arrows}
\pgfdeclarelayer{background}
\pgfdeclarelayer{foreground}
\pgfsetlayers{background,main,foreground}
\tikzstyle{materia}=[draw, fill=blue!20, text width=6.0em, text centered,
  minimum height=1.5em,drop shadow]
\tikzstyle{practica} = [materia, text width=8em, minimum width=10em,
  minimum height=3em, rounded corners, drop shadow]
\tikzstyle{appe} = [materia, text width=8em, minimum width=10em,
  minimum height=3em, rounded corners, drop shadow]
\tikzstyle{texto} = [above, text width=6em, text centered]
\tikzstyle{linepart} = [draw, thick, color=black!50, -latex', dashed]
\tikzstyle{line} = [draw, thick, color=black!50, -latex']
\tikzstyle{ur}=[draw, text centered, minimum height=0.01em]
\NewDocumentCommand{\diamondinclusion}{m >{\SplitArgument{1}{\\}}m m}
 {%
  \dodiamondinclusion{#1}#2{#3}%
 }
\NewDocumentCommand{\dodiamondinclusion}{mmmm}
 {%
  \begingroup
  \setlength{\arraycolsep}{0pt}%
  \renewcommand{\arraystretch}{-2}%
  \begin{matrix}
  && #2 \\
  & \rsubsetneq{45} && \rsubsetneq{-45} \\
  #1 &&&& #4 \\
  & \rsubsetneq{-45} && \rsubsetneq{45} \\
  && #3
  \end{matrix}%
  \endgroup
 }
\NewDocumentCommand{\rsubsetneq}{m}
 {%
  \rotatebox[origin=c]{#1}{$\subsetneq$}%
 }

\newcommand{\practica}[2]{node (p#1) [practica]
  {Chapter #1}}
\newcommand{\appe}[2]{node (ap#1) [appe]
  {Appendix #2  }}

\newcommand{\mycomment}[1]{%
}

\newcommand{\norm}[1]{\left\lVert#1\right\rVert}
\newcommand{\vertiii}[1]{{\left\vert\kern-0.25ex\left\vert\kern-0.25ex\left\vert #1          \right\vert\kern-0.25ex\right\vert\kern-0.25ex\right\vert}}

\makeatletter
\newcommand\mathcircled[1]{%
  \mathpalette\@mathcircled{#1}%
}
\newcommand\@mathcircled[2]{%
  \tikz[baseline=(math.base)] \node[draw,circle,inner sep=1pt] (math) {$\m@th#1#2$};%
}
\makeatother
\allowdisplaybreaks 
\DeclareMathAlphabet{\mathcal}{OMS}{cmsy}{m}{n}

\DeclareMathOperator{\Ker}{Ker}
\DeclareMathOperator{\Ima}{Im}
\DeclareMathOperator{\spn}{span}
\theoremstyle{plain}
\theoremstyle{definition}
\newtheorem{defin}{Definition}[section]
\newtheorem{prop}[defin]{Proposition}

\newtheorem{obs}[defin]{Observation}
\newtheorem{lem}[defin]{Lemma}
\newtheorem{ex}[defin]{Example}

\newtheorem{no}[defin]{Notation}
\newtheorem{theor}[defin]{Theorem}
\newtheorem{rec}[defin]{Recall}
\newtheorem{rem}[defin]{Remark}
\newtheorem{que}[defin]{Question}

\begin{document}
%


\begin{titlepage}
\begin{center}
{{\Large{\textsc{University of Helsinki}}}}\\
{\small{\textsc{ Department of Mathematics and Statistics}}}
\end{center}
\vspace{15mm}
\rule[0.1cm]{15.8cm}{0.1mm}
\rule[0.5cm]{15.8cm}{0.6mm}
\begin{center}
{\LARGE{Insight in the Rumin Cohomology and}}\\
\vspace{3mm}
{\LARGE{Orientability Properties of the}}\\
\vspace{3mm}
{\LARGE{Heisenberg Group}}\\
\end{center}
\rule[0.1cm]{15.8cm}{0.6mm}
\rule[0.5cm]{15.8cm}{0.1mm}
\begin{center}
\vspace{15mm}
\vspace{19mm} {\large{\textsc{ Giovanni Canarecci}}}
\end{center}
\vspace{36mm}
\vspace{20mm}
\begin{center}
{\large{  Licentiate Thesis\\    
November 2018}}
\end{center}
\end{titlepage}


\newpage
~\vfill
\thispagestyle{empty}

\noindent
\textbf{Insight in the Rumin Cohomology and Orientability Properties of the Heisenberg Group}\\
Licentiate Thesis\\\\

\noindent
\textbf{Abstract:}\\
\noindent
The purpose of this study is to analyse two related topics: the Rumin cohomology and the $\mathbb{H}$-orientability in the Heisenberg group $\mathbb{H}^n$. 
In the first three chapters we carefully describe the Rumin cohomology with particular emphasis at the second order differential operator $D$, giving examples in the cases $n=1$ and $n=2$. We also show the commutation between all Rumin differential operators and the pullback by a contact map and, more generally,  describe pushforward and pullback explicitly in different situations. 
Differential forms can be used to define the notion of orientability; indeed in the fourth chapter we define the $\mathbb{H}$-orientability for $\mathbb{H}$-regular surfaces and we prove that $\mathbb{H}$-orientability implies standard orientability, while the opposite is not always true. Finally we show that, up to one point, a Möbius strip in $\mathbb{H}^1$ is a $\mathbb{H}$-regular surface and we use this fact to prove that there exist $\mathbb{H}$-regular non-$\mathbb{H}$-orientable surfaces, at least in the case $n=1$. This opens the possibility for an analysis of Heisenberg currents mod $2$.\\\\

\noindent
\textbf{Contact information}\\\\ \indent
Giovanni Canarecci\\ \indent
email address: giovanni.canarecci@helsinki.fi\\\\ \indent
Office room: B418\\ \indent
Department of Mathematics and Statistics\\ \indent
P.O. Box 68 (Pietari Kalmin katu 5)\\ \indent
FI-00014 University of Helsinki\\ \indent
Finland\\\\




\newpage
\begin{flushright}
\null\vspace{\stretch{3}}
\emph{}\\                                                                     
``Beato colui che sa pensare al futuro senza farsi prendere dal panico"\\
\vspace{2cm}
\noindent
``Lucky  is he who can think about the future without panicking"
\emph{}
\vspace{\stretch{4}}\null
\end{flushright}


%
\tableofcontents 
%

%
%
\vspace{2cm}
\begin{tikzpicture}[scale=0.7,transform shape]
  \path \practica {1}{Diferencias en componentes electr\'onicos};
  \path (p1.south)+(0.0,-1.1) \practica{2}{Serie de Fourier};
  \path (p2.south)+(0.0,-1.1) \practica{3}{Antena para HF};
  \path (p3.south)+(0.0,-1.1) \practica{4}{Amplificador para HF};
  \path (p2.west)+(-5,0.0) \appe{1}{A, B and C};
  \path (p3.west)+(-5,0.0) \appe{2}{D};
  \path (p4.west)+(-5,0.0) \appe{3}{E};
  \path [line] (p1.south) -- node [above] {} (p2);
  \path [line] (p2.south) -- node [above] {} (p3);
  \path [line] (p3.south) -- node [above] {} (p4) ;
  \path [linepart] (ap1.east) -- node [left] {} (p2);
  \path [linepart] (ap2.east) -- node [left] {} (p3);
  \path [linepart] (ap3.east) -- node [left] {} (p4);
\end{tikzpicture}

\chapter*{Introduction}
\markboth{\MakeUppercase{Introduction}}{\MakeUppercase{Introduction}}
\addcontentsline{toc}{chapter}{Introduction}       

\noindent
The purpose of this study is to analyse two related topics: the Rumin cohomology and the orientability of a surface in the most classic example of Sub-Riemannian geometry, the Heisenberg group $\mathbb{H}^n$.\\
Our work begins with a quick definition of Lie groups, Carnot groups and left translation operators, moving then to define the Heisenberg group and its properties. There are many 
references for an introduction on the Heisenberg group; here we used, for example, parts of \cite{GCmaster}, \cite{CDPT}, \cite{FSSC} and \cite{TRIP}. The Heisenberg Group $\mathbb{H}^n$, $n \geq 1$, is a $(2n+1)$-dimensional manifold denoted $ ( \mathbb{R}^{2n+1}, * , d_{cc})$ where the group operation $*$ is given by 
$$
(x,y,t)*(x',y',t') := \left  (x+x',y+y', t+t'- \frac{1}{2} \langle J
 \begin{pmatrix} 
x \\
y
\end{pmatrix} 
, 
 \begin{pmatrix} 
x' \\
y' 
\end{pmatrix} \rangle_{\mathbb{R}^{2n}} \right  )
$$
with $x,y,x',y' \in \mathbb{R}^n$, $t,t' \in \mathbb{R}$ and $J=
  \begin{pmatrix} 
 0 &  I_n \\
-I_n & 0
\end{pmatrix} $. 
Additionaly, the Heisenberg Group 
is a Carnot group of step $2$ with algebra $\mathfrak{h} = \mathfrak{h}_1 \oplus \mathfrak{h}_2$. The first layer $\mathfrak{h}_1$ has a standard orthonormal basis of left invariant vector fields which are called \textit{horizontal}:
$$
\begin{cases}
X_j=\partial_{x_j} -\frac{1}{2} y_j \partial_t, \\
Y_j=\partial_{y_j} +\frac{1}{2} x_j \partial_t, \ \ \ \ j=1,\dots,n.
\end{cases}
$$
They hold the core property that $[X_j, Y_j] = \partial_t=:T $ for each $j$, where $T$ alone spans the second layer $\mathfrak{h}_2$ and is called the \textit{vertical} direction. By definition, the horizontal subbundle changes inclination at every point (see Figure \ref{fig:balls}), 
\begin{figure}[!ht]
\centering
{\includegraphics[width= 5 cm]{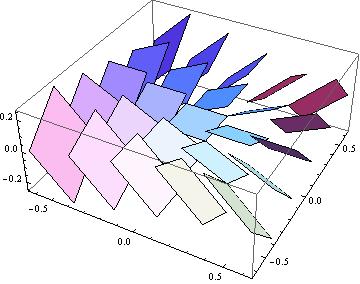}}
 \caption[Caption for LOF]{Horizontal subbundle in the first Heisenberg Group $\mathbb{H}^1$.\footnotemark}
\label{fig:balls}
\end{figure}
allowing movement from any point to any other point following only horizontal paths. The Carnot--Carathéodory distance $d_{cc}$, then, measures the distance between any two points along curves whose tangent vectors are horizontal.\\
The topological dimension of the Heisenberg group is $2n+1$, while its Hausdorff dimension with respect to the Carnot-Carathéodory distance is $2n+2$. Such dimensional difference leads to the existence of a natural cohomology 
called Rumin cohomology and introduced by M. Rumin in 1994 (see \cite{RUMIN}), whose behaviour is significantly different from the standard de Rham one. 
This is not the only effect of the dimensional difference: another is that there exist surfaces regular in the Heisenberg sense but fractal in the Euclidean sense (see \cite{KSC}).\\
\noindent
With a dual argument, one can associate at vector fields $X_j$'s,$Y_j$'s and $T$ the corresponding differential forms: $dx_j$'s and $dy_j$'s for $X_j$'s and $Y_j$'s respectively, and 
$$
\theta:=dt- \frac{1}{2} \sum_{j=1}^n (x_j dy_j - y_j dx_j )
$$
 for $T$. They also divide as \emph{horizontal} and \emph{vertical} in the same way as before.
These differential forms compose the complexes that, in the Heisenberg group, are described by the Rumin cohomology (see \cite{RUMIN} and 5.8 in \cite{FSSC}). 
Rumin forms 
are compactly supported on an open set $U$ and their sets are denoted by $\mathcal{D}_{\mathbb{H}}^{k} (U)$, where
\footnotetext{pictures shown with permission of the author Anton Lukyanenko.}
$$
\begin{cases}
\mathcal{D}_{\mathbb{H}}^{k} (U) :=      \frac{\Omega^k}{I^k}, \ \  \text{for } k=1,\dots,n            \\
\mathcal{D}_{\mathbb{H}}^{k} (U) :=      J^k,   \ \  \text{for } k=n+1,\dots,2n+1           ,
\end{cases}
$$
with $\Omega^k$ the space of $k$-differential forms, $I^k = \{ \alpha \wedge \theta + \beta \wedge d \theta \ / \  \alpha \in \Omega^{k-1}, \ \beta \in \Omega^{k-2}  \}$ and $J^k =\{ \alpha \in \Omega^{k} \ / \  \alpha \wedge \theta =0, \  \alpha \wedge d\theta=0   \}$.\\
The Rumin cohomology is the cohomology of this complex:
$$
0  \to \mathbb{R} \stackrel{ d_Q }{\to} \mathcal{D}_{\mathbb{H}}^{1} (U) \stackrel{ d_Q }{\to} \dots \stackrel{ d_Q }{\to} \mathcal{D}_{\mathbb{H}}^{n} (U)
 \stackrel{ D }{\to} \mathcal{D}_{\mathbb{H}}^{n+1} (U) \stackrel{ d_Q }{\to} \mathcal{D}_{\mathbb{H}}^{n+2} (U) \stackrel{ d_Q }{\to} \dots  \stackrel{ d_Q }{\to} \mathcal{D}_{\mathbb{H}}^{2n+1} (U) \to 0
$$
where $d$ is the standard differential operator and, for $k < n$, 
$
d_Q( [\alpha]_{I^k} ) :=  [d \alpha]_{I^{k+1}},
$ 
while, for $k \geq n +1$, 
$
d_Q := d_{\vert_{J^k}}.
$ 
The second order differential operator $D$ is defined as
$$
D( [\alpha]_{I^n} ) : = d \left ( \alpha +  L^{-1} \left (- (d \alpha)_{\vert_{ {\prescript{}{}\bigwedge}^{n+1} \mathfrak{h}_1 }} \right ) \wedge \theta \right )
$$
whose presence reflects the difference between the topological and Hausdorff dimensions of the space. In the definition above $L:  {\prescript{}{}\bigwedge}^{n-1} \mathfrak{h}_1  \to {\prescript{}{}\bigwedge}^{n+1} \mathfrak{h}_1 $, $L(\omega):=d\theta \wedge \omega$,  is a diffeomorphism among differential forms.\\
In Chapter \ref{DFaRC} we will carefully describe the cohomology and we show its complete behaviour in the cases $n=1$ and $n=2$. In particular we show how to compute the second order operator $D$ explicitly. In the appendices to this chapter we follow the presentation in \cite{TRIP} and explain how it is possible to write the Rumin differential operators as one operator $d_c$, reducing then the complex to $\left ( \mathcal{D}_{\mathbb{H}}^{k} (U), d_c \right )$ (Appendix \ref{A}). We also discuss the dimensions of the spaces in the Rumin complex (Appendix \ref{B}).\\
The differential operators $d_Q$ 
and $D$ look much more complicated than the standard operator $d$ and one could wonder whether they also hold the property of commuting with the pullback by a mapping. We show in Chapter \ref{PPHN} that this is true for contact maps, a map whose pushforward sends horizontal vectors to horizontal vectors. Namely one has that for a contact map $f: \mathbb{H}^n \to \mathbb{H}^n$ the following relations hold:
$$
\begin{cases}
f^* d_Q = d_Q f^* \ \ \ \ \text{ for } k \neq n,\\
f^* D = D f^* \ \ \ \ \text{ for } k = n.
\end{cases}
$$
We also show the behaviour of pushfoward and pullback in several situations in this setting, for which a useful starting point is \cite{KORR}.\\
Differential forms can be used to define the notion of orientability, so it is natural to ask whether the Rumin forms provide a different kind of orientability respect to the standard definition. In Chapter \ref{orient4} we show that this is indeed the case. 
First, we have to notice how in the Heisenberg group it is natural to give an ad hoc definition of regularity for surfaces, the $\mathbb{H}$-regularity (see \cite{FSSC2001} and \cite{FSSC}) which, roughly speaking, locally requires the surface to be a level set of a function with non-vanishing horizontal gradient. The points such gradient is null are called characteristic (see, for instance, \cite{BAL} and \cite{MAG}) and must usually be avoided. 
For such surfaces we give a new definition of orientability ($\mathbb{H}$-orientability) along with some properties. In particular we show that it behaves well with respect to the left-translations and the anisotropic dilation $\delta_r(x,y,t)=(rx,ry,r^2t)$. Furthermore, we prove that $\mathbb{H}$-orientability implies standard orientability, while the opposite is not always true. Finally we show that, up to one point, a Möbius Strip in $\mathbb{H}^1$ is a $\mathbb{H}$-regular surface and we use this fact to prove that there exist $\mathbb{H}$-regular non-$\mathbb{H}$-orientable surfaces, at least in the case $n=1$. \\
Apart from its connection with differential forms, another reason to study orientability is that it plays an important role in the theory of currents. Surfaces connected to currents are usually, but not always, orientable: in Riemanian settings there exists a notion of currents (currents mod $2$) made for surfaces that are not necessarily orientable (see \cite{MORGAN2}). The existence of $\mathbb{H}$-regular non-$\mathbb{H}$-orientable surfaces implies that the same kind of analysis would be meaningful in the Heisenberg group.





%
%

%
\chapter{Preliminaries}

In this chapter we will first provide definitions and notions about Lie and Carnot groups (Section \ref{liegroups}), as well as introduce the Heisenberg group and some basic properties and automorphisms associated to it. Second, we will present the standard basis of
vector fields in the Heisenberg group $\mathbb{H}^n$, including the behaviour of the Lie brackets and the left invariance, which will lead to the definition of dual differential forms. Next we will mention briefly different equivalent distances on $\mathbb{H}^n$: the \emph{Carnot-Carathéodory} distance $d_{cc}$ and  the \emph{Korányi} distance $d_{\mathbb{H}}$. Finally we will present the Heisenberg group's  topology and dimensions (Section \ref{defH}).
The Heisenberg group is maybe the most famous example of Sub-Riemannian geometry. we will commonly use the adjectives ``Riemannian" and ``Euclidean" as synonymous. As opposite to them, we will refer to ``Sub-Riemannian" and ``Heisenberg" for objects proper of the Sub-Riemannian structure of the Heisenberg group.\\
 There exists many good references for an introduction on the Heisenberg group; we will follow section 2.1 in \cite{GCmaster}, the introduction of \cite{TRIP}, sections 2.1 and 2.2 in \cite{FSSC} and section 2.1.3 and 2.2 in \cite{CDPT}.


\section{Lie Groups and Left Translation}\label{liegroups}

In this section we provide definitions and notions about Lie and Carnot groups. General references can be, for example, section 2.1 in \cite{GCmaster}, the introduction of \cite{TRIP}, section 2.1 in \cite{FSSC} and 2.1.3 in \cite{CDPT}.

\begin{defin}\label{lie}             
A group $\mathbb{G}$ with the  group operation $*$, $(\mathbb{G},*)$, is a \emph{Lie Group} if
\begin{itemize}
\item
$\mathbb{G}$ is a differentiable manifold,
\item
the map $\mathbb{G} \times \mathbb{G} \to \mathbb{G}, \ (p,q) \mapsto p*q=p q$ is differentiable,
\item
the map  $ \mathbb{G} \to \mathbb{G}, \ p \mapsto p^{-1}$ is differentiable.
\end{itemize}
\end{defin}

\begin{defin}\label{lefttr}             
Let $(\mathbb{G}, *)$ be a Lie group, with the following smooth operation
\begin{align*}
\mathbb{G} \times \mathbb{G} \to \mathbb{G}, \  (p, q) \mapsto p*q^{-1}.
\end{align*}
If $q \in \mathbb{G}$, then denote the \emph{left translation by $q$} as
\begin{align*}
L_q : \mathbb{G}  \to \mathbb{G}, \ p \mapsto q*p.
\end{align*}
In the literature, $L_q$ is often denoted also as $\tau_q$. For this reason we will write $\tau_q$ in Chapter \ref{orient4}, when talking specifically about the Heisenberg group.
\end{defin}

\begin{obs}             
It follows from the definition that
$$L_{q} L_{p}=L_{q p}.
$$
\end{obs}

\begin{defin}\label{leftinv}             
A vector field $V$ on a Lie group $\mathbb{G}$ is \emph{left-invariant} if $V$ commutes with $L_g$, for every $g \in \mathbb{G}$. Specifically, $V$  is \emph{ left-invariant } if 
$$
(L_q)_* V_p = V_{L_q (p)} =V_{q *p} \ \in  T_{q *p}  \mathbb{G}, 
$$
for every $p,q \in \mathbb{G}$, 
where $(\cdot)_*$ expresses the standard pushforward. Equivalently, one can express the definitions as
$$
V_p (\varphi \circ L_q ) = [( V \varphi ) \circ L_q ]_p \ \in  C^\infty( U_{q *p} ),
$$
for every $ p,q \in \mathbb{G}$ and $ \varphi \in C^\infty (\mathbb{G})$, where $U_{q *p} $ is a neighbourhood of $q*p$.
\end{defin}

\begin{no}\label{neighbourhood}             
Often we will need to refer to neighbourhoods of points. For this reason, we introduce the notation
$$
\mathcal{U}_p := \left  \{ U_p \ / \ U_p \text{ a neighbourhood of } p  \right \}.
$$
\end{no}

\begin{obs}             
The most important property of left invariant vector fields is that they are uniquely determined by their value at one point, which is usually taken to be the neutral element.\\
In general, to compute the value of a left invariant vector field $V_q \in T_q \mathbb{G}$ at another point $p$, we can simply left-translate by $p*q^{-1}$:
$$
( L_{p*q^{-1}} )_* V_q = V_{p*q^{-1}*q}=V_p.
$$
\end{obs}

\noindent
There are special Lie Groups that hold additional important properties; they are called Carnot Groups. Before defining them, we need to introduce the Lie bracket operation:

\begin{defin}         
The \emph{Lie bracket} or \emph{commutator} of vector fields is an operator defined as follows
\begin{align*}
[,] : \mathfrak{g} \times \mathfrak{g}  \to \mathfrak{g},   \   ( V_1, V_2)  \mapsto [V_1,V_2]:=V_1 V_2 - V_2 V_1,
\end{align*}
where $\mathfrak{g}$ is an algebra over a field $\mathbb{G}$.
\end{defin}


\begin{defin}         
A \emph{Lie algebra} $\mathfrak{g}$ of a Lie group $\mathbb{G}$ is a non-associative algebra with the following properties: for all $ V_1,V_2,V_3 \in \mathfrak{g}$,
\begin{itemize}
\item
$[V_1,V_1]=0 $ \quad (Alternativity),
\item
$[V_1+V_2,V_3]=[V_1,V_3]+[V_2,V_3] \ \text{ and } \ [V_3,V_1+V_2]=[V_3,V_1]+[V_3,V_2] $ \quad \\  (Bilinearity),
\item
$[V_1,[V_2,V_3]]+[V_2,[V_3,V_1]]+[V_3,[V_1,V_2]]=0 $ \quad (Jacobi identity).
\end{itemize}
\end{defin}

\begin{defin}             
A \emph{Carnot group of step $k$} is a connected, simply connected Lie group whose Lie algebra $\mathfrak{g}$ admits a step $k$ stratification, i.e.,
$$
\mathfrak{g} = V_1 \oplus \cdots \oplus V_k,
$$
where every $V_j$ is a linear subspace of $\mathfrak{g}$ satisfying $[V_1,V_j]=V_{j+1}$. Additionally, $V_k \neq \{ 0 \}$ and $V_j= \{ 0 \}$ for $j>k$.
\end{defin}

\begin{defin}\label{homdim}              
Let $\mathbb{G}$ be a Carnot group and call $m_i := dim (V_i) $ for each $V_i$ in the stratification of $\mathfrak{g}$. Then the \emph{homogeneous dimension} of $\mathbb{G}$ is
$$
Q:= \sum_{i=1}^k i m_i.
$$
\end{defin}

\section{Definition of $\mathbb{H}^n$}\label{defH}

In this section we introduce the Heisenberg group $\mathbb{H}^n$ as well as some basic properties and automorphisms associated to it. Then we present the standard basis of
vector fields in the Heisenberg group $\mathbb{H}^n$, including the behaviour of the Lie brackets and the left invariance, which will lead to the definition of dual differential forms. Next we mention briefly different equivalent distances on $\mathbb{H}^n$: the \emph{Carnot-Carathéodory} distance $d_{cc}$ and  the \emph{Korányi} distance $d_{\mathbb{H}}$. Finally we mention the Heisenberg group's topology and dimensions.\\
General references are section 2.1 in \cite{GCmaster}, sections 2.1 and 2.2 in \cite{FSSC} and section 2.2 in \cite{CDPT}.

\begin{defin}\label{Heisenberg_Group}             
The $n$-dimensional \emph{Heisenberg Group} $\mathbb{H}^n$ is defined as
$$
\mathbb{H}^n:= (\mathbb{R}^{2n+1}, * ),
$$
where $*$ is the following product:
$$
(x,y,t)*(x',y',t') := \left  (x+x',y+y', t+t'- \frac{1}{2} \langle J
 \begin{pmatrix} 
x \\
y
\end{pmatrix} 
, 
 \begin{pmatrix} 
x' \\
y' 
\end{pmatrix} \rangle_{\mathbb{R}^{2n}} \right  ),
$$
with $x,y,x',y' \in \mathbb{R}^n$, $t,t' \in \mathbb{R}$ and $J=  \begin{pmatrix} 
 0 &  I_n \\
-I_n & 0
\end{pmatrix} $.\\\\
Notationally it is common to write $x=(x_1,\dots,x_n) \in \mathbb{R}^n$. Furthermore, with a simple computation of the matrix product, we immediately have that
$$
 \langle 
\begin{pmatrix} 
 0 &  I_n \\
-I_n & 0
\end{pmatrix}
 \begin{pmatrix} 
x \\
y
\end{pmatrix} 
, 
 \begin{pmatrix} 
x' \\
y' 
\end{pmatrix} \rangle
=
 \langle
 \begin{pmatrix} 
y \\
-x
\end{pmatrix} 
, 
 \begin{pmatrix} 
x' \\
y' 
\end{pmatrix} \rangle
=
\sum_{j=1}^n \left ( y_j x_j' - x_j y_j'  \right ),
$$
and so one can rewrite the product as
$$
(x,y,t)*(x',y',t') := \left  (x+x',y+y', t+t' + \frac{1}{2} \sum_{j=1}^n \left ( x_j y_j'  -  y_j x_j'  \right ) \right  ).
$$
\end{defin}

\begin{obs}             
The Heisenberg group $\mathbb{H}^n$ satisfies the conditions of Definition \ref{lie} and is hence a Lie group.
\end{obs}

\begin{obs}
One can easily see the following properties
\begin{itemize}
\item 
$\mathbb{H}^n$ is a non-commutative group.
\item 
The neutral element of $\mathbb{H}^n$ is $(0,0,0)$.
\item 
The inverse of $(x,y,t) \in \mathbb{H}^n$ is $(x,y,t)^{-1}=(-x,-y,-t)$.
\item 
The center of the group, namely the elements that commute with all the elements of the group, is $\{ (0,t) \in \mathbb{R}^{2n} \times \mathbb{R}  \}$.
\end{itemize}
\end{obs}

\noindent
On the Heisenberg group $\mathbb{H}^n$ there are two important groups of automorphisms. The first one is the operation of left-translation (see Definition \ref{lefttr}) and the second one is the ($1$-parameter) group of the \emph{anisotropic dilatations $\delta_r$}:

\begin{defin}
\label{delta}
The ($1$-parameter) group of the \emph{anisotropic dilatations $\delta_r$}, with $r \in \mathbb{R}^+$, is defined as follows
\begin{align*}
\delta_r : \mathbb{H}^n   \to \mathbb{H}^n, \ (x,y,t)  \mapsto (rx,ry,r^2 t).
\end{align*}
\end{defin}

\subsection{Left Invariance and Horizontal Structure on $\mathbb{H}^n$}\label{lefthor}

In this subsection we present the standard basis of 
vector fields in the Heisenberg group $\mathbb{H}^n$, including the behaviour of the Lie brackets and the left invariance. This will lead to conclude that the Heisenberg group is a Carnot group and to the definition of dual differential forms.\\
General references are section 2.1 in \cite{GCmaster} and sections 2.1 and 2.2 in \cite{FSSC}.


\begin{defin}
\label{XYT}
A basis of left invariant vector fields in $\mathbb{H}^n$
consists of the following $2n+1$ vectors:
 $$
\begin{cases}
X_j = \partial_{x_j} - \frac{1}{2} y_j\partial_{t} \ \ \emph{\emph{ for }} \ j=1,\dots,n , \\
Y_j = \partial_{y_j} + \frac{1}{2} x_j\partial_{t} \ \ \emph{\emph{ for }} \ j=1,\dots,n,  \\
T = \partial_{t}.
\end{cases}
 $$
We will show in Lemma \ref{leftinvar} that they are indeed left invariant.
\end{defin}

\begin{obs}
One can observe the immediate property that $\{ X_1,\dots,X_n,Y_1,\dots,Y_n,T \}$ becomes $\{ \partial_{x_1},\dots, \partial_{x_n}, \partial_{y_1},\dots,\partial_{y_n}, \partial_{t} \}$ at the neutral element.
\end{obs}

\begin{lem}\label{simplecomposition}      
Consider a function $g: U \subseteq  \mathbb{H}^n \to \mathbb{H}^n$, $U$ open. Assume then $g \in \left [ C^1(U, \mathbb{R}) \right ]^{2n+1}$, meaning that all the $2n+1$ components of $g$ are $C^1$-regular.\\
Consider a second function $f: \mathbb{H}^n \to \mathbb{R}$, $f\in C^1(\mathbb{H}^n, \mathbb{R})$.
Then the following holds:
\begin{align}
X_j (f \circ g) &= (\nabla f)_g \cdot X_j g, \label{firstone} \\
Y_j (f \circ g) &= (\nabla f)_g \cdot Y_j g  , \\
T (f \circ g) &= (\nabla f)_g \cdot T g .
\end{align}
where $ j=1,\dots,n$ and  $\nabla$ describes the Euclidean gradient in $\mathbb{R}^{2n+1}=\mathbb{H}^n$.
\end{lem}

\begin{proof}
Equation \eqref{firstone} holds by direct computation. Indeed, we have
\begin{align*}
X_j(f \circ g) 
=&\left (  \partial_{x_j} -\frac{1}{2} y_j\partial_t  \right ) (f \circ g)\\
=& \sum_{i=1}^{n} \left ( \partial_{x_i} f \partial_{x_j} g^{i}+ \partial_{y_i} f \partial_{x_j} g^{n+i} \right ) + \partial_{t} f \partial_{x_j} g^{2n+1}\\
&
  -\frac{1}{2}y_j  \left [ \sum_{i=1}^{n} \left ( \partial_{x_i} f \partial_t g^{i}+ \partial_{y_i} f \partial_t g^{n+i} \right ) + \partial_{t} f \partial_t g^{2n+1} \right ]\\
=& \sum_{i=1}^{n} \left ( \partial_{x_i} f \partial_{x_j} g^{i}   -\frac{1}{2} y_{j}  \partial_{x_i} f \partial_t g^{i} \right )  +  \sum_{i=1}^{n} \left ( \partial_{y_i} f \partial_{x_j} g^{n+i}  -\frac{1}{2} y_{j}  \partial_{y_i} f \partial_t g^{n+i} \right ) \\
&
+ \partial_{t} f \partial_{x_j} g^{2n+1}   -\frac{1}{2} y_{j} \partial_{t} f \partial_t g^{2n+1} \\
=& \sum_{i=1}^{n} \partial_{x_i} f \left (  \partial_{x_j} g^{i}   -\frac{1}{2} y_{j}   \partial_t g^{i} \right ) + \partial_{y_i} f  \sum_{i=1}^{n} \left ( \partial_{x_j} g^{n+i}  -\frac{1}{2}  y_{j} \partial_t g^{n+i} \right ) \\
&
+ \partial_{t} f \left ( \partial_{x_j} g^{2n+1}   -\frac{1}{2}  y_{j}   \partial_t g^{2n+1} \right ) \\
=& \sum_{i=1}^{n} \partial_{x_i} f X_j g^{i} +  \sum_{i=1}^{n} \partial_{y_i} f X_j  g^{n+i} + \partial_{t} f X_j  g^{2n+1} \\
=&(\partial_{x_1} f, \dots, \partial_{x_n} f ,\partial_{y_1} f, \dots, \partial_{y_n} f \partial_t f)_g \cdot  (X_j g^1, \dots , X_j g^{2n+1} )\\
=& (\nabla f)_g \cdot X_j g.
\end{align*}
\noindent
This proves the first part of the statement and the proof for $Y_j(f \circ g)$ is analogous. The case of $T$ is much simpler:
\begin{align*}
T(f \circ g)
=& \left (  \partial_t  \right ) (f \circ g)\\
=&
\sum_{i=1}^{n}
\left (
\partial_{x_i} f \partial_t g^{i}+
\partial_{y_i} f \partial_t g^{n+i}
\right ) +
\partial_{t} f \partial_t g^{2n+1}\\
=&
\sum_{i=1}^{n}
\left (
\partial_{x_i} f T g^{i}+
\partial_{y_i} f T g^{n+i}
\right ) +
\partial_{t} f T g^{2n+1}\\
=&(\partial_{x_1} f, \dots, \partial_{x_n} f ,\partial_{y_1} f, \dots, \partial_{y_n} f \partial_t f)_g 
\cdot
 (T g^1, \dots , T g^{2n+1} )\\
=&(\nabla f)_g \cdot T g.
\end{align*}
\end{proof}

\begin{lem}\label{leftinvar}      
In Definition \ref{XYT} we claimed that $X_j$'s, $Y_j$'s and $T$ are left invariant vector fields. we prove it here.
\end{lem}

\begin{proof}
For notational simplicity, we consider $n=1$ and $X_1=X =  \partial_{x} - \frac{1}{2} y\partial_{t}$. The other cases and the general situation follow with hardly any change.\\
Consider $f \in C^1 (\mathbb{H}^1, \mathbb{R} )$ and $p=(x_p, y_p,t_p), q=(x_q, y_q,t_q) \in \mathbb{H}^1$. By Lemma  \ref{simplecomposition},
\begin{align*}
X_p (f \circ L_q)
=&   (\nabla f)_{L_q(p)} \cdot X_p (L_q) \\
=& (\partial_x f, \partial_y f, \partial_t f)_{L_q(p)} \cdot \left (   X (L_q^{(1)})(p), X (L_q^{(2)})(p), X (L_q^{(3)})(p)   \right )\\
=& (\partial_x f, \partial_y f, \partial_t f)_{L_q(p)} \cdot \left (  1 , 0 , - \frac {1}{2} y_q  - \frac{1}{2} y_p   \right )\\
=&(\partial_x f)_{L_q(p)} -\frac{1}{2} (y_q+y_p) (\partial_t f)_{L_q(p)}   \\
=& \left [
 \partial_x f -\frac{1}{2}y \partial_t f
\right ]_{L_q(p)} =
 [ (Xf) \circ L_q ]_p,
\end{align*}
where, for the third line, we used that
$$
\begin{cases}
 (L_q^{(1)})(p) =x_q + x_p,\\
(L_q^{(2)})(p) = y_q + y_p \quad \text{and}\\
(L_q^{(3)})(p) =t_q + t_p +\frac {1}{2} (x_q y_p - y_q x_p).
\end{cases}
$$
This proves the left invariance of $X$. Repeating the same argument for all $X_j$, $Y_j$ and $T$ completes the proof.
\end{proof}

\begin{obs}      
\label{[,]}
The only non-trivial commutators of the vector fields $X_j,Y_j$ and $T$ are
$$
[X_j,Y_j]=T  \ \emph{\emph{ for }} \ j=1,\dots,n.
$$
This immediately tells that all the higher-order commutators are zero.
\end{obs}

\begin{rem} \label{Hcarnot}     
The observation above shows that the Heisenberg group is a Carnot group of step 2. Indeed we can write its algebra $\mathfrak{h}$ as:
$$
\mathfrak{h} =\mathfrak{h}_1 \oplus \mathfrak{h}_2,
$$
with
$$
\mathfrak{h}_1 = \spn \{ X_1,  \ldots, X_n, Y_1, \ldots, Y_n \} \text{ and } \mathfrak{h}_2 =\spn \{ T \}.
$$
Usually one calls $\mathfrak{h}_1$ the space of \emph{horizontal vectors} and $\mathfrak{h}_2$ the space of \emph{vertical vectors}.
\end{rem}

\begin{obs}      
Consider a function $f \in C^1 (U, \mathbb{R} )$, $U\subseteq \mathbb{H}^n$ open. It is useful to mention that the vector fields $\{ X_1,\dots,X_n,Y_1,\dots,Y_n\}$ are homogeneous of order $1$ with respect to the dilatation $\delta_r, \  r \in \mathbb{R}^+$, i.e.,
$$
X_j (f\circ \delta_r)=r X_j(f)\circ \delta_r \ \ \ \  \text{ and }   \ \ \ \   Y_j (f\circ \delta_r)=r Y_j(f)\circ \delta_r ,
$$
for any $j=1,\dots,n$.\\
On the other hand, the vector field $T$ is homogeneous of order $2$, that is,
$$
T(f\circ \delta_r)=r^2T(f)\circ \delta_r.
$$
The proof is a simple application of Lemma \ref{simplecomposition}.\\
It is not a surprise, then, that the homogeneous dimension of $\mathbb{H}^n$ (see Definition \ref{homdim}) is
$$
Q=2n+2.
$$
\end{obs}

\begin{obs}\label{scal}
The vectors $X_1,\dots,X_n,Y_1,\dots,Y_n,T$ can be made an orthonormal basis of $\mathfrak{h}$ with a scalar product $\langle \cdot , \cdot \rangle $.\\
In the same way, the vectors $X_1,\dots,X_n,Y_1,\dots,Y_n$ form an orthonormal basis of $\mathfrak{h}_1$ with a scalar product $\langle \cdot , \cdot \rangle_H $ defined purely on $\mathfrak{h}_1$.
\end{obs}


\begin{no}\label{notW}
Sometimes it will be useful to consider all the elements of the basis of $\mathfrak{h}$ with one symbol. To do so, one can notationally write
$$
\begin{cases}
W_i &:= X_i \ \text{ for } j=1,\dots,n,\\
W_{n+i} &:= Y_i\ \text{ for } j=1,\dots,n,\\
W_{2n+1}&:=T.
\end{cases}
$$
In the same way, the point $(x_1,\dots,x_n,y_1,\dots,y_n,t)$ will be denoted as $(w_1,\dots,w_{2n+1})$.
\end{no}

\begin{defin}
\label{dual_basis}
Consider $ {\prescript{}{}\bigwedge}^1 \mathfrak{h}$ as the dual space of $\mathfrak{h}$, 
which inherits an inner product from the one of $\mathfrak{h}$. By duality, one can find a dual orthonormal basis of covectors $\{\omega_1,\dots,\omega_{2n+1}\}$ in $ {\prescript{}{}\bigwedge}^1 \mathfrak{h}$ such that 
$$
 \omega_j ( W_k ) =\langle \omega_j , W_k \rangle = \langle \omega_j \vert W_k \rangle =\delta_{jk}, \quad  j,k=1,\dots,2n+1,
$$
where $W_k$ is an element of the basis of $\mathfrak{h}$ and the notation varies in the literature. Such covectors are differential forms in the Heisenberg group. It turns out that the dual orthonormal basis is given by
$$
\{dx_1,\dots,dx_n,dy_1,\dots,dy_n,\theta \},
$$
where $\theta$ is called \emph{contact form} and is defined as
$$
\theta :=dt - \frac{1}{2}  \sum_{j=1}^{n} (x_j d y_j-y_j d x_j).
$$
\end{defin}

\begin{no}\label{dfnota}
As it will be useful sometimes to call all such forms by the same name, one can notationally write,
$$
\begin{cases}
\theta_i &:= dx_i\ \text{ for } j=1,\dots,n,\\
\theta_{n+i} &:= dy_i\ \text{ for } j=1,\dots,n,\\
\theta_{2n+1}&:=\theta.
\end{cases}
$$
In particular the covector $\theta_i$ is always the dual of the vector $W_i$, for all $i=1,\dots,2n+1$.
\end{no}

\begin{obs}\label{contactcontactcontact} 
Note that, one could have introduced the Heisenberg group $\mathbb{H}^n$ with a different approach and defined it as a \emph{contact manifold}. A contact manifold is a manifold with a contact structure, meaning that its
algebra $\mathfrak{h}$ has a $1$-codimensional subspace $Q$ that can be written as a kernel of a non-degenerate $1$-form, which is then called \emph{contact form}.\\
The just-defined $\theta$ satisfy all these requirements and is indeed the contact form of the Heisenberg group, while $Q=\mathfrak{h}_1$. The non-degeneracy condition is $\theta \wedge d \theta \neq 0$. A straightforward computation shows that
$$
d \theta =- \sum_{j=1}^{n} dx_j \wedge dy_j,
$$
and so indeed
$$
\theta \wedge d \theta =- \sum_{j=1}^{n} dx_j \wedge dy_j \wedge \theta  \neq 0.
$$
\end{obs}

\begin{obs}\label{df} 
As a useful example, we show here that the just-defined bases of vectors and covectors behave as one would expect when differentiating. Specifically, consider $f: U\subseteq \mathbb{H}^n \to \mathbb{R}$, $U$ open, $f \in C^1 (U, \mathbb{R} )$, then one has:
\begin{align*}
 df
=& \sum_{j=1}^{n}  \left ( \partial_{x_j} f d x_j + \partial_{y_j} f dy \right ) +  \partial_t f dt  \\
=&  \sum_{j=1}^{n}  \left ( \partial_{x_j} f d x_j + \partial_{y_j} f dy \right )  +  (\partial_t f) \left ( \theta + \frac{1}{2}\sum_{j=1}^{n} x_j dy_j -  \frac{1}{2}\sum_{j=1}^{n} y_j dx_j    \right )\\
=&  \sum_{j=1}^{n}  \left ( X_j f d x_j + Y_j f dy_j \right ) + Tf \theta.
\end{align*}
\end{obs}

\noindent
The next natural step is to define vectors and covectors of higher order.

\begin{defin} \label{kdim}   
We define the sets of $k$-vectors and $k$-covectors, respectively, as follows:
\begin{align*}
\Omega_k \equiv {\prescript{}{}\bigwedge}_k \mathfrak{h} &:= \spn \{ W_{i_1} \wedge \dots \wedge W_{i_k} \}_{1\leq i_1 \leq \dots \leq i_k \leq 2n+1 },
\end{align*}
and
\begin{align*}
\Omega^k \equiv {\prescript{}{}\bigwedge}^k \mathfrak{h} &:= \spn \{ \theta_{i_1} \wedge \dots \wedge \theta_{i_k} \}_{1\leq i_1 \leq \dots \leq i_k \leq 2n+1 }.
\end{align*}
The same definitions can be given for $ \mathfrak{h}_1$ and produces the spaces $ {\prescript{}{}\bigwedge}_k \mathfrak{h}_1 $ and $ {\prescript{}{}\bigwedge}^k \mathfrak{h}_1 $.
\end{defin}

\begin{defin}
For $k=1,\dots,2n+1$, if $\omega \in  {\prescript{}{}\bigwedge}^k \mathfrak{h}$, then we define $\omega^* \in  {\prescript{}{}\bigwedge}_k \mathfrak{h}$ so that
$$
\langle \omega^* , V   \rangle   =     \langle \omega \vert V   \rangle, \ \ \ V \in  {\prescript{}{}\bigwedge}_k \mathfrak{h}.
$$
\end{defin}

\noindent
We give here the definition of Pansu differentiability for maps between Carnot groups $\mathbb{G}_1$ and $\mathbb{G}_2$. After that, we state it in the special case of $\mathbb{G}_1=\mathbb{H}^n$ and $\mathbb{G}_2=\mathbb{R}$.\\
We call a function $h : (\mathbb{G}_1,*,\delta^1) \to (\mathbb{G}_2,*,\delta^2)$ \emph{homogeneous} if $h(\delta^1_r(p))= \delta^2_r \left ( h(p) \right )$ for all $r>0$.

\begin{defin}[see \cite{PANSU} and 2.10 in \cite{FSSC}]\label{dGGG}
Consider two Carnot groups $(\mathbb{G}_1,*,\delta^1)$ and $(\mathbb{G}_2,*,\delta^2)$. A function $f: U \to \mathbb{G}_2$, $U \subseteq \mathbb{G}_1$ open, is \emph{P-differentiable} at $p_0 \in U$ if there is a (unique) homogeneous Lie groups homomorphism $d_H f_{p_0} : \mathbb{G}_1 \to \mathbb{G}_2$ such that
$$
d_H f_{p_0} (p) := \lim\limits_{r \to 0} \delta^2_{\frac{1}{r}} \left ( f(p_0)^{-1} * f(p_0* \delta_r^1 (p) ) \right ),
$$
uniformly for $p$ in compact subsets of $U$.
\end{defin}

\begin{defin}\label{dHHH}
Consider $f: U \to \mathbb{R}$, $U \subseteq \mathbb{H}^n$ open. $f$ is \emph{P-differentiable} at $p_0 \in U$ if there is a (unique) homogeneous Lie groups homomorphism $d_H f_{p_0} : \mathbb{H}^n \to \mathbb{R}$ such that
$$
d_H f_{p_0} (p) := \lim\limits_{r \to 0} \frac{  f \left (p_0* \delta_r (p) \right ) - f(p_0) }{r},
$$
uniformly for $p$ in compact subsets of $U$.
\end{defin}

\begin{obs}
Consider $f:U\subseteq \mathbb{H}^n \to \mathbb{H}^n = \mathbb{R}^{2n+1}$, $U$ open, and interpret it in components as $f=(f^1,\dots,f^{2n+1})$. A straightforward computation shows immediately that, $f$ is P-differentiable in the sense of Definition \ref{dGGG}, then $f^1,\dots,f^{2n}$ are P-differentiable in the sense of Definition \ref{dHHH}.
\end{obs}

\begin{proof}
Consider n=1 for simplicity. The other cases follow immediately. Consider $p_0=(x_0,y_0,t_0)$ and $p=(x,y,t)$ in $\mathbb{H}^n$. A straightforward computation shows that
\begin{align*}
&\lim_{r\to 0+}  \delta_{\frac{1}{r}} \left ( f(p_0)^{-1} * f(p_0* \delta_r (p) ) \right )=\\
&= \lim_{r\to 0+} \Bigg (  
 \frac{f^1 (p_0 *\delta_r(p)  ) - f^1 (p_0)}{r}  ,  \frac{f^2 (p_0 *\delta_r(p)  ) - f^2 (p_0)}{r} ,\\
& \frac{f^3 (p_0 *\delta_r(p)  ) - f^3 (p_0)  +\frac{1}{2}  \left (   f^2 (p_0)f^1 (p_0 *\delta_r(p)  ) -    f^1 (p_0) f^2 (p_0 *\delta_r(p)  )        \right )   }{r^2}
\Bigg ).
\end{align*}
By hypothesis the limit exists and the first two components give us the claim.
\end{proof}

\begin{defin}[see 2.11 in \cite{FSSC}]
Let $f: U \subseteq \mathbb{H}^n \to \mathbb{R}$, $U$ open, be a P-differentiable function at $p \in U$. Then the \emph{Heisenberg gradient} or \emph{horizontal gradient} of $f$ at $p$ is defined as
$$
\nabla_\mathbb{H} f(p) := d_H f (p)^* \in \mathfrak{h}_1,
$$
or, equivalently,
$$
\nabla_\mathbb{H} f(p) = \sum_{j=1}^{n} \left [  (X_j f)(p) X_j  + (Y_j f)(p) Y_j  \right ].
$$
\end{defin}


\begin{no}[see 2.12 in \cite{FSSC}]\label{CH1}
Sets of differentiable functions can be defined with respect to the P-differentiability. Take $U \subseteq \mathbb{G}_1$ open, then
\begin{itemize}
\item
$C_{\mathbb{H}}^1 (U, \mathbb{G}_2)$ is the vector space of continuous functions $f:U \to \mathbb{G}_2 $  such that the P-differential $d_H f$ is continuous.
\item
$\left [ C_{\mathbb{H}}^1 (U, \mathbb{G}_2) \right ]^k$ is the set of $k$-tuples $f=\left (f^1,\dots, f^k \right)$ such that $f^i \in C_{\mathbb{H}}^1 (U, \mathbb{G}_2)$ for each $i=1 ,\dots , k$.
\end{itemize}
In particular, take $ U \subseteq \mathbb{H}^n$ open; then
\begin{itemize}
\item
$C_{\mathbb{H}}^1 (U, \mathbb{H}^n)$ is the vector space of continuous functions $f:U \to \mathbb{H}^n $  such that the P-differential $d_H f$ is continuous.
\item
$C_{\mathbb{H}}^1 (U, \mathbb{R})$ is the vector space of continuous functions $f:U \to \mathbb{R} $  such that $\nabla_\mathbb{H} f$ is continuous in $U$ (or, equivalently, such that the P-differential $d_H f$ is continuous).
\item
$C_{\mathbb{H}}^k (U, \mathbb{R})$ is the vector space of continuous functions $f:U \to \mathbb{R}$ such that the derivatives of the kind $W_{i_1} \dots W_{i_k}f$ are continuous in $U$, where $W_{i_h}$ is any $X_j$ or $Y_j$.
\item
$\left [ C_{\mathbb{H}}^m (U,\mathbb{R}) \right ]^k$ is the set of $k$-tuples $f=\left (f^1,\dots, f^k \right)$ such that $f^i \in C_{\mathbb{H}}^m (U,\mathbb{R})$ for each $i=1 ,\dots , k$.
\end{itemize}
\end{no}

\begin{obs}\label{diamond}
Given the notation above we have:
$$
\diamondinclusion{C^3 (U, \mathbb{R})}{C^2 (U, \mathbb{R})\\C_{\mathbb{H}}^3 (U, \mathbb{R})}{ C_{\mathbb{H}}^2 (U, \mathbb{R}) }
   \subsetneq C^1 (U, \mathbb{R}) \subsetneq  C_{\mathbb{H}}^1 (U, \mathbb{R}).
$$
\end{obs}

\noindent
We define here also an operator that will be useful later: the Hodge operator. The Hodge operator of a vector returns a second vector of dual dimension with the property to be orthogonal to the first. This will be used when talking about orientability as well as tangent and normal vector fields.

\begin{defin}[see 2.3 in \cite{FSSC} or 1.7.8 in \cite{FED}]\label{hodge}
Let $1 \leq k \leq 2n$. The \emph{Hodge operator} is the linear isomorphism
\begin{align*}
*: {\prescript{}{}\bigwedge}_k \mathfrak{h} &\rightarrow {\prescript{}{}\bigwedge}_{2n+1-k} \mathfrak{h} ,\\
\sum_I v_I V_I &\mapsto  \sum_I v_I (*V_I),
\end{align*}
where
$$
*V_I:=(-1)^{\sigma(I) }V_{I^*},
$$
and, for $1 \leq  i_1 \leq \cdots \leq i_k \leq 2n+1$,
\begin{itemize}
\item $I=\{ i_1,\cdots,i_k \}$,
\item $V_I= V_{i_1} \wedge \cdots \wedge V_{i_k} $,
\item $I^*=\{ i_1^*,\dots,i_{2n+1-k}^* \}=\{1, \cdots, 2n+1\} \smallsetminus I \quad $  and
\item $\sigma(I)$ is the number of couples $(i_h,i_l^*)$ with $i_h > i_l^*$.
\end{itemize}
\end{defin}


\subsection{Distances on $\mathbb{H}^n$}

In this subsection we mention briefly different equivalent distances on $\mathbb{H}^n$. General references are section 2.1 in \cite{GCmaster} and section 2.2.1 in \cite{CDPT}.\\
The usual intrinsic distance in the Heisenberg group is the \emph{Carnot} -- \emph{Carathéodory} distance $d_{cc}$, which measures the distance between any two points along shortest curves whose tangent vectors are horizontal.\\
Here we define more precisly another distance, equivalent to the first, called the \emph{Korányi} distance:

\begin{defin}
\label{norm}
We define the \emph{Korányi} distance on $\mathbb{H}^n$ by setting, for $p,q \in \mathbb{H}^n$,
$$
d_{\mathbb{H}} (p,q) :=  \norm{ q^{-1}*p }_{\mathbb{H}},
$$
where $ \norm{ \cdot }_{\mathbb{H}}$ is the \emph{Korányi}  norm
$$
\norm{(x,y,t)}_{\mathbb{H}}:=\left (  |(x,y)|^4+16t^2  \right )^{\frac{1}{4}},
$$
with $(x,y,t) \in \mathbb{R}^{2n} \times  \mathbb{R} $ and $| \cdot |$ the Euclidean norm.
\end{defin}

\begin{obs}
We show that $\norm{ \cdot }_{\mathbb{H}}$ is indeed a norm, as it satisfies the following properties:
\begin{enumerate}
\item
$\norm{(x,y,t)}_{\mathbb{H}} \geq 0, \ \norm{(x,y,t)}_{\mathbb{H}} = 0 \Leftrightarrow (x,y,t)=0$,
\item
$\norm{(x,y,t)*(x',y',t')}_{\mathbb{H}} \leq \norm{(x,y,t)}_{\mathbb{H}} + \norm{(x',y',t')}_{\mathbb{H}} $,
\item
Also $\norm{ \cdot }_{\mathbb{H}}$ is homogeneous of degree $1$ with respect to $\delta_r$:
$$
\norm{ \delta_r (x,y,t)}_{\mathbb{H}} = r  \norm{(x,y,t)}_{\mathbb{H}} ,
$$
where $\delta_r$ appears in Definition \ref{delta}.
\end{enumerate}
\end{obs}

\begin{proof}
The first and third point can be verified immediately. A proof of the triangle inequality can be found in $1.F$ in \cite{KORR}.
\end{proof}

\begin{obs}
The Korányi distance is left invariant, namely,
$$
d_{\mathbb{H}} (p*q,p*q')=d_{\mathbb{H}} (q,q'),  \quad  p,q,q' \in \mathbb{H}^n. 
$$
It is, moreover, homogeneous of degree $1$ with respect to $\delta_r$:
$$
d_\mathbb{H} \left ( \delta_r (p), \delta_r (q)  \right ) = r d_{\mathbb{H}} (p,q) .
$$
\end{obs}

\begin{obs}
\label{8.8stein}
We already mentioned that we use $| \cdot |$ for the Euclidean norm. One can prove the following inequality:
$$
|p| \leq \norm{p}_\mathbb{H} \leq |p|^{\frac{1}{2}} \ \ \text{ when } \norm{p}_\mathbb{H} \leq 1.
$$
\end{obs}

\subsection{Dimensions and Integration on $\mathbb{H}^n$}

In this subsection we add information on the Heisenberg group's topology, dimensions and integrals. General references are section 2.1 in \cite{GCmaster} and section 2.2.3 in \cite{CDPT}.


\begin{obs}      
The topology induced by the Korányi metric is equivalent to the Euclidean topology on $\mathbb{R}^{2n+1}$. The Heisenberg group $\mathbb{H}^n$ becomes, then, a locally compact topological group. As such, it has the \emph{right-invariant} and the \emph{left-invariant Haar measure}.
\end{obs}


\begin{defin}      
\label{haar}
We call an outer measure $\mu$ the \emph{left-invariant} (or \emph{right-invariant}) \emph{Haar measure}, on a locally compact Hausdorff topological group $G$, if the followings are satisfied:
\begin{itemize}
\item
$\mu(gE)=\mu(E) \text{ with  } E \subseteq G \text{ and } g \in G, \text{ where } gE:=\{ga ; \ a \in E\}$ \\
$\left ( \text{or } \mu(Eg)=\mu(E) \text{ with } E \subseteq G \text{ and } g \in G, \text{ where } Eg:=\{ag ; \ a \in E\}\right )  $,
\item
$\mu(K)< \infty, \text{ for all }  K \subset \subset G$,
\item
$\mu$ is outer regular: $\mu(E)=\inf \{ \mu(U) ;\ E \subseteq U \subseteq G, U \emph{\emph{ open}}\}, \  E \subseteq G $,
\item
$\mu$ is inner regular: $\mu(E)=\sup \{ \mu(K) ; \ K \subseteq E \subseteq G, K \emph{\emph{ compact}}\}, \  E \subseteq G$.
\end{itemize}
\end{defin}

\begin{obs}[see, among others, after remark 2.2 in \cite{FSSC2001}]
The ordinary Lebesgue measure on $\mathbb{R}^{2n+1}$ is invariant under both left and right translations on $\mathbb{H}^n$. In other words, the Lebesgue measure is both a left and right invariant Haar measure on $\mathbb{H}^n$.
\end{obs}

\begin{obs}      
\label{dimension}
It is easy to see that, denoting  the ball of radius $r>0$ as
$$
B_\mathbb{H}(0,r):=\{ (x,y,t) \in \mathbb{H}^n ; \ \norm{(x,y,t)}_\mathbb{H} <r \},
$$
a change of variables gives
$$
|B_\mathbb{H}(0,r)|=
\int_{B_\mathbb{H}(0,r)} dxdydt=
r^{2n+2} \int_{B_\mathbb{H}(0,1)} dxdydt =
 r^{2n+2}|B_\mathbb{H}(0,1)|.
$$
Thus $2n+2$ is the Hausdorff dimension of $ \left (\mathbb{H}^n, d_\mathbb{H} \right )$, which concides with its homogeneous dimension.
\end{obs}

\begin{no}\label{CCK}
Consider $S \subseteq \mathbb{H}^n$. we denote its Hausdorff dimension with respect to the Euclidean distance as
$$
\dim_{\mathcal{H}_{E}} S,
$$
while its Hausdorff dimension with respect to the Carnot-Carathéodory and Korányi distances as
$$
\dim_{\mathcal{H}_{cc}} S= \dim_{\mathcal{H}_{\mathbb{H}}} S.
$$
\end{no}

\chapter{Differential Forms and Rumin Cohomology}\label{DFaRC}

In this chapter we will present the precise definition of the Rumin complex 
 in any Heisenberg group (Section \ref{rumincomplex}). Then, to give a practical feeling of the difference between the de Rham and the Rumin complexes, we will write explicitly the differential complex of the Rumin cohomology in $\mathbb{H}^1$ and $\mathbb{H}^2$ (Sections \ref{cohomology1} and \ref{cohomology2}). As general references for this chapter, one can look at \cite{RUMIN} and \cite{FSSC}.\\
This chapter is connected with Appendices \ref{computationH2}, \ref{A} and \ref{B}. 
Appendix \ref{computationH2} contains the proof of Proposition \ref{exH2}. Appendix \ref{A} presents the Rumin cohomology, in $\mathbb{H}^1$ and $\mathbb{H}^2$, using only one operator $d_c$, as opposed to the three operators ($d_Q, \ D, \ d_Q$ again) used more frequently in the literature. The main reference for this appendix is \cite{TRIP}. As it will be clear from this chapter, direct computation of the Rumin differential operator are more and more challenging as the dimension of the space grows: Appendix \ref{B} offers the formulas to compute the dimension of the spaces involved in the Rumin complex for any dimension. There are also examples for $n=1,\dots,5$ which clearly show such computational challenges. 


\section{The Rumin Complex}\label{rumincomplex}

In this section we precisely present the definition of the Rumin complex in the general Heisenberg group $\mathbb{H}^n$. We start giving some basic definitions that can be found, for instance, in \cite{RUMIN} and \cite{FSSC}:

\begin{defin}\label{def_forms}
Consider $0\leq k \leq 2n+1$ and recall $\Omega^k$ from Definition \ref{kdim}. We denote:
\begin{itemize}
\item
$I^k := \{ \alpha \wedge \theta + \beta \wedge d \theta ; \  \alpha \in \Omega^{k-1}, \ \beta \in \Omega^{k-2}  \}$,
\item
$J^k :=\{ \alpha \in \Omega^{k}; \  \alpha \wedge \theta =0, \  \alpha \wedge d\theta=0   \}$.
\end{itemize}
\end{defin}

\begin{defin}[Rumin complex]\label{complexHn}
The Rumin complex, due to Rumin in \cite{RUMIN}, is given by
$$
0 \to \mathbb{R} \to C^\infty  \stackrel{d_Q}{\to} \frac{\Omega^1}{I^1}  \stackrel{d_Q}{\to}  \dots \stackrel{d_Q}{\to} \frac{\Omega^n}{I^n} \stackrel{D}{\to} J^{n+1}    \stackrel{d_Q}{\to} \dots   \stackrel{d_Q}{\to} J^{2n+1} \to 0,
$$
where $d$ is the standard differential operator and, for $k < n$,
$$
d_Q( [\alpha]_{I^k} ) :=  [d \alpha]_{I^{k+1}},
$$
while, for $k \geq n +1$,
$$
d_Q := d_{\vert_{J^k}}.
$$
The second order differential operator $D$ will be defined at the end of this section.
\end{defin}

\begin{rem}[proposition at page 286 in \cite{RUMIN}]
This structure defines indeed a complex. In other words, applying two consequential operators in the chain gives zero.
\end{rem}

\begin{rem}
When $k=1$, $d_Q$ is the same as $d_H$, from Definition \ref{dHHH}.
\end{rem}

\begin{no}
The spaces of the kind $\frac{\Omega^k}{I^k} $ are called \emph{low dimensional}, while the spaces $ J^{k}  $'s \emph{high dimensional} or \emph{low codimensional}.
\end{no}

\begin{rem}
From the definition of $I^k$, $k=1,\dots,n$, one can see that $\alpha \wedge \theta \in I^k$ for any $\alpha \in \Omega^{k-1}$. This means that, in modulo, $\theta$ is never present in the low dimensional spaces $\frac{\Omega^k}{I^k} $'s.\\
On the other hand, every $\beta \in J^k$ must be of the kind $\beta=\beta' \wedge \theta$ (as this is the only way to satisfy the condition $\beta \wedge \theta=0$). This means that $\theta$ will always be present the the high dimensional spaces $J^k$'s.
\end{rem}

\noindent
In order to be able to define $D$, some preliminary work is needed:

\begin{obs}\label{first}
First of all, notice that the definition, for $k < n$, of $d_Q( [\alpha]_{I^k} ) :=  [d \alpha]_{I^{k+1}}$  is well posed.
\end{obs}

\begin{proof}
The equality $[\alpha]_{I^k} = [\beta]_{I^k}$ means  $ \beta - \alpha \in I^k$, which implies
\begin{align*}
 \beta - \alpha = \sigma \wedge \theta + \tau \wedge d \theta,
\end{align*}
for some $\sigma \in \Omega^{k-1}, \ \tau \in \Omega^{k-2}$. Then one can write
\begin{align*}
d \beta - d \alpha =& d \sigma \wedge \theta +(- 1)^{k-1} \sigma \wedge d \theta +d \tau \wedge d \theta + 0\\
=& d \sigma \wedge \theta + ( (- 1)^{k-1} \sigma  +d \tau ) \wedge d \theta \in I^{k+1}.
\end{align*}
Then $[d \alpha]_{I^{k+1}} = [d \beta]_{I^{k+1}}.$ This gives the well-posedness.
\end{proof}

\begin{no}\label{middleequivclass}
Let $\gamma \in \Omega^{k-1}$ and consider the equivalence class
$$
  {\prescript{}{}\bigwedge}^k \mathfrak{h}_1  =       \left \{  \beta \in \Omega^k ; \ \beta =0 \ \text{or} \ \beta \wedge \theta \neq 0 \right  \}  \cong  \frac{\Omega^k}{ \{ \gamma \wedge \theta \} }    ,
$$
where $ {\prescript{}{}\bigwedge}^k \mathfrak{h}_1 $ appears in Definition \ref{kdim} and we write $\{ \gamma \wedge \theta \} = \{ \gamma \wedge \theta ; \ \gamma \in \Omega^{k-1} \}$ for short. The equivalence is given by $\beta \mapsto ( \beta)_{\vert_{ {\prescript{}{}\bigwedge}^{k} \mathfrak{h}_1 }}$.\\
Then, given $ \alpha \in \Omega^k$, denote $[\alpha]_{ \{ \gamma \wedge \theta \} }$ an element in this equivalence class. \\
\end{no}

\begin{obs}\label{finalequivclass}
Let $\gamma \in \Omega^{k-1}$  and $\beta \in \Omega^{k-2}$. One can see, straight by the definition of $I^k$, that
$$
\frac{   \frac{\Omega^k}{ \{ \gamma \wedge \theta \} } }{  \{ \beta \wedge d \theta \}    } \cong   \frac{\Omega^k}{I^k}, \quad  k=1,\dots,n,
$$
where we also write $\{ \beta \wedge d\theta \} = \{ \beta \wedge d\theta ; \ \beta \in \Omega^{k-2} \}$ for short.
\end{obs}

\noindent
The following lemma is necessary to define the second order differential operator $D$. Given $[\alpha]_{ \{ \gamma \wedge \theta \} } \in \frac{\Omega^n}{ \{ \gamma \wedge \theta \} }$, a lifting of $[\alpha]_{ \{ \gamma \wedge \theta \} }$ is any $\alpha' \in \Omega^n$ such that $ [\alpha]_{ \{ \gamma \wedge \theta \} } =  [\alpha']_{ \{ \gamma \wedge \theta \} }$.

\begin{lem}[Rumin \cite{RUMIN}, page 286]\label{lemma}
For every form $[\alpha]_{ \{ \gamma \wedge \theta \} } \in \frac{\Omega^n}{ \{ \gamma \wedge \theta \} }$, there exists a unique lifting $\tilde{\alpha} \in \Omega^n$ of $[\alpha]_{ \{ \gamma \wedge \theta \} }$ so that $d \tilde{\alpha} \in J^{n+1}$.
\end{lem}

\begin{proof}
Note that this proof is not exactly the one given by Rumin, but still follows the same steps.\\
Let $ \alpha \in \Omega^n  $ $ \left ( \text{so }[\alpha]_{ \{ \gamma \wedge \theta \} } \in \frac{\Omega^n}{ \{ \gamma \wedge \theta \} } \right )$ and define
$$
\tilde{\alpha}:= \alpha + \beta \wedge \theta \in \Omega^n,
$$
where $\beta  \in {\prescript{}{}\bigwedge}^{n-1} \mathfrak{h}_1.
$ 
Then compute
\begin{align*}
\theta \wedge d \tilde{\alpha} 
=&  \theta \wedge d \alpha + \theta \wedge d( \beta \wedge \theta)\\
=& \theta \wedge d \alpha +  \theta \wedge d \beta \wedge \theta+(- 1)^{\vert \beta \vert} \theta \wedge \beta \wedge d \theta\\
=&  \theta \wedge d \alpha + (- 1)^{\vert \beta \vert} \theta \wedge \beta \wedge d \theta\\
=&  \theta \wedge d \alpha + (- 1)^{\vert \beta \vert} \theta \wedge  d \theta \wedge \beta\\
=& \theta \wedge d \alpha + (- 1)^{\vert \beta \vert} \theta \wedge L(\beta) \\
=& \theta \wedge \left ( d \alpha + (- 1)^{\vert \beta \vert}  L(\beta) \right ),
\end{align*}
where $L$ (see $2$ in \cite{RUMIN}) is the isomorphism  
\begin{align*}
L:  {\prescript{}{}\bigwedge}^{n-1} \mathfrak{h}_1  \to {\prescript{}{}\bigwedge}^{n+1} \mathfrak{h}_1, \  \beta  \mapsto d \theta \wedge \beta .
\end{align*}
Notice that, since $d \alpha \in \Omega^{n+1}$, we can divide it as
$$
d \alpha = (d \alpha)_{\vert_{ \left ( {\prescript{}{}\bigwedge}^{n+1} \mathfrak{h}_1  \right )^\perp }}   +   (d \alpha)_{\vert_{ {\prescript{}{}\bigwedge}^{n+1} \mathfrak{h}_1 }},
$$
where
$$
\theta \wedge (d \alpha)_{\vert_{ \left ( {\prescript{}{}\bigwedge}^{n+1} \mathfrak{h}_1  \right )^\perp }}   =0,
$$
and, by isomorphism, there exists a unique $\beta$ so that
$$
  (- 1)^{\vert \beta \vert}  L(\beta)=- (d \alpha)_{\vert_{ {\prescript{}{}\bigwedge}^{n+1} \mathfrak{h}_1 }}.
$$
With such a choice of $\beta$ one gets
$$
\theta \wedge \left ( d \alpha + (- 1)^{\vert \beta \vert}  L(\beta) \right ) =0.
$$
Then
$$
\theta \wedge d \tilde{\alpha} =0,
$$
and, finally, also
$$
d \theta \wedge d \tilde{\alpha} = d(\theta \wedge d \tilde{\alpha}) =0.
$$
Then, by definition, $d \tilde{\alpha} \in J^{n+1}$.
\end{proof}

\begin{obs}\label{explicitalphatilde}
In the proof of Lemma \ref{lemma}, instead of $\beta$, we could have chosen $\beta':=(- 1)^{\vert \beta \vert} \beta$, which would give
$$
   L(\beta')=- (d \alpha)_{\vert_{ {\prescript{}{}\bigwedge}^{n+1} \mathfrak{h}_1 }},
$$
or, equivalently,
$$
 \beta' = L^{-1} \left (- (d \alpha)_{\vert_{ {\prescript{}{}\bigwedge}^{n+1} \mathfrak{h}_1 }}   \right ).
$$
Then the lifting can be written explicitly as
$$
\tilde{\alpha} = \alpha +  L^{-1} \left (- (d \alpha)_{\vert_{ {\prescript{}{}\bigwedge}^{n+1} \mathfrak{h}_1 }}   \right ) \wedge \theta.
$$
\end{obs}

\begin{defin}\label{D}
Using the observation above, finally we can define $D$ as
$$
D( [\alpha]_{I^n} ) := d \tilde{\alpha} = d \left ( \alpha +  L^{-1} \left (- (d \alpha)_{\vert_{ {\prescript{}{}\bigwedge}^{n+1} \mathfrak{h}_1 }} \right ) \wedge \theta \right ),
$$
and the definition is well-posed.
\end{defin}


\section{Cohomology of $\mathbb{H}^1$}\label{cohomology1}

In this section we explicitly write the differential complex of the Rumin cohomology in $\mathbb{H}^1$ and compare it to the de Rham cohomology. 
Furthermore, this sets a method for the more challenging case of $\mathbb{H}^2$, as well as hints at the qualitative difference between the first Heisenberg group $\mathbb{H}^1$ and all the others.

\begin{obs}
In the case $n=1$, the spaces of the Rumin cohomology presented in Definition \ref{def_forms} are reduced to
\begin{align*}
\Omega^1 &= \spn  \{ dx, dy, \theta  \},\\
I^1&=\spn \{ \theta \},\\
\frac{\Omega^1}{I^1} &\cong  \spn \{ dx, dy \}  ,     \\
J^2&= \spn \{ dx \wedge \theta, dy \wedge \theta \},\\
J^3&= \spn \{ dx \wedge  dy \wedge \theta \}.
\end{align*}
Moreover, in this case the isomorphism $L : {\prescript{}{}\bigwedge}^0 \mathfrak{h}_1 \to {\prescript{}{}\bigwedge}^2 \mathfrak{h}_1 $ is given by
\begin{align*}
L :   \spn \{ f \}  \to  \spn \{ dx \wedge  dy\}, \  f  \mapsto   -f dx \wedge dy.
\end{align*}
\end{obs}

\noindent
The following proposition shows the explicit action of each differential operator in the Rumin complex of $\mathbb{H}^1$.

\begin{prop}[Explicit Rumin complex in $\mathbb{H}^1$]\label{exH1}
In the case $n=1$, the Rumin complex presented in Definition \ref{complexHn} is becomes
\begin{equation*}
0 \to \mathbb{R} \to C^\infty \stackrel{d_Q^{(1)}}{\to}  \frac{\Omega^1}{I^1}\stackrel{D}{\to},
J^{2}    \stackrel{d_Q^{(3)}}{\to}    J^{3} \to 0
\end{equation*}
with
\begin{align*}
f   &\mathbin{ \stackrel{d_Q^{(1)}}{\mapsto} }  [Xf dx + Yf dy]_{  I^1 },\\
 [\alpha_1 dx + \alpha_2 dy]_{  I^1 }    &\mathbin{  \stackrel{D}{\mapsto} }  (XX \alpha_2 - XY \alpha_1 -T\alpha_1 )dx\wedge \theta  + (YX \alpha_2 - YY\alpha_1 -T\alpha_2)dy\wedge \theta, \\
\alpha_1 dx\wedge \theta  + \alpha_2 dy\wedge \theta   &\mathbin{  \stackrel{d_Q^{(3)}}{\mapsto} }  (X\alpha_2 -Y\alpha_1 )  dx \wedge dy \wedge \theta.
\end{align*}
\end{prop}




\begin{proof}
This proposition can be proved by simple computations. Two of the three cases are trivial.\\
Indeed, by Definition \ref{complexHn} and Observation \ref{df}, we have
$$
d_Q^{(1)} f  =     [Xf dx + Yf dy]_{  I^1 }.
$$
By the same definition and observation, we also get
\begin{align*}
d_Q^{(3)} (  \alpha_1 dx\wedge \theta  + \alpha_2 dy\wedge \theta  ) =&  Y\alpha_1 dy \wedge  dx \wedge  \theta + X\alpha_2  dx \wedge dy \wedge \theta \\
=&  (X\alpha_2 -Y\alpha_1 )  dx \wedge dy \wedge \theta .
\end{align*}
Finally we have to compute $D$. We remind that, by Observation \ref{df},
$$
d g = Xg dx + Y g dy + Tg \theta,
$$
with $g: U\subseteq \mathbb{H}^1 \to \mathbb{R}$ smooth.\\
Consider now $\alpha = \alpha_1 dx + \alpha_2 dy \in \Omega^1$. Then $[\alpha]_{  I^1 }=[\alpha_1 dx + \alpha_2 dy]_{  I^1 } \in  \frac{\Omega^1}{I^1} $, and the (full) exterior derivative of $\alpha$ is:
\begin{align*}
d \alpha
=& d ( \alpha_1 dx + \alpha_2 dy  ) = Y \alpha_1 dy \wedge dx + T\alpha_1 \theta \wedge dx +
X \alpha_2 dx \wedge dy + T\alpha_2 \theta \wedge dy \\
=&( X \alpha_2 - Y \alpha_1 ) dx \wedge dy - T\alpha_1 dx \wedge  \theta   -  T\alpha_2 dy \wedge  \theta.
\end{align*}
Then
$$
 (d \alpha)_{\vert_{ {\prescript{}{}\bigwedge}^{1} \mathfrak{h}_1 }} = ( X \alpha_2 - Y \alpha_1 ) dx \wedge dy = - ( X \alpha_2 - Y \alpha_1 ) d \theta,
$$
and
$$
 L^{-1} \left (- (d \alpha)_{\vert_{ {\prescript{}{}\bigwedge}^{1} \mathfrak{h}_1 }} \right ) =  X \alpha_2 - Y \alpha_1 .
$$
Finally
\begin{align*}
D([\alpha]_{I^1} )  =& d \left ( \alpha +   ( X \alpha_2 - Y \alpha_1 ) \theta  \right )\\
=& d \alpha + d \left ( X \alpha_2 - Y \alpha_1 \right ) \wedge \theta + \left ( X \alpha_2 - Y \alpha_1 \right ) d \theta\\
=&    ( X \alpha_2 - Y \alpha_1 ) dx \wedge dy - T\alpha_1 dx \wedge  \theta   -  T\alpha_2 dy \wedge  \theta \\
&+  X( X \alpha_2 - Y \alpha_1 ) dx \wedge \theta + Y ( X \alpha_2 - Y \alpha_1 ) dy \wedge \theta  - ( X \alpha_2 - Y \alpha_1 ) dx \wedge dy \\
=&  ( XX \alpha_2 - XY \alpha_1 - T\alpha_1 ) dx \wedge \theta + (YX \alpha_2 -YY \alpha_1 -  T\alpha_2 ) dy \wedge \theta.
\end{align*}
\end{proof}


\section{Cohomology of $\mathbb{H}^2$}\label{cohomology2}

In this section, as we did in the previous one for $\mathbb{H}^1$, we explicitly write the differential complex of the Rumin cohomology in $\mathbb{H}^2$. The computation is quantitative more challenging than the previous one and so we report it in Appendix \ref{computationH2}. In this case the bases of the spaces of the complex have more variety, as one must take into account more possible combinations than in the previous case. In a qualitative sense, this is caused by the fact that the algebra of $\mathbb{H}^n$ allows a strict subalgebra of step $2$ only for $n >1$.

\begin{obs}
\label{obsH2}
For $n=2$, the spaces of the Rumin cohomology presented in Definition \ref{def_forms} are reduced to
\begin{align*}
\Omega^1 &= \spn \{ dx_1,  dx_2, dy_1, dy_2, \theta \},\\
I^1&=\spn \{ \theta \},\\
\frac{\Omega^1}{I^1} &\cong  \spn \{dx_1,  dx_2, dy_1, dy_2 \},\\
\Omega^2 &= \spn \{ dx_1 \wedge  dx_2, dx_1 \wedge dy_1, dx_1 \wedge dy_2, dx_1 \wedge \theta, dx_2 \wedge dy_1, dx_2 \wedge dy_2, \\
& \hspace{7.8cm} dx_2 \wedge \theta,  dy_1 \wedge dy_2, dy_1 \wedge \theta, dy_2 \wedge \theta \} ,\\
I^2&=\spn \{  dx_1\wedge \theta,  dx_2\wedge \theta, dy_1\wedge \theta, dy_2 \wedge \theta, dx_1 \wedge dy_1 + dx_2 \wedge dy_2  \},\\
\frac{\Omega^2}{I^2} &\cong  \spn \{
 dx_1 \wedge  dx_2, dx_1 \wedge dy_2,  dx_2 \wedge dy_1,  dy_1 \wedge dy_2 \}  \oplus \frac{  \spn \{  dx_1 \wedge dy_1,  dx_2 \wedge dy_2     \}     }{     \spn \{ dx_1 \wedge dy_1 + dx_2 \wedge dy_2 \}    } ,\\
&\cong   \spn \{  dx_1 \wedge  dx_2, dx_1 \wedge dy_2,  dx_2 \wedge dy_1,  dy_1 \wedge dy_2,  dx_1 \wedge dy_1 - dx_2 \wedge dy_2 \} ,\\
J^3&= \spn \{ dx_1 \wedge  dx_2 \wedge \theta, dx_1 \wedge dy_2 \wedge \theta,  dx_2 \wedge dy_1 \wedge \theta,  dy_1 \wedge dy_2 \wedge \theta,\\
&\hspace{8.8cm}   dx_1 \wedge dy_1 \wedge \theta - dx_2 \wedge dy_2 \wedge \theta \},\\
J^4&= \spn \{ dx_1 \wedge  dx_2 \wedge dy_1 \wedge \theta, dx_1 \wedge dx_2 \wedge dy_2 \wedge \theta,  dx_1 \wedge  dy_1 \wedge dy_2 \wedge \theta,\\
&\hspace{10.5cm}  dx_2 \wedge  dy_1 \wedge dy_2 \wedge \theta \},\\
J^5&= \spn \{ dx_1 \wedge  dx_2 \wedge dy_1  \wedge dy_2 \wedge \theta  \}.
\end{align*}
Note that, in rewriting $\frac{\Omega^2}{I^2} $, we simply observe that $\{ dx_1 \wedge dy_1,  dx_2 \wedge dy_2 \}$ and $\{  dx_1 \wedge dy_1 + dx_2 \wedge dy_2,  dx_1 \wedge dy_1 - dx_2 \wedge dy_2 \}$ span the same subspace.
\end{obs}


\begin{obs}
In this case the isomorphism $L$ acts as follows
\begin{align*}
L :  {\prescript{}{}\bigwedge}^1 \mathfrak{h}_1   & \to {\prescript{}{}\bigwedge}^3 \mathfrak{h}_1 ,\\
 \omega & \mapsto  \omega \wedge d \theta,\\
 dx_1 & \mapsto  -dx_1 \wedge dx_2 \wedge dy_2,\\
 dy_1 & \mapsto  -  dy_1 \wedge dx_2 \wedge   dy_2,\\
dx_2 & \mapsto -  dx_2 \wedge dx_1 \wedge  dy_1,\\
 dy_2 & \mapsto -  dy_2 \wedge  dx_1 \wedge  dy_1.
\end{align*}
\end{obs}

\begin{rem}
Notice in particular that in the highest low dimensional space, $\frac{\Omega^2}{I^2}$, and in the lowest high dimensional space, $J^3$, there are generators that did not appear in the case $n=1$, namely $dx_1 \wedge dy_1 - dx_2 \wedge dy_2$ and $ dx_1 \wedge dy_1 \wedge \theta - dx_2 \wedge dy_2 \wedge \theta$ respectively. This is due to the fact that th first Heisenberg group $\mathbb{H}^1$ is the only Heisenberg group to be also a free group.
\end{rem}


\begin{prop}[Explicit Rumin complex in $\mathbb{H}^2$]\label{exH2}
\begin{equation*}
0 \to \mathbb{R} \to C^\infty  \stackrel{d_Q^{(1)}}{\to} \frac{\Omega^1}{I^1}
 \stackrel{d_Q^{(2)}}{\to}   \frac{\Omega^2}{I^2}
\stackrel{D}{\to}  J^{3}  
 \stackrel{d_Q^{(3)}}{\to}    J^{4} 
  \stackrel{d_Q^{(4)}}{\to}
 J^{5} \to 0, \quad \text{with}
\end{equation*}
\begin{align*}
 f & \mathbin{  \stackrel{d_Q^{(1)}}{\mapsto} } [X_1f dx_1 + X_2f dx_2+ Y_1f dy_1+ Y_2f dy_2 ]_{I^1}  ,\\
& \\
& [ \alpha_1 dx_1 + \alpha_2 dx_2  + \alpha_3 dy_1+ \alpha_4 dy_2 ]_{I^1} \\
&  \mathbin{  \stackrel{d_Q^{(2)}}{\mapsto} }  \Big [ ( X_1 \alpha_2  - X_2 \alpha_1 ) dx_1 \wedge dx_2 + ( X_2 \alpha_3   - Y_1 \alpha_2 ) dx_2 \wedge dy_1\\
&\ \ \ \ \  + ( Y_1 \alpha_4   -  Y_2 \alpha_3  ) dy_1 \wedge dy_2  + (    X_1 \alpha_4  -  Y_2 \alpha_1 ) dx_1 \wedge dy_2\\
& \ \ \ \ \  + \left ( \frac{ X_1 \alpha_3 - Y_1 \alpha_1 -  X_2 \alpha_4  + Y_2 \alpha_2    }{2} \right ) ( dx_1 \wedge dy_1 - dx_2 \wedge dy_2 ) \Big ]_{I^2},\\
& \\
& [\alpha_1 dx_1 \wedge  dx_2 + \alpha_3  dx_1 \wedge dy_2 + \alpha_4 dx_2 \wedge dy_1 +\alpha_6 dy_1 \wedge dy_2 \\
& + \beta (dx_1 \wedge dy_1 - dx_2 \wedge dy_2)    ]_{I^2}\\
&  \mathbin{\stackrel{D}{\mapsto} }  \left [ (  - X_1 Y_1 -Y_2 X_2) \alpha_1  +X_2 X_2 \alpha_3  -X_1 X_1 \alpha_4   +2X_1 X_2 \beta  \right  ] dx_1 \wedge  dx_2 \wedge \theta  \\
&\ \ \ \ \  +\left [  -Y_2 Y_2 \alpha_1 +(X_2 Y_2 \alpha_3 -X_1 Y_1 )\alpha_3  +X_1 X_1 \alpha_6  +2X_1 Y_2 \beta   \right  ]   dx_1 \wedge dy_2 \wedge \theta \\
&\ \ \ \ \  +\left [ +Y_1 Y_1 \alpha_1     +(Y_1 X_1  -Y_2 X_2) \alpha_4   -X_2 X_2 \alpha_6  -2 X_2 Y_1 \beta   \right  ] dx_2 \wedge dy_1 \wedge \theta  \\
&\ \ \ \ \  +\left [  -Y_1 Y_1 \alpha_3 +Y_2 Y_2 \alpha_4  +(Y_1 X_1 + X_2 Y_2) \alpha_6   + 2Y_1 Y_2 \beta   \right  ]   dy_1 \wedge dy_2 \wedge \theta \\
&\ \ \ \ \  +\left [  -Y_1 Y_2 \alpha_1   +Y_1 X_2 \alpha_3   -X_1 Y_2 \alpha_4    -X_1 X_2 \alpha_6    \right  ]   ( dx_1 \wedge  dy_1 \wedge \theta -dx_2 \wedge dy_2\wedge \theta),\\
& \\
& \alpha_1 dx_1 \wedge  dx_2 \wedge \theta +\alpha_2 dx_1 \wedge dy_2 \wedge \theta + \alpha_3  dx_2 \wedge dy_1 \wedge \theta + \alpha_4  dy_1 \wedge dy_2 \wedge \theta \\
& + \alpha_5( dx_1 \wedge dy_1 \wedge \theta - dx_2 \wedge dy_2 \wedge \theta) \\
 &  \mathbin{ \stackrel{d_Q^{(3)}}{\mapsto} } ( Y_1 \alpha_1  + X_1 \alpha_3 - X_2  \alpha_5)   dx_1 \wedge  dx_2 \wedge dy_1 \wedge \theta \\
&\ \ \ \ \  +( Y_2 \alpha_1  -X_2 \alpha_2 - X_1\alpha_5      )dx_1 \wedge  dx_2 \wedge dy_2 \wedge \theta\\
&\ \ \ \ \  +( -Y_1 \alpha_2 + X_1 \alpha_4   + Y_2  \alpha_5       )     dx_1 \wedge dy_1 \wedge dy_2 \wedge \theta \\
&\ \ \ \ \  +(  Y_2  \alpha_3  +  X_2    \alpha_4 +Y_1 \alpha_5 ) dx_2 \wedge dy_1 \wedge dy_2 \wedge \theta ,\\
& \\
& \alpha_1 dx_1 \wedge  dx_2 \wedge dy_1 \wedge \theta + \alpha_2 dx_1 \wedge dx_2 \wedge dy_2 \wedge \theta  \\
& + \alpha_3 dx_1 \wedge  dy_1 \wedge dy_2 \wedge \theta + \alpha_4  dx_2 \wedge  dy_1 \wedge dy_2 \wedge \theta \\
 & \mathbin{  \stackrel{d_Q^{(4)}}{\mapsto} }   ( -Y_2  \alpha_1 + Y_1  \alpha_2 - X_2  \alpha_3 + X_1  \alpha_4 )  dx_1 \wedge  dx_2 \wedge dy_1 \wedge dy_2 \wedge \theta .
\end{align*}
\end{prop}


\begin{proof}[Proof of Proposition \ref{exH2} is at Appendix \ref{computationH2}]
\end{proof}



\chapter{Pushforward and Pullback in $\mathbb{H}^n$}\label{PPHN}

In this chapter we define pushforward and pullback on $\mathbb{H}^n$ and, after some properties, we prove that the pullback by a contact map commutes with the Rumin differential at every dimension. Then we show explicit formulas for pushforward and pullback, in Heisenberg notations, in three different cases: for a general function, for a contact map and a contact diffeomorphism. Finally we present the same formulas in the Rumin cohomology.   
References for this chapter are \cite{KORR}, from which we use some results, \cite{RUMIN} and \cite{FSSC}.\\
Appendix \ref{explicitcommutation} concerns an explicit proof of commutation between pullback and Rumin differential. 


\section{Definitions and Properties}


\begin{defin}[Pushforward in $\mathbb{H}^n$]
Let $f:  U\subseteq \mathbb{H}^n \to \mathbb{H}^n$, $U$ open, $f=\left (f^1,\dots,f^{2n+1} \right ) \in \left [ C^1 (U,\mathbb{R}) \right  ]^{2n+1}$. The \emph{pushforward by $f$} is defined as follows:\\
if $k=1$, we set
\begin{align*}
f_*:=d f  : \mathfrak{h} &\to \mathfrak{h},\\
v&\mapsto v(f).
\end{align*}
If $k > 1$, we set
$$
f_* \equiv \Lambda_k d f  : {\prescript{}{}\bigwedge}_k \mathfrak{h} \to {\prescript{}{}\bigwedge}_k \mathfrak{h}
$$
to be the linear map satisfying
$$
\Lambda_k d f (v_1 \wedge \dots \wedge v_k) :=df (v_1) \wedge \dots \wedge d f (v_k),   
$$
i.e.,
$$
 f_* (v_1 \wedge \dots \wedge v_k) :=f_* v_1 \wedge \dots \wedge  f_* v_k,  
$$
for $v_1,\dots,v_k \in \mathfrak{h}$.
\end{defin}

\begin{defin}[Pullback in $\mathbb{H}^n$]
Let $f:  U\subseteq \mathbb{H}^n \to \mathbb{H}^n$, $U$ open, $f=\left (f^1,\dots,f^{2n+1} \right ) \in \left [ C^1 (U,\mathbb{R}) \right  ]^{2n+1}$. The \emph{pullback by $f$ }is defined by duality with respect to the pushforward as:
$$
f^* \equiv \Lambda^k d f  : \Omega^k \to \Omega^k, \text{ so that}
$$
$$
\langle \Lambda^k d f (\omega) \vert v \rangle =\langle \omega \vert \Lambda_k d f (v) \rangle,
$$
i.e.,
$$
\langle f^* (\omega) \vert v \rangle =\langle \omega \vert f_*  (v) \rangle.
$$
\end{defin}

\noindent
As in the Riemannian case, we have the following

\begin{obs}\label{pullprop} 
Let $f:  U\subseteq \mathbb{H}^n \to \mathbb{H}^n$, $U$ open, $f=\left (f^1,\dots,f^{2n+1} \right ) \in \left [ C^1 (U,\mathbb{R}) \right  ]^{2n+1}$, and $\alpha \in \Omega^k$,  $\beta \in \Omega^h$ differential forms. One can verify that
$$
f^* (\alpha \wedge \beta ) =f^* \alpha \wedge f^* \beta.
$$
\end{obs}



\begin{defin}\label{contact}
A map $f:  U\subseteq \mathbb{H}^n \to \mathbb{H}^n$, $U$ open, $f=\left (f^1,\dots,f^{2n+1} \right ) \in \left [ C^1 (U,\mathbb{R}) \right  ]^{2n+1}$ is a \emph{contact map} if 
$$
 f_* \left (\spn \{ X_1, \dots, X_n,Y_1,\dots, Y_n \} \right ) \subseteq \spn \{  X_1, \dots, X_n,Y_1,\dots, Y_n \}.
$$
\end{defin}

\begin{obs}
Remembering Notation \ref{notW}, the pushforward $f_*$ can be expressed in terms of the basis given by $\{W_1,\dots, W_{2n+1} \}$, where one can think of them as vectors: $W_j=e_j^T$ ($e_j$ the vector with value $1$ at position $j$ and zero elsewhere). 
Then the general pushforward matrix looks like
\begin{align*}
f_*=&
\left (
\begin{matrix}
\langle dw_1 , f_* W_1 \rangle                   & \ldots          & \langle dw_1 , f_* W_{2n+1} \rangle  \\
                                   \vdots                       &                    & \vdots\\
\langle dw_{2n+1} , f_* W_1 \rangle         &   \ldots        & \langle dw_{2n+1} , f_* W_{2n+1} \rangle  
\end{matrix}
\right ).
\end{align*}
If the function $f$ is also contact, then by definition we have
\begin{align*}
f_*=&
\left (
\begin{matrix}
\langle dw_1 , f_* W_1 \rangle                   & \ldots         & \langle dw_1 , f_* W_{2n} \rangle   & \langle dw_1 , f_* W_{2n+1} \rangle  \\
                                   \vdots                       &                    & \vdots       & \vdots   \\
\langle dw_{2n} , f_* W_1 \rangle         &   \ldots       & \langle dw_{2n} , f_* W_{2n} \rangle    & \langle dw_{2n} , f_* W_{2n+1} \rangle  \\
                                                     0         &   \ldots      &   0      & \langle dw_{2n+1} , f_* W_{2n+1} \rangle  
\end{matrix}
\right ).
\end{align*}
By definition of pulback, the pullback matrix is the transpose of the pushforward matrix, so
\begin{align*}
f^*=(f_*)^T=&
\left (
\begin{matrix}
\langle dw_1 , f_* W_1 \rangle                   & \ldots         & \langle dw_{2n} , f_* W_1 \rangle         & 0  \\
                                   \vdots                       &                    & \vdots                                                       & \vdots   \\
\langle dw_1 , f_* W_{2n} \rangle              &   \ldots       & \langle dw_{2n} , f_* W_{2n} \rangle    & 0 \\
 \langle dw_1 , f_* W_{2n+1} \rangle         &   \ldots      &    \langle dw_{2n} , f_* W_{2n+1} \rangle   & \langle dw_{2n+1} , f_* W_{2n+1}  \rangle    
\end{matrix}
\right ).
\end{align*}
This shows that an equivalent condition for contactness is to ask
$$
  f^* \theta = \lambda_f \theta, \quad \text{with} \quad \lambda_f =  \langle f^* \theta , T \rangle  .
$$
\end{obs}

\begin{obs}[See proposition 6 in \cite{KORREI}]
If $f : U \subseteq \mathbb{H}^n \to \mathbb{H}^n$ is a P-differentiable function from $\mathbb{H}^n$ to $\mathbb{H}^n$ as in the Definition \ref{dGGG}, then $f$ is a contact map.
\end{obs}

\begin{ex}\label{examplecontact}
The anisotropic dilation $\delta_r (\bar x, \bar y, t)=(r \bar x, r \bar y,r^2 t )$, $ (\bar x, \bar y, t) \in \mathbb{H}^n$, is a contact map. This will be shown later in Example \ref{contactdebt}.
\end{ex}

\begin{obs}[Section 2.B \cite{KORR}] \label{d_theta}  
Note that, if $f$ is a contact map and $\lambda_f =  \langle f^* \theta , T \rangle$, then
$$
f^* d \theta = d (f^* \theta) = d (\lambda_f  \theta) = d \lambda_f \wedge \theta + \lambda_f d\theta,
$$
where $d$ is the (full) exterior Riemannian derivative.
\end{obs}

\begin{obs}[Section 2.B \cite{KORR}]\label{section2.b} 
Note that, if $f$ is a contact map, $\lambda_f =  \langle f^* \theta , T \rangle$ and $v_1, v_2 \in  \spn \{ X_1,\dots,X_n, Y_1,\dots, Y_n \}$, then
$$
\langle d \theta \vert  f_* (v_1 \wedge v_2) \rangle  = \lambda_f \langle d\theta \vert  v_1 \wedge v_2  \rangle .
$$
\end{obs}

\begin{proof}
By the definition of pullback and Observation \ref{section2.b}, we have
\begin{align*}
\langle d \theta \vert  f_* (v_1 \wedge v_2) \rangle   & = \langle f^* d \theta \vert  v_1 \wedge v_2 \rangle \\
&=  \langle d \lambda_f \wedge \theta \vert v_1 \wedge v_2  \rangle + \langle  \lambda_f d\theta \vert v_1 \wedge v_2  \rangle = \lambda_f \langle   d \theta \vert v_1 \wedge v_2  \rangle.
\end{align*}
\end{proof}




\section{Commutation of Pullback and Rumin Complex in $\mathbb{H}^n$}

We consider a contact map $f$ on $\mathbb{H}^n$. We know that $f^*$ commutes with the 
exterior derivative $d$ and here we show that the commutation holds also for the differentials in the Rumin complex.\\\\
\noindent
Recall from Definitions \ref{complexHn} and \ref{D} that the Rumin complex 
is given by
$$
0 \to \mathbb{R} \to C^\infty  \stackrel{d_Q}{\to} \frac{\Omega^1}{I^1}  \stackrel{d_Q}{\to}  \dots \stackrel{d_Q}{\to} \frac{\Omega^n}{I^n} \stackrel{D}{\to} J^{n+1}    \stackrel{d_Q}{\to} \dots   \stackrel{d_Q}{\to} J^{2n+1} \to 0,
$$
where, for $k<n$, $d_Q( [\alpha]_{I^k} ) :=  [d \alpha]_{I^k},$ while, for $k \geq n+1$, $d_Q := d_{\vert_{J^k}}.$ For $k=n$, $D$ was the second order differential operator uniquely defined as $D( [\alpha]_{I^n} )  = d \left ( \alpha +  L^{-1} \left (- (d \alpha)_{\vert_{ {\prescript{}{}\bigwedge}^{n+1} \mathfrak{h}_1 }} \right ) \wedge \theta \right ).$

\begin{theor}\label{Fdc=dcF}
A smooth contact map $f: \mathbb{H}^n \to \mathbb{H}^n$ satisfies
$$
f^* d_Q = d_Q f^* \ \ \ \ \text{ for } k \neq n,
$$
and
$$
f^* D = D f^* \ \ \ \ \text{ for } k = n.
$$
Namely, the pullback by a contact map $f$ commutes with the 
operators of the Rumin complex.\\\\
To the best of my knowledge, this result does not apper in the literature, but the main steps were explained to me by Bruno Franchi in September 2017. We present here a complete proof.\\
A computationally more explicit proof of this statement is available in Appendix \ref{explicitcommutation}.
\end{theor}

\noindent
Before starting the proof of Theorem \ref{Fdc=dcF}, some lemmas and definitions are necessary. 

\begin{lem}\label{I}
Consider a smooth contact map $f: \mathbb{H}^n \to \mathbb{H}^n$ and write $\{ \gamma \wedge \theta \}=\{ \gamma \wedge \theta ; \ \gamma \in \Omega^{k-1}\}$. Then
$$
f^* (\{ \gamma \wedge \theta \}) \subseteq \{ \gamma \wedge \theta \} \quad  \text{ and } \quad   f^* (I^k) \subseteq I^k \quad  k=1,\dots,n.
$$
\end{lem}

\begin{proof}
First notice, from Definition \ref{contact} and Observation \ref{d_theta} that
$$
f^* \theta = \lambda_f \theta 
\quad \text{and} \quad
f^* d \theta =
d \lambda_f \wedge \theta +\lambda_f d\theta .
$$
Then one has
$$
f^* (\alpha \wedge \theta ) = f^* \alpha \wedge f^* \theta  =\lambda_f f^* \alpha \wedge  \theta \in \{ \gamma \wedge \theta \}.
$$
This means $f^* (\{ \gamma \wedge \theta \} ) \subseteq \{ \gamma \wedge \theta \}$. Furthermore
\begin{align*}
f^* (\alpha \wedge \theta + \beta \wedge d \theta) &= f^* \alpha \wedge f^* \theta + f^* \beta \wedge f^* d \theta \\
&
= \lambda_f f^* \alpha \wedge  \theta + f^* \beta \wedge (d \lambda_f \wedge \theta +\lambda_f d\theta)\\
&
 =( \lambda_f f^* \alpha  + f^* \beta \wedge d \lambda_f ) \wedge \theta + \lambda_f f^* \beta \wedge  d \theta
\in I^k
\end{align*}
for $\alpha \in \Omega^{k-1}$ and $\beta \in \Omega^{k-2}$. This means that $f^* (I^k) \subseteq I^k$.
\end{proof}

\begin{defin}
Recall the equivalence class in Notation \ref{middleequivclass}:\\
$\frac{\Omega^k}{ \{ \gamma \wedge \theta \} }    \cong    \left \{  \beta \in \Omega^k ; \ \beta =0 \ \text{or} \ \beta \wedge \theta \neq 0 \right  \}=  {\prescript{}{}\bigwedge}^k \mathfrak{h}_1 $ where $[\alpha]_{ \{ \gamma \wedge \theta \} }$ is an element in this equivalence class. Then consider a smooth contact map $f: \mathbb{H}^n \to \mathbb{H}^n$. We define a pullback on such equivalence class as
$$
f^* : \frac{\Omega^k}{ \{ \gamma \wedge \theta \} } \to \frac{\Omega^k}{ \{ \gamma \wedge \theta \} },
$$
where
$$
f^* ( [\alpha]_{ \{ \gamma \wedge \theta \} }   ):= [f^* (\alpha )]_{ \{ \gamma \wedge \theta \} }, \quad  [\alpha]_{ \{ \gamma \wedge \theta \} } \in \frac{\Omega^k}{ \{ \gamma \wedge \theta \} }.
$$
This definition is well posed, as follows from the following lemma.
\end{defin}

\begin{lem} \label{Iq}
Consider a smooth contact map $f: \mathbb{H}^n \to \mathbb{H}^n$ and $\{ \gamma \wedge \theta \}=\{ \gamma \wedge \theta ; \ \gamma \in \Omega^{k-1}\}$. Then we have
$$
[\alpha]_{ \{ \gamma \wedge \theta \} } = [\beta]_{ \{ \gamma \wedge \theta \} } \Rightarrow f^* ( [\alpha ]_{ \{ \gamma \wedge \theta \} } )= f^* ([ \beta]_{ \{ \gamma \wedge \theta \} }  ).
$$
\end{lem}

\begin{proof}
By definition of equivalence class we have that $[\alpha]_{ \{ \gamma \wedge \theta \} } = [\beta]_{ \{ \gamma \wedge \theta \} } $ means
$$
 \beta - \alpha \in \{ \gamma \wedge \theta \} ,
$$
which, by Lemma \ref{I}, implies
$$
f^*\beta - f^* \alpha = f^* (  \beta -  \alpha ) \in  \{ \gamma \wedge \theta \},
$$
which, again by definition, means
$$
   [f^* (\alpha )]_{ \{ \gamma \wedge \theta \} }  = [f^* (\beta )]_{ \{ \gamma \wedge \theta \} }.
$$
The claim follows by the definition of pushforward in this equivalence class.
\end{proof}

\begin{defin}\label{pushequivclass}
Consider another equivalence class as given in Observation \ref{finalequivclass}: 
$
\frac{   \frac{\Omega^k}{ \{ \gamma \wedge \theta \} } }{  \{ \beta \wedge d \theta \}    } =   \frac{\Omega^k}{I^k}
$ 
and consider a smooth contact map $f: \mathbb{H}^n \to \mathbb{H}^n$. Again there is a pushforward defined as
\begin{align*}
f^* : \frac{\Omega^k}{I^k}  \to \frac{\Omega^k}{I^k},
\end{align*}
where
\begin{align*}
 f^* ([\alpha]_{I^k}   )  := [f^* (\alpha )]_{I^k},   \quad  [\alpha]_{I^*} \in \frac{\Omega^k}{I^k}.
\end{align*}
Also this definition is well posed, as shown in the lemma below.
\end{defin}

\begin{lem}\label{Iqq}
Consider a smooth contact map $f: \mathbb{H}^n \to \mathbb{H}^n$. Then we have
$$
[\alpha]_{I^k} = [\beta]_{I^k} \Rightarrow  [f^* (\alpha )]_{I^k} = [f^* (\beta )]_{I^k}.
$$
\end{lem}

\begin{proof}
By definition $[\alpha]_{I^k} = [\beta]_{I^k}$ means
$$
  \beta - \alpha \in I^k,
$$
which implies, by Lemma \ref{I},
$$
 f^* \beta - f^* \alpha =  f^* (\beta - \alpha )   \in I^k.
$$
So
$$
   [f^* (\alpha )]_{I^k} = [f^* (\beta )]_{I^k}.
$$
\end{proof}

\noindent
After all these lemmas about lower order object in the Rumin comples, we show here one on higher order spaces.

\begin{lem}\label{lastone}
Consider a smooth contact map $f: \mathbb{H}^n \to \mathbb{H}^n$. Then
$$
f^* (J^k) \subseteq J^k.
$$
\end{lem}

\begin{proof}
By definition of $J^k$, $\alpha \in J^k$ means
$$
 \theta \wedge \alpha = 0 \ \text{ and } \ d \theta \wedge \alpha =0.
$$
Then
$$
0=f^* ( \theta \wedge \alpha )=  f^* \theta \wedge f^* \alpha = \lambda_f  \theta \wedge f^* \alpha,
$$
which implies
$$
  \theta \wedge f^* \alpha =0.
$$
Moreover,
\begin{align*}
0&=f^* ( d \theta \wedge \alpha )=  f^* (d \theta ) \wedge f^* \alpha =d ( \lambda_f \theta ) \wedge f^* \alpha \\
&= ( d  \lambda_f \wedge  \theta ) \wedge f^* \alpha +   \lambda_f d \theta  \wedge f^* \alpha.
\end{align*}
Since $\theta  \wedge f^* \alpha=0$, we get that
$$
 d \theta \wedge f^* \alpha =0 .
$$
And, finally,
$$
\theta \wedge f^* \alpha =0 \ \text{ and } \ d \theta \wedge f^* \alpha =0, \  \text{ i.e. }  \ f^* \alpha \in J^k.
$$
\end{proof}

\noindent
At last, we will prove a lemma regarding the case $k=n$. After this we am finally ready to prove the main theorem.

\begin{lem}\label{lemmaDtilde}
Consider a smooth contact map $f: \mathbb{H}^n \to \mathbb{H}^n$. We know that, by Lemma \ref{lemma}, every form $[\alpha]_{ \{ \gamma \wedge \theta \} } \in \frac{\Omega^n}{ \{ \gamma \wedge \theta \} } $  has a unique lifting $\tilde{\alpha} \in \Omega^n$ so that $d \tilde{\alpha} \in J^{n+1}$. Then we have
$$
\widetilde{f^* \alpha}= f^* \tilde{\alpha}.
$$
\end{lem}

\begin{proof}
Now we only have to prove the claim that $\widetilde{f^* \alpha}= f^* \tilde{\alpha}$.\\
Following the proof of Lemma \ref{lemma}, one knows that there exists a unique $ \beta \in {\prescript{}{}\bigwedge}^{n-1} \mathfrak{h}_1 $ so that
$$
 \tilde{\alpha} = \alpha + \beta \wedge \theta,
$$
and such $\beta$ is the only one for which the following condition is satisfied:
$$
\theta \wedge \left ( d \alpha + (- 1)^{\vert \beta \vert}  L(\beta) \right ) =0.
$$
Thus one has that
$$
f^* \tilde{\alpha} = f^* \alpha + f^* (\beta \wedge \theta ) = f^* \alpha + \lambda_f f^* \beta \wedge \theta.
$$
On the other hand, one can repeat the lifting process for $f^* \alpha \in  \frac{\Omega^n}{ \{ \gamma \wedge \theta \} } \cong {\prescript{}{}\bigwedge}^{n} \mathfrak{h}_1 $  (the congruence tells that if $\alpha$ belongs to $\frac{\Omega^n}{ \{ \gamma \wedge \theta \} } $, so does $f^*$).\\
Then there exists a unique $\gamma $ so that
$$
\widetilde{f^* \alpha} = f^* \alpha + \gamma \wedge \theta,
$$
where, as before, such $\gamma$ is the only one which satisfies
\begin{equation}\label{uniquecondition}
\theta \wedge \left ( d f^* \alpha + (- 1)^{\vert \gamma \vert}  L(\gamma) \right ) =0.
\end{equation}
To prove the claim one needs to show that $ \gamma = \lambda_f f^* \beta $. We can substitute $ \lambda_f f^* \beta$ in place of $\gamma$ in the condition \eqref{uniquecondition} and, by uniqueness, it is enough to show that 
\begin{equation}\label{11star}
\theta \wedge \left ( d f^* \alpha + (- 1)^{\vert \lambda_f f^* \beta \vert}  L(\lambda_f f^* \beta) \right ) =0.
\end{equation}
Indeed
\begin{align*}
d ( f^*  \alpha ) & =  
f^* (d  \alpha ) = 
 f^* \left [ d  \alpha+ L \left ( (- 1)^{\vert \beta \vert}\beta  \right ) -  L \left ( (- 1)^{\vert \beta \vert}\beta  \right ) \right ] \\
&
= f^* \left (d  \alpha+ L \left ( (- 1)^{\vert \beta \vert}\beta  \right ) \right ) -  f^* \left (  L \left ( (- 1)^{\vert \beta \vert}\beta  \right ) \right ) .
\end{align*}
Then
\begin{align*}
\theta \wedge d f^* \alpha & = \theta \wedge  f^* \left (d  \alpha+ L \left ( (- 1)^{\vert \beta \vert}\beta  \right ) \right ) 
-  \theta \wedge  f^* \left (  L \left ( (- 1)^{\vert \beta \vert}\beta  \right ) \right )\\
&
=
\lambda_f^{-1} f^* \left (    \theta \wedge   \left (d  \alpha+ (- 1)^{\vert \beta \vert}  L \left ( \beta  \right ) \right )  \right )     -  (- 1)^{\vert \beta \vert} \theta \wedge  f^* \left (   L \left ( \beta  \right ) \right ).
\end{align*}
Since  $  \theta \wedge   \left (d  \alpha+ (- 1)^{\vert \beta \vert}  L \left ( \beta  \right ) \right )=0$, we get
\begin{align*}
\theta \wedge d f^* \alpha & =  -  (- 1)^{\vert \beta \vert} \theta \wedge  f^* \left (   L \left ( \beta  \right ) \right )  =-  (- 1)^{\vert \beta \vert} \theta \wedge  f^* \left (  d\theta \wedge \beta  \right )   \\
& = -  (- 1)^{\vert \beta \vert} \theta \wedge  f^*   d\theta \wedge   f^* \beta = -  (- 1)^{\vert \beta \vert} \theta \wedge \lambda_f d \theta \wedge   f^*  \beta \\
&=    -  (- 1)^{\vert \beta \vert} \lambda_f \theta \wedge  d \theta \wedge   f^*  \beta,
\end{align*}
where, in the second to last equality, we used that $ f^*   d\theta = d f^*   \theta= d (\lambda_f \theta)= d \lambda_f \wedge \theta +\lambda_f d \theta$.\\
On the other hand, since $\vert \lambda_f f^* \beta \vert = \vert \beta \vert$,
\begin{align*}
  \theta \wedge \left (  (- 1)^{\vert \lambda_f f^* \beta \vert}  L(\lambda_f f^* \beta) \right ) &=(- 1)^{\vert \beta \vert} \theta \wedge \left (    L(\lambda_f f^* \beta) \right )\\
&=  (- 1)^{\vert \beta \vert} \theta \wedge \left (   d\theta \wedge \lambda_f f^* \beta \right ) \\
&= (- 1)^{\vert \beta \vert}  \lambda_f \theta \wedge   d\theta \wedge  f^* \beta .
\end{align*}
This shows that equation \eqref{11star} holds and thus ends the proof.
\end{proof}

\begin{proof}[Proof of Theorem \ref{Fdc=dcF}]
As the definition of Rumin complex is done by cases, we will also divide this proof by cases. As one may expect, the case $k=n$ is the one that requires more work.\\\\
\textbf{First case: $k \geq n+1$.}\\
If $k \geq n+1$, then $d_Q = d_{\vert_{J^k}}$ and we need to consider a differentual form $\alpha \in J^k$. The we have
$$
d_Q   \alpha  = d \alpha  \text{ on } J^k.
$$
So, since also $f^* \left (  \alpha  \right ) \in J^k$, we can already conclude that
$$
f^* \left (d_Q   \alpha  \right ) = f^* \left ( d \alpha  \right )  =  d   f^* \left (  \alpha  \right )  = d_Q f^* \left (  \alpha  \right ).
$$
\noindent
\textbf{Second case: $k < n$.}\\
For $k < n$, the definition says $d_Q \left ([  \alpha ]_{I^k} \right )  =[ d \alpha ]_{I^k}$ for any $\alpha \in \Omega^k$. 
Then, by Definition \ref{pushequivclass} of $f^*$,
$$
f^*d_Q ([  \alpha ]_{I^k}) = f^* \left ( [ d \alpha ]_{I^k}  \right )   = [  f^* d \alpha ]_{I^k} = [  d f^*  \alpha ]_{I^k} =d_Q [ f^*  \alpha ]_{I^k}= d_Q f^*  ( [ \alpha ]_{I^k}).
$$
\noindent
\textbf{Third and last case: $k =n$.}\\
We know that, by Lemma \ref{lemma}, every form $[\alpha]_{ \{ \gamma \wedge \theta \} } \in \frac{\Omega^n}{ \{ \gamma \wedge \theta \} } $ 
 has a unique lifting $\tilde{\alpha} \in \Omega^n$ so that $d \tilde{\alpha} \in J^{n+1}$. Given this existence and unicity, we can now define the following operator:
\begin{align*}
\tilde{D} : \frac{\Omega^n}{ \{ \gamma \wedge \theta \} } & \to J^{n+1},\\
\tilde{D} \left ( [\alpha]_{ \{ \gamma \wedge \theta \} } \right ) & := d\tilde{\alpha}.
\end{align*}
By Lemma \ref{lemmaDtilde}, we know that $\widetilde{f^* \alpha}= f^* \tilde{\alpha}$ holds. Then
$$
\tilde{D} f^* [  \alpha]_{ \{ \gamma \wedge \theta \} }= \tilde{D} [ f^* \alpha]_{ \{ \gamma \wedge \theta \} }
= d \left ( \widetilde{f^* \alpha} \right )= d ( f^* \tilde{\alpha}) =  f^* d \tilde{\alpha} = f^* \tilde{D} ( [\alpha]_{ \{ \gamma \wedge \theta \} } ).
$$
In Definition \ref{D}, we posed the second-order differential operator $D$ to be $D  (  [\alpha]_{I^n}  )  = d \tilde{\alpha}$, which means
$$
D  (  [\alpha]_{I^n}  ) = \tilde{D} \left ( [\alpha]_{ \{ \gamma \wedge \theta \} } \right ).
$$
%
Then, since $ D  [\alpha ]_{I^n} \in J^{n+1}$ and using the definitions of $D$ and $f^*$, as well as the fact that $f^*\tilde{D}=\tilde{D}f^* $, we get
\begin{align*}
 f^*    D  [\alpha ]_{I^n} 
& =  f^* \tilde{D} \left ( [\alpha]_{ \{ \gamma \wedge \theta \} } \right ) 
= \tilde{D} f^* \left ( [\alpha]_{ \{ \gamma \wedge \theta \} } \right ) 
= \tilde{D} \left ( [  f^*\alpha]_{ \{ \gamma \wedge \theta \} } \right )\\
& = D  (  [f^* \alpha]_{I^n}  ) =
 D  f^* [ \alpha]_{I^n}.
\end{align*}
This concludes the proof.
\end{proof}


\section{Derivative of Compositions, Pushforward and Pullback}\label{dercompuspul}

\noindent
In this section we start by writing the derivatives of composition of functions. After that we will move to writing explicitely the pushforward and pullback by such functions, in different situations. Unfortunately, if we ask only regularity but no contact properties, the calculation becomes quite heavy and, since its meaning is relative (as contactness is a natural assumption), we will not push this case after the first derivatives. One can see section 1.D in  \cite{KORR} as a reference. 

\subsection{General Maps}

\noindent
First we introduce the following notation:

\begin{no}
Remember Notation \ref{notW} and let $j=1,\dots,2n$. Define
$$
\tilde w_{ j}:=
\begin{cases}
w_{n+j}, \quad   &j=1,\dots,n,\\
-w_{j-n}, \quad   & j=n+1,\dots,2n.
\end{cases}
$$
Then we have that
$$
W_j = \partial_{w_j} -\frac{1}{2}  \tilde w_{ j} \partial_t, \quad j=1,\dots,2n.
$$
\end{no}

\begin{no}\label{callA}
Let $f: U \subseteq  \mathbb{H}^n \to \mathbb{H}^n$, $U$ open, $f=(f^1,\dots,f^{2n+1}) \in \left [ C_\mathbb{H}^2(\mathbb{H}^n, \mathbb{R}) \right ]^{2n+1}$. Denote
$$
\mathcal{A}(j,f):= W_{j} f^{2n+1}  + \frac{1}{2}  \sum_{l=1}^{2n} \tilde w_{ l} (f)  W_{j} f^l   , \quad j=1,\dots,2n+1.
$$
\end{no}

\begin{lem}\label{composition}
Consider a map $f: U \subseteq  \mathbb{H}^n \to \mathbb{H}^n$, $U$ open, $f=(f^1,\dots,f^{2n+1}) \in \left [ C^1(\mathbb{H}^n, \mathbb{R}) \right ]^{2n+1}$ and $g: \mathbb{H}^n \to \mathbb{R}$, $g\in C^1(\mathbb{H}^n, \mathbb{R})$.  
Then
\begin{align}
\begin{aligned}
 W_j (g \circ f)  &= \sum_{l=1}^{2n}  ( W_l g)_f  W_{j} f^l    + (Tg)_f   \left (   W_{j} f^{2n+1}  +  \frac{1}{2}  \sum_{l=1}^{2n} \tilde w_{ l} (f)  W_{j} f^l   \right )  \\
& = (\nabla_{\mathbb{H}} g)_f \cdot (W_j f^1, \dots, W_j f^{2n})     + (Tg)_f \mathcal{A}(j,f) ,
\end{aligned}
\end{align}
for $j=1,\dots,2n+1$. 
In particular, if $n=1$,
$$
\begin{cases}
 X (g \circ f)  = ( X g)_f  X f^1 +  ( Y g)_f  X f^2  + (Tg)_f   \left (   X f^{3}  +  \frac{1}{2} \left (  f^2 Xf^1 - f^1 X f^2   \right ) \right )   ,\\
 Y (g \circ f) = ( X g)_f  Y f^1 +  ( Y g)_f  Y f^2    + (Tg)_f   \left (   Y f^{3} +  \frac{1}{2} \left (  f^2 Yf^1 - f^1 Y f^2   \right )  \right )   ,\\
 T (g \circ f) = ( X g)_f  T f^1 +  ( Y g)_f  T f^2    + (Tg)_f   \left (   T f^{3} +  \frac{1}{2} \left (  f^2 Tf^1 - f^1 T f^2   \right )  \right )   .
\end{cases}
$$
\end{lem}

\noindent
Note that, with our regularity hypotheses, we are not ready to use the equality $T=[W_j,W_{n+j}]$, as the double derivative is not well-defined yet. We will do this later when considering contact maps.

\begin{proof}[Proof of Lemma \ref{composition}]
\begin{align*}
 W_j (g \circ f)  &= \left (\partial_{w_j}- \frac{1}{2} \tilde w_{ j} \partial_t \right ) (g \circ f)\\
&= \sum_{l=1}^{2n+1} (\partial_{w_l} g)_f  \partial_{w_{j}} f^l  - \frac{1}{2} \tilde w_{ j}     \sum_{l=1}^{2n+1}  (\partial_{w_l} g)_f  \partial_{t} f^l \\
&= \sum_{l=1}^{2n+1} \left (  \partial_{w_{j}} f^l   - \frac{1}{2} \tilde w_{ j}    \partial_{t} f^l \right ) (\partial_{w_l} g)_f  = \sum_{l=1}^{2n+1} W_{j} f^l   (\partial_{w_l} g)_f\\
&= \sum_{l=1}^{2n} W_{j} f^l  \left (   W_l g+ \frac{1}{2} \tilde w_{ l}  T g \right )_f   +  W_{j} f^{2n+1}  (Tg)_f  \\
&= \sum_{l=1}^{2n}  ( W_l g)_f  W_{j} f^l    + (Tg)_f   \left (   W_{j} f^{2n+1}  +  \frac{1}{2}  \sum_{l=1}^{2n} \tilde w_{ l} (f)  W_{j} f^l   \right ) ,
\end{align*}
for $j=1,\dots,2n$. In the case of $j=2n+1,$
\begin{align*}
T (g \circ f)  &= \partial_{t} (g \circ f) = \sum_{l=1}^{2n+1} (\partial_{w_l} g)_f  \partial_{t} f^l  \\
&= \sum_{l=1}^{2n}\left (   W_l g+ \frac{1}{2} \tilde w_{ l}  T g \right )_f   T f^l   +   (Tg)_f   T f^{2n+1}     \\
&= \sum_{l=1}^{2n}  ( W_l g)_f  T f^l    + (Tg)_f   \left (   T f^{2n+1}  +  \frac{1}{2}  \sum_{l=1}^{2n} \tilde w_{ l} (f)  T f^l   \right )   .
\end{align*}
\end{proof}


\begin{prop}\label{pushforwardgeneral}
Let $f: U\subseteq \mathbb{H}^n \to \mathbb{H}^n$, $U$ open, $f=(f^1,\dots,f^{2n+1}) \in \left [ C^1(\mathbb{H}^n, \mathbb{R}) \right ]^{2n+1}$. 
Then
\begin{align}
\begin{aligned}
f_* W_j 
&= \sum_{l=1}^{2n+1}  \langle dw_l , f_* W_j \rangle W_l  \\
&= \sum_{l=1}^{2n} W_{j} f^l  W_l  +    \left (   W_{j} f^{2n+1}  +  \frac{1}{2}  \sum_{l=1}^{2n} \tilde w_{ l}(f)  W_{j} f^l   \right ) T \\
&=   \sum_{l=1}^{2n} W_{j} f^l  W_l  +   \mathcal{A}(j,f) T           ,
\end{aligned}
\end{align}
%
for $j=1,\dots,2n+1$.   In particular, if $n=1$,
 \begin{align}
  \begin{aligned}
f_* X=&
 \langle dx , f_* X \rangle X + \langle dy , f_* X \rangle Y + \langle \theta , f_* X \rangle    T \\
= &Xf^1 X +  Xf^2 Y +   \left (   X f^{3}  +  \frac{1}{2} \left (  f^2 Xf^1 - f^1 X f^2   \right ) \right )      T, 
  \end{aligned}
 \end{align}
 \begin{align}
  \begin{aligned}
f_* Y=&
 \langle dx , f_* Y \rangle X + \langle dy , f_* Y \rangle Y + \langle \theta , f_* Y \rangle    T \\
=&  Yf^1 X + Yf^2 Y +  \left (   Y f^{3}  +  \frac{1}{2} \left (  f^2 Yf^1 - f^1 Y f^2   \right ) \right )       T,
  \end{aligned}
 \end{align}
 \begin{align}
  \begin{aligned}
f_* T=&
 \langle dx , f_* T \rangle X + \langle dy , f_* T \rangle Y + \langle \theta , f_* T \rangle    T \\
=&  Tf^1 X + Tf^2 Y +  \left (   T f^{3}  +  \frac{1}{2} \left (  f^2 T f^1 - f^1 T f^2   \right ) \right )       T.
  \end{aligned}
 \end{align}
\end{prop}

\begin{proof}
The proof follows immediately from Lemma \ref{composition} remembering that $(f_* W_j )h=W_j(h \circ f)$. Otherwise one can make the computation directly as following exactly the same strategy as in Lemma \ref{composition}:
\begin{align*}
f_* W_j &=f_* \left (\partial_{w_j}- \frac{1}{2} \tilde w_{ j} \partial_t \right )\\
&= \sum_{l=1}^{2n+1} \partial_{w_{j}} f^l \partial_{w_l}  - \frac{1}{2} \tilde w_{ j}     \sum_{l=1}^{2n+1}   \partial_{t} f^l \partial_{w_l}\\
&= \sum_{l=1}^{2n+1} \left (  \partial_{w_{j}} f^l   - \frac{1}{2} \tilde w_{ j}    \partial_{t} f^l \right ) \partial_{w_l} = \sum_{l=1}^{2n+1} W_{j} f^l    \partial_{w_l}\\
&= \sum_{l=1}^{2n} W_{j} f^l  \left (   W_l + \frac{1}{2} \tilde w_{ l}  T  \right )   +  W_{j} f^{2n+1}  T  \\
&= \sum_{l=1}^{2n} W_{j} f^l  W_l  +    \left (   W_{j} f^{2n+1}  +  \frac{1}{2}  \sum_{l=1}^{2n} \tilde w_{ l}  W_{j} f^l   \right ) T  ,
\end{align*}
for $j=1,\dots,2n$. Similarly for $f_* T$.
\end{proof}



\subsection{Contact Maps}

\begin{rem}
Recall that the Definition \ref{contact} of contact map says that 
$$
\langle \theta , f_* W_j \rangle =0,  \quad j=1,\dots,2n.
$$
Proposition \ref{pushforwardgeneral} shows clearly that this is the same as asking
$$
\mathcal{A}(j,f)= W_{j} f^{2n+1}  + \frac{1}{2}  \sum_{l=1}^{2n} \tilde w_{ l} (f)  W_{j} f^l   =0,  \quad j=1,\dots,2n.
$$
\end{rem}

\begin{ex}\label{contactdebt}
In Example \ref{examplecontact} we promised to prove that the anisotropic dilation
$$\delta_r (w_1, \dots, w_{2n},w_{2n+1})=(rw_1, \dots, r w_{2n}, r^2w_{2n+1}  )$$
is a contact map. In other words, we have to show that $\mathcal{A}(j,\delta_r)(w)=0$ for $j=1,\dots,2n$. Indeed
\begin{align*}
\mathcal{A}(j,\delta_r) (w)&= W_{j} \delta_r^{2n+1} (w)  + \frac{1}{2}  \sum_{l=1}^{2n} [\tilde w_{ l} (\delta_r)](w)  W_{j} \delta_r^l   (w) \\
&= \left  (\partial w_j -\frac{1}{2} \tilde w_{ j} \partial_{w_{2n+1}} \right ) ( r^2w_{2n+1}  )  + \frac{1}{2}  \sum_{l=1}^{2n} \tilde w_{ l} (rw) \left  (\partial w_j -\frac{1}{2} \tilde w_{ j} \partial_{w_{2n+1}} \right ) (r w_l)                \\
&=  -\frac{1}{2} \tilde w_{ j}  r^2      + \frac{1}{2} r \tilde w_{ j} \cdot  r             =0.
\end{align*}
For completeness we show also that $\mathcal{A}(2n+1,\delta_r)(w) \neq 0$, indeed:
\begin{align*}
\mathcal{A}(2n+1,\delta_r) (w)&= W_{2n+1} \delta_r^{2n+1} (w)  + \frac{1}{2}  \sum_{l=1}^{2n} [\tilde w_{ l} (\delta_r)](w)  W_{2n+1} \delta_r^l   (w) \\
&= \left  ( \partial_{w_{2n+1}} \right ) ( r^2w_{2n+1}  )  + \frac{1}{2}  \sum_{l=1}^{2n} \tilde w_{ l} (rw) \left  (  \partial_{w_{2n+1}} \right ) (r w_l)      =r^2 \neq 0 .
\end{align*}
\end{ex}


\begin{no}\label{callL}
Let $f: U \subseteq  \mathbb{H}^n \to \mathbb{H}^n$, $U$ open, $f=(f^1,\dots,f^{2n+1}) \in \left [ C_\mathbb{H}^2(\mathbb{H}^n, \mathbb{R}) \right ]^{2n+1}$. Denote
$$
\lambda(j,f):= \sum_{l=1}^{n} \left (    W_j f^l  W_{n+j} f^{n+l} - W_{n+j} f^{l}W_{j} f^{n+l}    \right ), \quad j=1,\dots,n.
$$
\end{no}

\begin{lem}\label{T3=XY12}
Let $f: U \subseteq  \mathbb{H}^n \to \mathbb{H}^n$, $U$ open, be a contact map, $f=(f^1,\dots,f^{2n+1}) \in \left [ C_\mathbb{H}^2 (\mathbb{H}^n, \mathbb{R}) \right ]^{2n+1}$. Then, for $j=1,\dots,n$,
$$
\mathcal{A}(2n+1,f) = T f^{2n+1}  + \frac{1}{2}  \sum_{l=1}^{2n} \tilde w_{ l} (f)  T f^l   
= \sum_{l=1}^{n} \left (    W_j f^l  W_{n+j} f^{n+l} - W_{n+j} f^{l}W_{j} f^{n+l}    \right ) = \lambda (j,f).
$$
In particular, for $n=1$, one has $j=1$ and
$$
\mathcal{A}(3,f) = T f^{3}  + \frac{1}{2} \left (  f^2  T f^1  - f^1  T f^2 \right ) =  \left (    X f^1  Y f^2 - Y f^1 X f^2    \right ) = \lambda (1,f).
$$
\end{lem}

\begin{proof}
\begin{align*}
\mathcal{A}&(2n+1,f) = T f^{2n+1}  + \frac{1}{2}  \sum_{l=1}^{2n} \tilde w_{ l} (f)  T f^l  = \\
=& \left ( W_j   W_{n+j}  - W_{n+j} W_{j}   \right ) f^{2n+1}  + \frac{1}{2}  \sum_{l=1}^{2n} \tilde w_{ l} (f)  \left ( W_j   W_{n+j}  - W_{n+j} W_{j}   \right )  f^l         \\
=&    W_j   W_{n+j}  f^{2n+1}  + \frac{1}{2}  \sum_{l=1}^{2n} \tilde w_{ l} (f)   W_j   W_{n+j}    f^l   
 - W_{n+j} W_{j} f^{2n+1}  - \frac{1}{2}  \sum_{l=1}^{2n} \tilde w_{ l} (f)  W_{n+j} W_{j}  f^l     \\
=& W_j  \left (  W_{n+j}  f^{2n+1}  + \frac{1}{2}  \sum_{l=1}^{2n} \tilde w_{ l} (f)    W_{n+j}    f^l    \right )  - \frac{1}{2}  \sum_{l=1}^{2n}   W_j \tilde w_{ l} (f)      W_{n+j}    f^l \\
& - W_{n+j} \left (   W_{j} f^{2n+1}  + \frac{1}{2}  \sum_{l=1}^{2n} \tilde w_{ l} (f)  W_{j}  f^l     \right)  +  \frac{1}{2}  \sum_{l=1}^{2n} W_{n+j} \tilde w_{ l} (f)   W_{j}  f^l         \\
=& W_j \left ( \mathcal{A}(n+j,f) \right ) - \frac{1}{2}  \sum_{l=1}^{2n}   W_j \tilde w_{ l} (f)      W_{n+j}    f^l  - W_{n+j} \left ( \mathcal{A}(j,f) \right ) +  \frac{1}{2}  \sum_{l=1}^{2n} W_{n+j} \tilde w_{ l} (f)   W_{j}  f^l         \\
=&    - \frac{1}{2}  \sum_{l=1}^{2n}   W_j \tilde w_{ l} (f)      W_{n+j}    f^l   +\frac{1}{2}  \sum_{l=1}^{2n} W_{n+j} \tilde w_{ l} (f)   W_{j}  f^l    \\
=&  - \frac{1}{2}  \sum_{l=1}^{n}  \left (     W_j f^{n+l}      W_{n+j}    f^l  -   W_j  f^l     W_{n+j}    f^{n+l}  \right  ) + \frac{1}{2} \sum_{l=1}^{n}    \left ( W_{n+l} f^{n+l}   W_{j}  f^l  -    W_{n+j} f^l   W_{j}  f^{n+l}   \right )  \\
=& \sum_{l=1}^{n} \left (    W_j f^l  W_{n+j} f^{n+l} - W_{n+j} f^{l}W_{j} f^{n+l}    \right )= \lambda (j,f).
\end{align*}
\end{proof}

\noindent
The lemma shows clearly that $ \lambda(j,f)$ does not actually depend on  $j$, so from this point we can write $ \lambda(f)= \lambda(j,f)$.

\begin{lem}\label{composition_contact}
Consider a contact map $f: U \subseteq  \mathbb{H}^n \to \mathbb{H}^n$, $U$ open, $f=(f^1,\dots,f^{2n+1}) \in \left [ C^1(\mathbb{H}^n, \mathbb{R}) \right ]^{2n+1}$ and a map $g: \mathbb{H}^n \to \mathbb{R}$, $g\in C^1(\mathbb{H}^n, \mathbb{R})$. \\
Then, given the definition of contactness, it follows immediately from Lemma \ref{composition} that
\begin{align}
\begin{aligned}
 W_j (g \circ f)  &= \sum_{l=1}^{2n}  ( W_l g)_f  W_{j} f^l      \\
& = (\nabla_{\mathbb{H}} g)_f \cdot (W_j f^1, \dots, W_j f^{2n})      .
\end{aligned}
\end{align}
If $n=1$, they become
$$
\begin{cases}
 X (g \circ f)  = ( X g)_f  X f^1 +  ( Y g)_f  X f^2  ,\\
 Y (g \circ f) = ( X g)_f  Y f^1 +  ( Y g)_f  Y f^2    .
\end{cases}
$$
\end{lem}

\begin{lem}\label{nabla_comp1}
Consider a contact map $f: U \subseteq  \mathbb{H}^n \to \mathbb{H}^n$, $U$ open, $f=(f^1,\dots,f^{2n+1}) \in \left [ C^1(\mathbb{H}^n, \mathbb{R}) \right ]^{2n+1}$ and a map $g: \mathbb{H}^n \to \mathbb{R}$, $g \in C^1(\mathbb{H}^n, \mathbb{R})$. 
Then
\begin{equation}\label{tantinabla}
\nabla_{\mathbb{H}}  (g \circ f) =f_*^T (\nabla_{\mathbb{H}} g)_f .
\end{equation}
\end{lem}

\begin{proof}
Consider a horizontal vector $V$ and compute the scalar product of the Heisenberg gradient of the composition (which is horizontal by definition) against such vector $V$.
$$
\langle \nabla_{\mathbb{H}}  (g \circ f) ,V \rangle_H   =  \langle d_Q  (g \circ f) \vert V \rangle .
$$
Note that here we can substitute $d$ to $d_Q$ (and viceversa) because the last component of the differential does not play any role as the computation regards only horizontal objects. Formally we have
\begin{align*}
\langle d  \phi \vert V \rangle
=  &  \langle \sum_{j=1}^{n}  \left ( X_j \phi d x_j + Y_j \phi dy_j \right ) + T\phi \theta \vert \sum_{j=1}^{n}  \left ( V_j  X_j + V_{n+j} Y_j \right ) \rangle  \\
=&   \langle \sum_{j=1}^{n}  \left ( X_j \phi d x_j + Y_j \phi dy_j \right )  \vert \sum_{j=1}^{n}  \left ( V_j  X_j + V_{n+j} Y_j \right ) \rangle  \\
=&  \langle d_Q  \phi \vert V \rangle.
\end{align*}
We can also repeat this below for $f_*   V $ below, since  $f_*   V $ is still a horizontal vector field. Then
\begin{align*}   
\langle \nabla_{\mathbb{H}}  (g \circ f) ,V \rangle_H   =&  \langle d_Q  (g \circ f) \vert V \rangle =   \langle d  (g \circ f) \vert V \rangle =   
 (g \circ f)_* ( V ) \\
=& (  (g_*)_f  \circ f_* ) ( V ) =    (d g)_f  ( f_* V ) =    
 \langle    (d g)_f  \vert  f_*   V   \rangle \\
=&   \langle    (d_Q g)_f  \vert  f_*   V   \rangle  =      \langle    (\nabla_{\mathbb{H}}  g)_f  ,  f_*   V   \rangle_H =  \langle   f_{*}^T (\nabla_{\mathbb{H}}  g)_f, V \rangle_H .
\end{align*}   
Then, since $V$ is a general horizontal vector,
$$
\nabla_{\mathbb{H}}  (g \circ f) = f_*^T (\nabla_{\mathbb{H}}  g)_f .
$$
\end{proof}

\begin{rem}\label{nabla_comp2}
Equation \eqref{tantinabla} can be rewritten as
\begin{align*}   
 \nabla_{\mathbb{H}}  (g \circ f) = 
(X_1 g, \dots, X_n g, Y_1 g, \dots, Y_n g )_f \cdot
\left (
\begin{matrix}
X_1 f^1           & \ldots    & X_n f^1            & Y_1 f^1          & \ldots    & Y_n f^1 \\
\vdots              &               & \vdots               & \vdots              &              & \vdots\\
X_1 f^{2n}      &   \ldots  & X_n f^{2n}      & Y_1 f^{2n}      & \ldots   & Y_n f^{2n}
\end{matrix}
\right ).
\end{align*}   
\end{rem}

\noindent
Next we show the double derivative of the composition of two functions. By Lemma \ref{T3=XY12}, we will find our previous expression for $T$. 


\begin{lem}\label{doublederivativecomposition}
Consider a contact map $f: U \subseteq  \mathbb{H}^n \to \mathbb{H}^n$, $U$ open, $f=(f^1,\dots,f^{2n+1}) \in \left [ C^1(\mathbb{H}^n, \mathbb{R}) \right ]^{2n+1}$ and a map $g: \mathbb{H}^n \to \mathbb{R}$, $g\in  C_\mathbb{H}^2 (\mathbb{H}^n, \mathbb{R})$. For $j,i\in \{ 1,\dots, 2n\}$ one has:
\begin{align}
W_j W_i (g \circ f) =& \sum_{l=1}^{2n} \left [   
\left ( \sum_{h=1}^{2n} \left (W_h W_l g \right )_f W_j f^h  \right )  W_i f^l +   \left (W_l g \right )_f W_j W_i f^l
\right ], \quad \quad \quad \quad 
\end{align}
 \begin{align}
  \begin{aligned}\label{compoT}
\quad  T(g \circ f) =&   \sum_{l=1}^{2n}   \left (W_l g \right )_f T f^l   +
(Tg)_f    \sum_{l=1}^{n} \left (    W_h f^l  W_{n+h} f^{n+l} - W_{n+h} f^{l}W_{h} f^{n+l}    \right )\\
=&   \sum_{l=1}^{2n}   \left (W_l g \right )_f T f^l   +
(Tg)_f  \lambda (h,f) .
  \end{aligned}
 \end{align}
In case $n=1$, one gets
 \begin{align}\label{compT}
  \begin{aligned}
T(g \circ f) =&XgTf^1 + YgTf^2 +\lambda (1,f) Tg,\quad \quad \quad \quad \quad \quad \quad 
  \end{aligned}
 \end{align}
where $\lambda (1,f):=
Xf^1 Yf^2-Yf^1Xf^2$.
\end{lem}

\begin{proof}
Remember that
$$
W_j  (g \circ f) = \sum_{l=1}^{2n} \left  (W_l g \right )_f W_j  f^l.
$$
Then
\begin{align*}   
W_j W_i (g \circ f) =& \sum_{l=1}^{2n}   W_j   \left (  \left ( W_l g \circ f \right ) W_i   f^l \right )  \\
=& \sum_{l=1}^{2n}  \left [   W_j \left (W_l g \circ f \right ) W_i f^l  +   \left (W_l g \right )_f W_j W_i f^l
\right ]\\
=& \sum_{l=1}^{2n}  \left [    \sum_{h=1}^{2n} \left (
\left (W_h W_l g \right )_f W_j f^h \right ) W_i f^l +   \left (W_l g \right )_f W_j W_i f^l
\right ].
\end{align*}   
\begin{align*}  
T(g \circ f)
=& \left  (W_j W_{n+j} - W_{n+j}  W_{j}    \right   ) (g\circ f)\\
=& \sum_{l=1}^{2n} \Bigg [   
\left ( \sum_{h=1}^{2n} \left (W_h W_l g \right )_f W_j f^h  \right )  W_{n+j} f^l + \left (W_l g \right )_f W_j W_{n+j} f^l \\
&- \left ( \sum_{h=1}^{2n} \left (W_h W_l g \right )_f W_{n+j} f^h  \right )  W_j f^l -   \left (W_l g \right )_f W_{n+j} W_j f^l   \Bigg ],\\
=& \sum_{l=1}^{2n} \left [   
\sum_{h=1}^{2n} \left (W_h W_l g \right )_f  \left (  W_j f^h   W_{n+j} f^l  - W_{n+j} f^h  W_j f^l   \right )  
+   \left (W_l g \right )_f T f^l   \right ]\\
=& \sum_{l=1}^{2n}    \left (W_l g \right )_f T f^l  + \sum_{l,h=1}^{2n}     \left (W_h W_l g \right )_f  \left (  W_j f^h   W_{n+j} f^l  - W_{n+j} f^h  W_j f^l   \right )  .
\end{align*}   
Note that every time that $l=h$, the corresponding term is zero. Furthermore the term corresponding to a pair $(l,h)$ is the same as the one of the pair $(h,l)$ with a change of sign. Then we can rewrite as
\begin{align*}  
T(g \circ f)
=& \sum_{l=1}^{2n}    \left (W_l g \right )_f T f^l  + 
    \sum_{\mathclap{\substack{     l,h=1\\   l<h   }}}^{2n} 
    \left (W_h W_l g - W_l W_h g \right )_f  \left (  W_j f^h   W_{n+j} f^l  - W_{n+j} f^h  W_j f^l   \right )  .
\end{align*}   
Then notice that all the terms in the second sum are zero apart from when $h=n+l$. So we can finally write
\begin{align*}  
T(g \circ f)
=& \sum_{l=1}^{2n}    \left (W_l g \right )_f T f^l  + 
   (Tg)_f \sum_{ l=1}^{n} 
     \left (  W_j f^l   W_{n+j} f^{n+l}  - W_{n+j} f^l  W_j f^{n+l}   \right )  .
\end{align*} 
\end{proof}


\begin{prop}\label{pushforwardcontact}
Consider a contact map $f: U \subseteq  \mathbb{H}^n \to \mathbb{H}^n$, $U$ open, $f=(f^1,\dots,f^{2n+1}) \in \left [ C^1(\mathbb{H}^n, \mathbb{R}) \right ]^{2n+1}$. 
Then the pushforward matrix can be written as
\begin{align*}
f_*=&
\left (
\begin{matrix}
 W_1 f^1       & \ldots         &   W_{2n} f^1         &  W_{2n+1} f^1  \\
     \vdots             &                  & \vdots                     & \vdots   \\
 W_1 f^{2n}  &   \ldots       &  W_{2n} f^{2n}     &  W_{2n+1} f^{2n}   \\
                 0         &   \ldots      &   0                            &  \lambda (h,f) 
\end{matrix}
\right )
,
\end{align*}
with $h \in\{1,\dots,n\}$.
This is the same as writing
\begin{align}
\begin{aligned}\label{pushforwardWj}
f_* W_j &= \sum_{l=1}^{2n}  \langle dw_l , f_* W_j \rangle W_l  =   \sum_{l=1}^{2n} W_{j} f^l  W_l  , \quad j=1,\dots,2n, \quad  \ \ 
\end{aligned}
\end{align}
\begin{align}
\begin{aligned}\label{pushforwardT}
f_* T 
&= \sum_{l=1}^{2n+1}  \langle dw_l , f_* T \rangle W_l  \\
&= \sum_{l=1}^{2n} T f^l  W_l  +       \sum_{l=1}^{n} \left (    W_h f^l  W_{n+h} f^{n+l} - W_{n+h} f^{l}W_{h} f^{n+l}    \right )   T \\
&=   \sum_{l=1}^{2n} T f^l  W_l  +   \lambda (h,f) T              .
\end{aligned}
\end{align}
In particular, if $n=1$, we have:
\begin{align*}
f_*=
\begin{pmatrix}
Xf^1 &Yf^1 & Tf^1 \\
Xf^2 &  Yf^2  & Tf^2 \\
0 & 0 &   \lambda (1,f)   
\end{pmatrix},
\quad \text{ i.e.,} \quad
\begin{cases}
f_* X = Xf^1 X +  Xf^2 Y,   \\
f_* Y =   Yf^1 X + Yf^2 Y, \\
f_* T =    Tf^1 X + Tf^2 Y +\lambda (1,f)  T.
\end{cases}
\end{align*}
\end{prop}

\begin{proof}
The proof of equation \eqref{pushforwardWj} comes immediately from the definition of contactness. To prove equation \eqref{pushforwardT}, we can use Proposition \ref{pushforwardgeneral} together with Lemma \ref{T3=XY12}.
Instead, if we want to show the proof directly, we just need to work exactly as in the proof of equation \eqref{compoT}.
\end{proof}

\noindent
In the same way, since $\langle f^* \omega\vert v \rangle =\langle \omega \vert f_*  v \rangle$, we have an equivalent proposition for the pullback.

\begin{prop}
Consider a contact map $f: U \subseteq  \mathbb{H}^n \to \mathbb{H}^n$, $U$ open, $f=(f^1,\dots,f^{2n+1}) \in \left [ C^1(\mathbb{H}^n, \mathbb{R}) \right ]^{2n+1}$.  
Then
\begin{align*}
f^*=(f_*)^T=&
\left (
\begin{matrix}
 W_1 f^1             & \ldots          &    W_1 f^{2n}         &  0   \\
     \vdots             &                    & \vdots                     & \vdots   \\
 W_{2n} f^1        &   \ldots       &  W_{2n} f^{2n}     &  0   \\
W_{2n+1} f^1      &   \ldots        &   W_{2n+1} f^{2n}        &  \lambda (h,f) 
\end{matrix}
\right ),
\end{align*}
with $h \in\{1,\dots,n\}$.
This is the same as writing
\begin{align}
\begin{aligned}
f^* dw_j &= \sum_{l=1}^{2n+1}  \langle f^* dw_j ,  W_l \rangle dw_l =   \sum_{l=1}^{2n+1} W_l f^j  dw_l    , \quad j=1,\dots,2n     ,
\end{aligned}
\end{align}
\begin{align}
\begin{aligned}
f^* \theta = \sum_{l=1}^{2n+1}  \langle  f^* \theta , W_l \rangle dw_l = \langle  f^* \theta , T \rangle \theta =  \lambda (h,f) \theta           .
\end{aligned}
\end{align}
In particular, if $n=1$ we have:
\begin{align*}
f^*=(f_*)^T=
\begin{pmatrix}
Xf^1 & Xf^2 &  0 \\
Yf^1 &  Yf^2  & 0\\
 Tf^1  & Tf^2 & \lambda (1,f)
\end{pmatrix},
\quad \text{ i.e.,} \quad
\begin{cases}
f^* dx=  Xf^1 dx + Y f^1 dy + T f^1 \theta, \\
f^* dy = Xf^2 dx + Y f^2 dy + T f^2 \theta,\\
f^* \theta = \lambda (1,f) \theta.
\end{cases}
\end{align*}
\end{prop}


\subsection{Contact diffeomorphisms}

\begin{lem}\label{spamT}
Let $f: U \subseteq  \mathbb{H}^n \to \mathbb{H}^n$, $U$ open, be a contact diffeomorphism such that $f=(f^1,\dots,f^{2n+1}) \in \left [ C^1(\mathbb{H}^n, \mathbb{R}) \right ]^{2n+1}$. Then
$$
f_* T \in \spn \{ T \},
$$
or equivalently, given equation \eqref{pushforwardT},
$$
\langle dw_j , f_* T \rangle = T f^j=0,
$$
for all  $j=1,\dots,2n$. Furthermore, for a diffeomorphism $f$ the matrix $(f_*)$ must be invertible, so we get that $ \lambda (h,f) =Tf^{2n+1} \neq 0$, with $h=1,\dots,n$.
\end{lem}

\begin{proof}
Since $f$ is a diffeomorphism, it admits an inverse mapping and therefore the matrix $(f_*)$ must be invertible. In particular, this means that $ \lambda (h,f)  \neq 0$, with $h=1,\dots,n$. \\
Again since $f$ is a diffeomorphism, $f_* : \mathfrak{h} \mapsto  \mathfrak{h}$ is a linear isomorphism. By contactness, we have that
$$
{f_*}_{\vert_{ \mathfrak{h}_1 }} : \mathfrak{h}_1 \to   \mathfrak{h}_1  ,
$$
which is still an isomorphism. Hence, since the lie algebra divides as $\mathfrak{h}=\mathfrak{h}_1 \oplus \mathfrak{h}_2$, also
$$
{f_*}_{\vert_{ \mathfrak{h}_2 }} : \mathfrak{h}_2 \to   \mathfrak{h}_2  
$$
is a linear isomorphism and the dimensions of domain and codomain must coincide and be $1$. Then
$$
f_* T \in \spn \{ T \}= \mathfrak{h}_2.
$$
\end{proof}

\begin{obs}\label{cmdiff}
Let $f: U \subseteq  \mathbb{H}^n \to \mathbb{H}^n$, $U$ open, be a contact diffeomorphism such that $f=(f^1,\dots,f^{2n+1}) \in \left [ C_\mathbb{H}^2 (\mathbb{H}^n, \mathbb{R}) \right ]^{2n+1}$. Lemma \ref{T3=XY12} then becomes
$$
T f^{2n+1}  = \sum_{l=1}^{n} \left (    W_j f^l  W_{n+j} f^{n+l} - W_{n+j} f^{l}W_{j} f^{n+l}    \right ) = \lambda (j,f)
$$
$j=1,\dots,n.$
Then Lemma \ref{spamT} says that
$$
\begin{cases}
\langle \theta , f_* W_j \rangle =0,\\
\langle dw_j , f_* T \rangle =0,\\
j=1,\dots,2n,\\
\end{cases}
\quad \text{ i.e.,} \quad
\begin{cases}
\mathcal{A}(j,f)= W_{j} f^{2n+1}  + \frac{1}{2}  \sum_{l=1}^{2n} \tilde w_{ l} (f)  W_{j} f^l  =0,\\
T f^j=0,\\
j=1,\dots,2n,\\
\end{cases}
$$
\end{obs}

\begin{rem}
Let $f: U \subseteq  \mathbb{H}^n \to \mathbb{H}^n$, $U$ open, be a contact diffeomorphism such that $f=(f^1,\dots,f^{2n+1}) \in \left [ C_\mathbb{H}^2 (\mathbb{H}^n, \mathbb{R}) \right ]^{2n+1}$. Then equation \eqref{compoT} becomes 
$$
T(g \circ f)   =(Tg)_f  T f^{2n+1}.
$$
\end{rem}

\begin{prop}
Let $f: U \subseteq  \mathbb{H}^n \to \mathbb{H}^n$, $U$ open, be a contact diffeomorphism such that $f=(f^1,\dots,f^{2n+1}) \in \left [ C_\mathbb{H}^2 (\mathbb{H}^n, \mathbb{R}) \right ]^{2n+1}$. Then
\begin{align*}
f_*=&
\left (
\begin{matrix}
 W_1 f^1       & \ldots         &   W_{2n} f^1         &  0  \\
     \vdots             &                  & \vdots                     & \vdots   \\
 W_1 f^{2n}  &   \ldots       &  W_{2n} f^{2n}     &  0   \\
                 0         &   \ldots      &   0                            &  T f^{2n+1} 
\end{matrix}
\right ).
\end{align*}
This is the same as writing
\begin{align}
\begin{aligned}
f_* W_j &= \sum_{l=1}^{2n}  \langle dw_l , f_* W_j \rangle W_l  =   \sum_{l=1}^{2n} W_{j} f^l  W_l  , \quad j=1,\dots,2n ,
\end{aligned}
\end{align}
\begin{align}
\begin{aligned}
f_* T 
&=  \langle d \theta , f_* T \rangle T \\
&=   \sum_{l=1}^{n} \left (    W_h f^l  W_{n+h} f^{n+l} - W_{n+h} f^{l}W_{h} f^{n+l}    \right )   T  \quad \quad \quad \quad    \\
&=   T f^{2n+1}    T .
\end{aligned}
\end{align}
In particular, if $n=1$ we have:
\begin{align*}
f_*=
\begin{pmatrix}
Xf^1 &Yf^1 & 0\\
Xf^2 &  Yf^2  & 0 \\
0 & 0 &   Tf^{3} 
\end{pmatrix},
\quad \text{ i.e.,} \quad
\begin{cases}
f_* X = Xf^1 X +  Xf^2 Y,\\
f_* Y =  Yf^1 X + Yf^2 Y,\\
f_* T =   Tf^3 T ,
\end{cases}
\end{align*}
with $Tf^{3} =Xf^1Yf^2-Xf^2Yf^1$.
\end{prop}

\noindent
In the same way again,

\begin{prop}
Let $f: U \subseteq  \mathbb{H}^n \to \mathbb{H}^n$, $U$ open, be a contact diffeomorphism such that $f=(f^1,\dots,f^{2n+1}) \in \left [ C_\mathbb{H}^2 (\mathbb{H}^n, \mathbb{R}) \right ]^{2n+1}$. Then
\begin{align*}
f^*=(f_*)^T=&
\left (
\begin{matrix}
 W_1 f^1             & \ldots          &    W_1 f^{2n}         &  0   \\
     \vdots             &                    & \vdots                     & \vdots   \\
 W_{2n} f^1        &   \ldots       &  W_{2n} f^{2n}     &  0   \\
0                           &   \ldots        &   0                          &  T^{2n+1}
\end{matrix}
\right ).
\end{align*}
This is the same as writing
\begin{align}
\begin{aligned}
f^* dw_j &= \sum_{l=1}^{2n+1}  \langle f^* dw_j ,  W_l \rangle dw_l =   \sum_{l=1}^{2n+1} W_l f^j  dw_l    , \quad j=1,\dots,2n     ,
\end{aligned}
\end{align}
\begin{align}
\begin{aligned}
f^* \theta = \sum_{l=1}^{2n+1}  \langle  f^* \theta , W_l \rangle dw_l = \langle  f^* \theta , T \rangle \theta =  \lambda (h,f) \theta           .
\end{aligned}
\end{align}
In particular, if $n=1$ we have:
\begin{align*}
f^*=(f_*)^T=
\begin{pmatrix}
Xf^1 & Xf^2 &  0 \\
Yf^1 &  Yf^2  & 0\\
 0  & 0 &    Tf^3   
\end{pmatrix},
\quad \text{ i.e.,} \quad
\begin{cases}
f^* dx=  Xf^1 dx + Y f^1 dy,\\
f^* dy = Xf^2 dx + Y f^2 dy,\\
f^* \theta =  T f^3  \theta ,
\end{cases}
\end{align*}
with $ T f^3   =   Xf^1Yf^2-Xf^2Yf^1 $.
\end{prop}


\subsection{Higher Order}

\begin{lem}\label{XCYC}
Let $f: U \subseteq  \mathbb{H}^n \to \mathbb{H}^n$, $U$ open, be a contact map such that $f=(f^1,\dots,f^{2n+1}) \in \left [ C_\mathbb{H}^2 (\mathbb{H}^n, \mathbb{R}) \right ]^{2n+1}$. Then
\begin{align*}
W_j (\lambda  (f))=&   \sum_{l=1}^{n}   \left (      W_j f^{n+l}  T f^l  -   Tf^{n+l} W_j f^{l}    \right )\\
=&   \sum_{l=1}^{2n}      W_j(   \tilde w_{ l} (f)  )  T f^l   .
\end{align*}
In the case $n=1$ we get
$$
X(\lambda (f))=  Xf^1 Tf^2 -  Tf^1  Xf^2 \quad \text{ and } \quad  Y(\lambda (f))=  Yf^1 Tf^2 -  Tf^1  Yf^2 .
$$
\end{lem}

\begin{proof}
Observe first then
$$
T(f^l W_j f^{n+l})=Tf^l W_j f^{n+l} + f^l TW_j f^{n+l} = Tf^l W_j f^{n+l} + f^l W_j T f^{n+l},
$$
which means
$$
-  f^l W_j T f^{n+l} =- T(f^l W_j f^{n+l}) +    Tf^l W_j f^{n+l}    .
$$
Likewise we have
$$
 f^{n+l} W_j T f^{l} = T(f^{n+l} W_j f^{l})-   Tf^{n+l} W_j f^{l}    .
$$
Then
\begin{align*}
W_j& (\lambda  (h,f))=\\
=&   W_j \left ( T f^{2n+1}  + \frac{1}{2}  \sum_{l=1}^{2n} \tilde w_{ l} (f)  T f^l   \right )\\
=&    W_j  T f^{2n+1}  + \frac{1}{2}  \sum_{l=1}^{2n}   \left (      W_j( \tilde w_{ l} (f))  T f^l   +
\tilde w_{ l} (f)    W_j T f^l  
\right ) \\
=&    W_j  T f^{2n+1}  + \frac{1}{2}  \sum_{l=1}^{n}   \left ( 
     W_j( \tilde w_{ l} (f))  T f^l   +\tilde w_{ l} (f)    W_j T f^l  
+      W_j( \tilde w_{ {n+l}} (f))  T f^{n+l}   +\tilde w_{ {n+l}} (f)    W_j T f^{n+l}
\right ) \\
=&    W_j  T f^{2n+1}  + \frac{1}{2}  \sum_{l=1}^{n}   \left ( 
     W_j f^{n+l}  T f^l    + f^{n+l}    W_j T f^l  
-     W_j  f^l  T f^{n+l}   -  f^l    W_j T f^{n+l}    
\right ) \\
=&   T W_j f^{2n+1}  + \frac{1}{2}  \sum_{l=1}^{n}   \left ( 
      T(f^{n+l} W_j f^{l})  - T(f^l W_j f^{n+l})  
\right )
+  \sum_{l=1}^{n}   \left (      W_j f^{n+l}  T f^l  -   Tf^{n+l} W_j f^{l}    \right ) \\
=&   \sum_{l=1}^{n}   \left (      W_j f^{n+l}  T f^l  -   Tf^{n+l} W_j f^{l}    \right )\\
=&   \sum_{l=1}^{2n}      W_j(   \tilde w_{ l} (f)  )  T f^l   ,
\end{align*}
where we use  $ T W_j f^{2n+1}  + \frac{1}{2}  \sum_{l=1}^{n}   \left (   T(f^{n+l} W_j f^{l})  - T(f^l W_j f^{n+l})  \right )  =    T (\mathcal{A}(j,f))=0 $.
\end{proof}

\noindent
From this point to the end of the chapter we will consider $n=1$. The choice is not only of notational convenience, as in the other cases the computation becomes more problematic and the results do not allow an easy interpretation.

\begin{prop}\label{highdimensioncontact}
Let $f: U \subseteq  \mathbb{H}^1 \to \mathbb{H}^1$, $U$ open, be a contact map such that $f=(f^1,f^2,f^{3}) \in \left [ C_\mathbb{H}^2 (\mathbb{H}^1, \mathbb{R}) \right ]^{3}$. Then the pushforward matrix in the basis $\{ X \wedge Y, X \wedge T , Y \wedge T \}$ is 
\begin{align*}
f_*=
\begin{pmatrix}
\lambda (f) & X(\lambda (f) ) & Y(\lambda (f) )  \\
0   &  Xf^1  \lambda (f)  &  Yf^1  \lambda (f) \\
 0  & Xf^2  \lambda (f) &    Yf^2  \lambda (f)   
\end{pmatrix},
\quad \text{ i.e.,} \quad
\end{align*}
\begin{align}
f_* (X \wedge Y) &=\lambda (f) X\wedge Y,\\
f_* (X \wedge T) &=X(\lambda (f) ) X\wedge Y  +Xf^1  \lambda (f) X\wedge T+  Xf^2 \lambda (f) Y\wedge T,\\
f_* (Y \wedge T)&= Y(\lambda (f)) X\wedge Y  + Yf^1 \lambda (f) X\wedge T+  Yf^2 \lambda (f) Y\wedge T,
\end{align}
and
\begin{align}
f_* (X \wedge Y \wedge T) =&\lambda (f)^2  X \wedge  Y \wedge   T.
\end{align}
Likewise, the pullback is:
\begin{align*}
f^*= (f_*)^T=
\begin{pmatrix}
\lambda (f) & 0 & 0  \\
 X(\lambda (f) )   &  Xf^1  \lambda (f)  &  Xf^2  \lambda (f) \\
 Y(\lambda (f) )  & Yf^1  \lambda (f) &    Yf^2  \lambda (f)   
\end{pmatrix},
\quad \text{ i.e.,} \quad
\end{align*}
\begin{align}
f^* (dx \wedge dy)&= \lambda (f) dx \wedge dy + X(\lambda (f)) dx \wedge \theta+ Y(\lambda (f)) dy \wedge \theta,\\
f^* (dx \wedge \theta) &=Xf^1 \lambda (f) dx \wedge \theta +  Y f^1\lambda (f) dy \wedge \theta,\label{2dimx}\\
f^* (dy \wedge \theta) &= Xf^2 \lambda (f) dx \wedge \theta +  Y f^2 \lambda (f) dy \wedge \theta,\label{2dimy}
\end{align}
and
\begin{align}
f^* (dx \wedge dy \wedge \theta ) =&\lambda (f)^2 dx \wedge dy \wedge \theta.\label{3dim}
\end{align}
\end{prop}

\begin{proof}
\begin{align*}
f_* (X \wedge Y)=& f_* X \wedge f_* Y =( Xf^1  Yf^2 -  Yf^1  Xf^2) X\wedge Y =\lambda (1,f) X\wedge Y.\\
&\\
f_* (X \wedge T)=&f_* X \wedge  f_* T = ( Xf^1 Tf^2 -  Tf^1  Xf^2 ) X\wedge Y  \\
&+Xf^1  \left (  Tf^3 + \frac{1}{2}f^2 Tf^1 - \frac{1}{2}f^1 Tf^2    \right ) X\wedge T\\
&+ Xf^2 \left (  Tf^3 + \frac{1}{2}f^2 Tf^1 - \frac{1}{2}f^1 Tf^2    \right ) Y\wedge T\\
=& X(\lambda (1,f))  X\wedge Y  +Xf^1  \lambda (1,f) X\wedge T+  Xf^2 \lambda (1,f) Y\wedge T.\\
&\\
f_* (Y \wedge T)=& f_* Y \wedge  f_* T = ( Yf^1 Tf^2 -  Tf^1  Yf^2 ) X\wedge Y  \\
&+Yf^1  \left (  Tf^3 + \frac{1}{2}f^2 Tf^1 - \frac{1}{2}f^1 Tf^2    \right ) X\wedge T\\
&+ Yf^2 \left (  Tf^3 + \frac{1}{2}f^2 Tf^1 - \frac{1}{2}f^1 Tf^2    \right ) Y\wedge T\\
=&Y(\lambda (1,f))  X\wedge Y  + Yf^1 \lambda (1,f) X\wedge T+  Yf^2 \lambda (1,f) Y\wedge T.\\
&\\
f_* (X \wedge Y \wedge T)=&f_* X \wedge f_* Y \wedge  f_* T \\
=&( Xf^1  Yf^2 -  Yf^1  Xf^2) \left (  Tf^3 + \frac{1}{2}f^2 Tf^1 - \frac{1}{2}f^1 Tf^2    \right )  X \wedge  Y \wedge   T\\ 
=& \lambda (1,f)^2 X \wedge  Y \wedge   T.\\
&\\
f^* (dx \wedge dy)=&f^* dx \wedge f^*  dy = ( Xf^1  Yf^2 -  Yf^1  Xf^2) dx \wedge dy \\
&+ ( Xf^1  Tf^2 -  Tf^1  Xf^2 )dx \wedge \theta+ ( Yf^1  Tf^2 -  Tf^1  Yf^2 )dy \wedge \theta\\
=& \lambda (1,f) dx \wedge dy + X(\lambda (1,f)) dx \wedge \theta+ Y(\lambda (1,f)) dy \wedge \theta.\\
&\\
f^* (dx \wedge \theta)=& Xf^1 \left ( T f^3   - \frac{1}{2}f^1  Tf^2 + \frac{1}{2}f^2  Tf^1   \right ) dx \wedge \theta \\
& +Y f^1 \left ( T f^3   - \frac{1}{2}f^1  Tf^2 + \frac{1}{2}f^2  Tf^1   \right ) dy \wedge \theta\\
=&Xf^1 \lambda (1,f) dx \wedge \theta +  Y f^1\lambda (1,f) dy \wedge \theta.\\
&\\
f^* (dy \wedge \theta)= &Xf^2 \left ( T f^3   - \frac{1}{2}f^1  Tf^2 + \frac{1}{2}f^2  Tf^1   \right ) dx \wedge \theta \\
&+ Y f^2 \left ( T f^3   - \frac{1}{2}f^1  Tf^2 + \frac{1}{2}f^2  Tf^1   \right ) dy \wedge \theta\\
=& Xf^2 \lambda (1,f) dx \wedge \theta +  Y f^2 \lambda (1,f) dy \wedge \theta.\\
&\\
f^* (dx \wedge dy \wedge \theta )=&  ( Xf^1  Yf^2 -  Yf^1  Xf^2) \left ( T f^3   - \frac{1}{2}f^1  Tf^2 + \frac{1}{2}f^2  Tf^1   \right ) dx \wedge dy \wedge \theta\\
=& \lambda (1,f)^2 dx \wedge dy \wedge \theta.
\end{align*}
\end{proof}


\begin{obs}\label{simply}
Let $f: U \subseteq  \mathbb{H}^1 \to \mathbb{H}^1$, $U$ open, be a contact diffeomorphism such that $f=(f^1,f^2,f^{3}) \in \left [ C_\mathbb{H}^2 (\mathbb{H}^1, \mathbb{R}) \right ]^{3}$. Recall
$$
\lambda (f)=Tf^3=Xf^1 Yf^2- Xf^2 Yf^1.
$$
Given the conditions in Observation \ref{cmdiff}, Lemma \ref{XCYC} then becomes
\begin{align*}
X(\lambda (f))=    0 \quad \text{and} \quad   Y(\lambda (f))=0 .  
\end{align*}
\end{obs}

\begin{prop}
Let $f: U \subseteq  \mathbb{H}^1 \to \mathbb{H}^1$, $U$ open, be a contact diffeomorphism such that $f=(f^1,f^2,f^{3}) \in \left [ C_\mathbb{H}^2 (\mathbb{H}^1, \mathbb{R}) \right ]^{3}$. In this case, Proposition \ref{highdimensioncontact} becomes
\begin{align*}
f_*=
\begin{pmatrix}
T f^3 & 0 & 0  \\
0   &  Xf^1 T f^3  &  Yf^1 T f^3 \\
 0  & Xf^2  T f^3 &    Yf^2  T f^3   
\end{pmatrix},
\quad \text{ i.e.,} \quad
\end{align*}
$$
\begin{cases}
f_* (X \wedge Y) =T f^3  X\wedge Y ,\\
f_* (X \wedge T)= Xf^1    Tf^3    X\wedge T+ Xf^2   Tf^3  Y\wedge T,\\
f_* (Y \wedge T) = Yf^1   Tf^3  X\wedge T+ Yf^2   Tf^3  Y\wedge T,
\end{cases}
$$
and
\begin{align}
f_* (X \wedge Y \wedge T) = (Tf^3)^2   X \wedge  Y \wedge   T .
\end{align}
Likewise, the pullback is:
\begin{align*}
f^*= (f_*)^T=
\begin{pmatrix}
T f^3& 0 & 0  \\
0  &  Xf^1 T f^3  &  Xf^2  T f^3 \\
0 & Yf^1  T f^3 &    Yf^2  T f^3  
\end{pmatrix},
\quad \text{ i.e.,} \quad
\end{align*}
$$
\begin{cases}
f^* (dx \wedge dy) = Tf^3  dx \wedge dy,\\
f^* (dx \wedge \theta)= Xf^1  T f^3   dx \wedge \theta +  Y f^1  T f^3  dy \wedge \theta,\\
f^* (dy \wedge \theta)= Xf^2  T f^3  dx \wedge \theta +  Y f^2  T f^3  dy \wedge \theta,
\end{cases}
$$
and
\begin{align}
f^* (dx \wedge dy \wedge \theta )=  (Tf^3)^2 dx \wedge dy \wedge \theta.
\end{align}
\end{prop}

\begin{rem}
Note that so far we never considered the Rumin cohomology but only the definitions of pushforward and pullback on the whole algebra and the notions of contactness and diffeomorphicity. In the Rumin cohomology, as we described in the previous chapter, some differential forms (and subsequently their dual vectors) are either zero in the equivalence class or do not appear at all. For $n=1$ these are: $T$, $X\wedge Y$,$\theta$ and $dx\wedge dy$.
\end{rem}

\chapter{Orientability} \label{orient4}

The objective of this chapter is to give a definition of orientability in the Heisenberg sense and to study which surfaces can be called orientable in such sense. 
One reason to study orientability is that it is possible to define it using the cohomology of differential forms. As we have seen so far, the Rumin cohomology is different from the (usual) Riemannian one, so one can naturally ask how will the orientability change in this prospective. 
There are, however, other reasons: orientability plays an important role in the theory of currents. Currents are linear functionals on the space of differential forms and can be identified, with some hypotheses, with regular surfaces. This creates an important bridge between analysis and geometry. Such surfaces are usually orientable but not always: in Riemannian geometry there exists a notion of currents (currents mod $2$) made for surfaces that are not necessarily orientable (see, for instance, \cite{MORGAN2}). 
Also in the Heisenberg group regular orientable surfaces can be associated to currents, but this happens with different hypothesis than the Riemannian case and it is not clear whether it is meaningful to study currents mod $2$ in this setting. We will show that also in the Heisenberg group there exists non-orientable surfaces with enough regularity to be associated to currents. This means that, also in the Heisenberg group, it is a meaningful task to study currents mod $2$.\\
First we will state the definitions of $\mathbb{H}$-regularity for low dimensional and low codimensional surfaces. The main reference in this section is \cite{FSSC}. Next we will show that there exist surfaces with the property of being both $\mathbb{H}$-regular and non-orientable in the Euclidean sense. Then we will introduce the notion of Heisenberg orientability ($\mathbb{H}$-orientability) and prove that, under left translations and anisotropic dilations, $\mathbb{H}$-regularity is invariant for $1$-codimensional surfaces and $\mathbb{H}$-orientability is invariant for $\mathbb{H}$-regular $1$-codimensional surfaces. 
Finally, we show how the two notions of orientability are related, concluding that non-$\mathbb{H}$-orientable $\mathbb{H}$-regular surfaces exist, at least when $n=1$.\\
The idea is to consider a ``Möbius strip"-like surface as target. We will restrict ourselves and take $S$, a $1$-codimensional surface in $\mathbb{H}^n$, to be $C^1$-Euclidean, with $C(S) = \varnothing$, where $C(S)$ is defined as the set of characteristic points of $S$, meaning points where $S$ fails to respect the $\mathbb{H}$-regularity condition. These hypotheses imply that $S$ is $\mathbb{H}$-regular. All these terms will be made precise below. We will then show that it is possible to build such $S$ to be a  ``Möbius strip"-like surface.

\section{$\mathbb{H}$-regularity in $\mathbb{H}^n$}

We state here the definitions of $\mathbb{H}$-regularity for low dimension and low codimension. The terminology comes from Subsection \ref{lefthor}. The main reference in this section is \cite{FSSC}.

\begin{defin}[See 3.1 in \cite{FSSC}]
Let $1\leq k \leq n$. A subset $S \subseteq \mathbb{H}^n$ is a $\mathbb{H}$-\emph{regular} $k$-\emph{dimensional surface} if for all $p \in S$   there exists  a neighbourhood $ U \in \mathcal{U}_p$, an open set $ V \subseteq \mathbb{R}^k$ and a function $\varphi : V \to U$, $ \varphi \in [C_{\mathbb{H}}^1(V,U)]^{2n+1} $  injective with $d_H \varphi $ injective such that $ S \cap U = \varphi (V)$.
%
%
\end{defin}

\begin{defin}[See 3.2 in \cite{FSSC}]\label{Hreg}
Let $1\leq k \leq n$. A subset $S \subseteq \mathbb{H}^n$ is a $\mathbb{H}$-\emph{regular} $k$-\emph{codimensional surface} if for all $ p \in S $ there exists a neighbourhood $ U \in \mathcal{U}_p$ and a function  $ f : U \to \mathbb{R}^k$, $ f \in [C_{\mathbb{H}}^1(U,\mathbb{R}^k)]^k$, such that  $  {\nabla_\mathbb{H} f_1} \wedge \dots \wedge {\nabla_\mathbb{H} f_k}   \neq 0 $ on   $ U $ and  $  S \cap U = \{ f=0 \} $.
%
%
\end{defin}

\noindent
We will almost always work with the codimensional definition, that is, the surfaces of higher dimension.\\
If a surface is $\mathbb{H}$-regular, it is natural to associate to it, locally, a normal and a tanget vector:

\begin{defin}
Consider a $\mathbb{H}$-regular $k$-codimensional surface $S$ and $p \in S$. Then the \emph{(horizontal) normal $k$-vector field} $n_{\mathbb{H},p}$ is defined as
$$
n_{\mathbb{H},p} := \frac{ {\nabla_\mathbb{H} f_1}_p \wedge \dots \wedge {\nabla_\mathbb{H} f_k}_p   }{ \vert {\nabla_\mathbb{H} f_1}_p \wedge \dots \wedge {\nabla_\mathbb{H} f_k}_p \vert }    \in {\prescript{}{}\bigwedge}_{k,p} \mathfrak{h}_1 .
$$
In a natural way, the \emph{tangent $(2n+1-k)$-vector field} $t_{\mathbb{H},p}$ is defined as its dual:
$$
t_{\mathbb{H},p} := * n_{\mathbb{H},p} \in {\prescript{}{}\bigwedge}_{2n+1-k,p} \mathfrak{h},
$$
where the Hodge operator $*$ appears in Definition \ref{hodge}.
\end{defin}

\begin{no}
We defined the regularity of a surface in the Heisenberg sense. 
When considering the regularity of a surface in the Euclidean sense, we will say $C^k$\emph{-regular in the Euclidean sense} or $C^k$\emph{-Euclidean} for short.
\end{no}



\begin{lem}\label{quickcomputation}
Let $S$ be a $C^1$-Euclidean surface in $\mathbb{R}^{2n+1} = \mathbb{H}^n$. Then $\dim_{\mathcal{H}_{cc}} S=Q-1=2n+1$ if and only if $\dim_{\mathcal{H}_{E}} S =2n$.
\end{lem}

\noindent
Hence, from now on we may consider a $C^1$-Euclidean surface $S$ and ask it to be $1$-codimensional without further specifications.

\begin{proof}[Proof of Lemma \ref{quickcomputation}]
Consider first $\dim_{\mathcal{H}_{cc}} S=2n+1$. The Hausdorff dimension of $S$ with respect of the Euclidean distance is equal to the dimension of the tangent plane, which is well defined everywhere by hypothesis; hence such dimension is an integer. By theorems $2.4$-$2.6$ (and fig. 2) in \cite{BTW} or by \cite{BRSC} (for $\mathbb{H}^1$ only), one has that:
$$
\max \{  k, 2k-2n \} \leq \dim_{\mathcal{H}_{cc}} S \leq \min \{ 2k, k+1 \} ,
$$
with $k:=\dim_{\mathcal{H}_{E}} S$.\\
Let's first we check the minimum: if $2k\leq k+1$, then $k\leq 1$ and
$$
\max \{  k, 2k-2n \} \leq 2n+1 \leq  2k \leq 2 ,
$$
which is impossible because $2n+1\geq 3$. So $k+1\leq 2k$ and the minimum is decided. For the maximum, if $k \leq  2k-2n$, then $k \geq 2n$ and so either $k=2n$ or $k=2n+1$. But then
$$
 2k-2n  \leq 2n+1 ( \leq k+1) ,
$$
which says that $2k \leq 2n+1$; this is verified only if $k=2n$. The second case of the maximum is $2k-2n\leq k$, meaning $k\leq 2n$. In this case
$$
k  \leq 2n+1 \leq k+1 ,
$$
which again says that the only possibility is $k=2n$.\\
On the other hand, if we consider a $C^1$-Euclidean surface $S$ in $\mathbb{R}^{2n+1}$ with $\dim_{\mathcal{H}_{E}} S =2n$ (an hypersurface in the Euclidean sense), then it follows (see page 64 in \cite{BAL} or \cite{GROMOV}) that $\dim_{\mathcal{H}_{cc}} S=2n+1$.
\end{proof}

\begin{defin}
Denote $ T_p S $ the space of  vectors tangent to $S$ at the point $p$. Define the characteristic set $C(S)$ of a surface $S \subseteq \mathbb{H}^n$ as 
$$
C(S):= \left \{   p \in S ; \ T_p S \subseteq \mathfrak{h}_{1,p}   \right \}.
$$
To say that a point $p \in C(S)$ is the same as saying that $n_{\mathbb{H},p}=0$, which says that, for the $k$-codimensional case, it is not possible to find a map $f$ such as in Definition \ref{Hreg}. Viceversa, $p \notin C(S)$ if $n_{\mathbb{H},p} \neq 0$.
\end{defin}

\begin{obs}[1.1 in \cite{BAL} and 2.16 in \cite{MAG}]
The set of characteristic points of a $k$-codimensional surface $S \subseteq \mathbb{H}^n$ has always measure zero:
$$
\mathcal{H}_{cc}^{2n+2-k} \left (  C(S)   \right ) =0.
$$
\end{obs}

\begin{obs}\label{Cpoints}[See page195 in \cite{FSSC}]
Consider a $C^1$-Euclidean surface $S$ with $C(S)= \varnothing$. It follows that $S$ is a 
$\mathbb{H}$-regular surface in $\mathbb{H}^n$.
\end{obs}

\begin{proof}
Since S is $C^1$-Euclidean, it can be written as
$$
S=C(S) \cup \left ( S   \ \setminus \  C(S) \right ),
$$
where $S   \ \setminus \  C(S) $ is $\mathbb{H}$-regular by definition of $C(S)$. Since $C(S)= \varnothing$, $S$ is also $\mathbb{H}$-regular.
\end{proof}

\section{The Möbius Strip in $\mathbb{H}^1$}

In this section we show that there exist $\mathbb{H}$-regular surfaces which are non-orientable in the Euclidean sense. We prove it for a ``Möbius strip"-like surface. This result will be crucial later on. We define precisely the orientability in the Euclidean sense later in Definition \ref{Eorientable}; we will not use the notion in this section but only the knowledge that the Möbius strip is not orientable.\\\\
Let $\mathcal{M}$ be any Möbius strip. Is whether $C(\mathcal{M}) = \varnothing$? Or, if this is not possible, is there a surface $\widebar{\mathcal{M}}  \subseteq \mathcal{M}$, $\widebar{\mathcal{M}}$ still non-orientable in the Euclidean sense, such that  $C( \widebar{\mathcal{M}} ) = \varnothing$?\\
\noindent
We will make this question more precise and everything will be proved carefully.  
Our strategy will be the following: considering one specific parametrization of the Möbius strip $\mathcal{M} $, we show that there is only at most one characteristic point $\bar p$. Therefore the surface
$$
\widebar{\mathcal{M}}   :=\mathcal{M}  \ \setminus \  U_{\bar p},
$$
where $U_{\bar p}$ is a neighbourhood of $\bar p$ taken with smooth boundary, is indeed $C^1$-Euclidean with $C(\widebar{\mathcal{M}}) = \varnothing$ and non-orientable in the Euclidean sense.  
In particular, to prove the existence of $\widebar{\mathcal{M}}$, some steps are needed: 1. parametrize the Möbius strip $\mathcal{M}$ as $\gamma(r,s)$, 2. write the two tangent vectors $\vec\gamma_r$ and $\vec\gamma_s$ in Heisenberg coordinates, 3. compute the components of the normal vector field $\vec N=\vec\gamma_r \times_H \vec\gamma_s$ and 4. compute on which points both the first and second components are not zero. \\\\
\noindent
Consider a fixed $R  \in \mathbb{R}^+ $ and $w \in \mathbb{R}^+$ so that $w<R$. Then consider the map
$$
\gamma : [0,2 \pi ) \times [-w,w] \to  \mathbb{R}^3  
$$
defined as follows
\begin{align*}
\gamma (r,s):&=(x(r,s), y(r,s), t(r,s))\\
&= \left (  \left [ R+s \cos \left  ( \frac{r}{2} \right ) \right ] \cos r  , \  \left [R+s \cos \left ( \frac{r}{2} \right ) \right ] \sin r , \  s \sin \left  ( \frac{r}{2} \right ) \right  ).
\end{align*}
This is indeed the parametrization of a Möbius strip of half-width $w$ with midcircle of radius $R$  in $\mathbb{R}^3$. We can pose then $\mathcal{M} := \gamma \left  ( [0,2 \pi ) \times [-w,w] \right  ) \subseteq  \mathbb{R}^3$.

\begin{prop}\label{Mobius}
Let $\mathcal{M}$ be the Möbius strip parametrised by the curve $\gamma$. Then $\mathcal{M}$ contains at most only one characteristic point $\bar p$ and so there exists a $1$-codimensional surface $\widebar{\mathcal{M}}  \subseteq \mathcal{M}$, $\bar p  \notin \widebar{\mathcal{M}}$,  $\widebar{\mathcal{M}}$ still non-orientable in the Euclidean sense, such that  $C( \widebar{\mathcal{M}} ) = \varnothing$.
\end{prop}

\noindent
This  says that $\widebar{\mathcal{M}}$ is a $\mathbb{H}$-regular surface (by Observation \ref{Cpoints}) and is non-orientable in the Euclidean sense. The proof will follow after some lemmas.

\begin{lem}[Step 1.]
Consider the parametrization $\gamma$. The two tangent vector fields of $\gamma$, in the basis $\{\partial_x, \partial_y, \partial_t \}$, are
\begin{align*}
\vec\gamma_r (r,s) = & \bigg (     - \frac{s}{2} \sin \left  ( \frac{r}{2} \right )     \cos r -  \left [ R+s \cos \left  ( \frac{r}{2} \right ) \right ]     \sin r     , \\ 
& - \frac{s}{2} \sin \left  ( \frac{r}{2} \right )     \sin r +  \left [ R+s \cos \left  ( \frac{r}{2} \right ) \right ]   \cos r,  \     \frac{s}{2} \cos \left  ( \frac{r}{2} \right )    \bigg  ),\\
\text{and}\hspace{0.8cm} &\\
\vec\gamma_s (r,s) =& \left (   \cos \left  ( \frac{r}{2} \right )  \cos r   , \    \cos \left ( \frac{r}{2} \right )  \sin r  , \    \sin \left  ( \frac{r}{2} \right ) \right  ).
\end{align*}
\end{lem}

\begin{lem}[Step 2.]
Consider the parametrization $\gamma$. The two tangent vector fields $\vec\gamma_r$ and $\vec\gamma_s$ can be written in Heisenberg coordinates as:
\begin{align*}
\vec\gamma_r (r,s)= & \left (   - \frac{1}{2}s \sin \left  ( \frac{r}{2} \right )     \cos r -  \left [ R+s \cos \left  ( \frac{r}{2} \right ) \right ] \sin r  \right ) X \\
&+ \bigg ( - \frac{1}{2}s \sin \left  ( \frac{r}{2} \right )     \sin r + \left [ R+s \cos \left  ( \frac{r}{2} \right ) \right ]  \cos r \bigg ) Y\\
& +  \left  (  s \frac{1}{2} \cos \left  ( \frac{r}{2} \right )  - \left [ R+s \cos \left  ( \frac{r}{2} \right ) \right ]^2 \frac{1}{2} \right )T,\\
\text{and}\hspace{0.8cm} &\\
\vec\gamma_s (r,s) = &     \cos \left  ( \frac{r}{2} \right )  \cos r X   +  \cos \left ( \frac{r}{2} \right )  \sin r Y  +  \sin \left  ( \frac{r}{2} \right )  T .     
\end{align*}
\end{lem}

\noindent
The computation is at Observations \ref{Hcoordinates1} and  \ref{Hcoordinates2}.\\\\
Call $\vec{N} (r,s)= \vec{N}_1  (r,s) X + \vec{N}_2  (r,s) Y + \vec{N}_3  (r,s) T$ the normal vector field of $\mathcal{M}$. Such vector is given by the cross product of the two tangent vector fields $\vec\gamma_r$ and $\vec\gamma_s$. Specifically:
$$
\vec{N}= \vec{N}_1 X + \vec{N}_2 Y + \vec{N}_3 T =
\vec\gamma_r \times_\mathbb{H}  \vec\gamma_s =
$$ 
\[ =
\begin{vmatrix}
X & Y & T \\ 
    \frac{-s  \cos r \sin  \frac{r}{2} }{2}   -   \left [ R+s \cos  \frac{r}{2}  \right ] \sin r&
 \frac{-s  \sin r \sin  \frac{r}{2}  }{2}   +  \left [ R+s \cos \frac{r}{2}  \right ]  \cos r  &
  \frac{s  \cos  \frac{r}{2}  }{2}    -  \frac{   \left [ R+s \cos \frac{r}{2} \right ]^2   }{2}\\
\cos \left  ( \frac{r}{2} \right )  \cos r &  \cos \left ( \frac{r}{2} \right )  \sin r &  \sin \left  ( \frac{r}{2} \right )
\end{vmatrix}.
\]

\begin{lem}[Step 3.]
Consider $\vec{N} (r,s)= \vec{N}_1  (r,s) X + \vec{N}_2  (r,s) Y + \vec{N}_3  (r,s) T$ the normal vector field of $\mathcal{M}$. A calculation shows that:
\begin{align*}
\vec{N}_1 =&
 - \frac{1}{2}s    \sin r 
+  \left [ R+s \cos \left  ( \frac{r}{2} \right ) \right ]  \cos r  \sin \left  ( \frac{r}{2} \right )     
 +  \left [ R+s \cos \left  ( \frac{r}{2} \right ) \right ]^2 \frac{1}{2}  \cos \left ( \frac{r}{2} \right )  \sin r  ,\\
\vec{N}_2 =&
\left (   - z^5     +  \frac{1}{2}  z^3   \right )   s^2   
+   \left (    - 2 (R+1)z^4       +    ( R+3) z^2     - \frac{1}{2}   \right )   s   
   - ( R^2  + 2   R) z^3  \\
& + \left ( \frac{1}{2} R^2    +  2   R  \right  )   z  ,\\
\text{and}&\\
 \vec{N}_3 =& -   \left [ R+s \cos \left  ( \frac{r}{2} \right ) \right ]   \cos \left ( \frac{r}{2} \right ) .
\end{align*}
with $z=\cos \left  ( \frac{r}{2} \right ) $, $r\in [0,2\pi)$ and $s \in [-w,w]$. 
\end{lem}

\noindent
The full computation is at Observations \ref{TvecN}, \ref{XvecN} and \ref{YvecN}.

\begin{proof}[Proof of Proposition \ref{Mobius}] 
To find  pairs of parameters $(r,s)$ corresponding to characteristic points we have to impose
$$
\begin{cases}
\vec{N}_1  (r,s)= 0\\
\vec{N}_2  (r,s)= 0.
\end{cases}
$$
We find that $\vec{N}_1  (r,s)= 0$ only at the points $(x(r,s),y(r,s),t(r,s))$ with
$$
(r,s)=
\begin{cases}
(0,s), \quad  &s \in [-w,w], \quad \text{or} \\
\left (r, \frac{ -(R+1) z^2 + 1 \pm \sqrt{ z^4  -( R+2 ) z^2 +1 } }{ z^3} \right  ),
 \quad &r \in [0, 2 \pi), \ r \neq \pi, \ z=\cos \frac{r}{2}.
\end{cases}
$$
Evaluating these possibilities on $\vec{N}_2  (r,s)= 0$ (full computation is at Observations \ref{N1=0N2=0} and \ref{partialconclusion}), we find that the system is verified only by the pair (that defines then a characteristic point)    
$$
(r,s)= \left (0,\frac{ -2R+1 - \sqrt{-4R+1}  }{2} \right ),\quad \text{when} \quad 0< R < \frac{1}{4},
$$
which corresponds to the point $\bar p = (\bar x,\bar y,\bar t) = \left ( \frac{1}{2} - \sqrt{-R + \frac{1}{4} } ,0,0 \right ) $:
$$
\begin{cases}
\bar x	=	[R+s \cos ( \frac{r}{2} ) ] \cos r =R +\frac{ -2R+1 - \sqrt{-4R+1}  }{2} = \frac{ 1 - \sqrt{-4R+1}  }{2} =\frac{1}{2} - \sqrt{-R + \frac{1}{4} }>0	\\
\bar y	=	[R+s \cos ( \frac{r}{2} ) ] \sin r=0	\\
\bar t	=	s \sin ( \frac{r}{2} )=0.	
\end{cases}
$$
Notice that it is not strange that the number of characteristic points depends on the radius $R$ as changing the radius is not an anisotropic dilation.\\
Therefore the surface
$$
\widebar{\mathcal{M}} :=\mathcal{M}  \ \setminus \  U_{\bar p},
$$
where $U_{\bar p}$ is a neighbourhood of $\bar p$ with smooth boundary, is indeed a $C^1$-Euclidean surface with $C(\widebar{\mathcal{M}}) = \varnothing$ (and so $1$-codimensional $\mathbb{H}$-regular) and not Euclidean-orientable. This completes the proof. 
\end{proof}


\section{Comparing Orientabilities}

In this section we first recall the definition of Euclidean orientability, then introduce the notion of Heisenberg orientability ($\mathbb{H}$-orientability) and show the connection between differential forms and orientability.  Next we prove that, under left translations and anisotropic dilations, $\mathbb{H}$-regularity is invariant for $1$-codimensional surfaces and $\mathbb{H}$-orientability is invariant for $\mathbb{H}$-regular $1$-codimensional surfaces. Finally, we show how the two notions of orientability are related, concluding that non-$\mathbb{H}$-orientable $\mathbb{H}$-regular surfaces exist, at least when $n=1$.

\begin{obs}
Recall that, if $S$ is a $\mathbb{H}$-regular $1$-codimensional surface in $\mathbb{H}^n$, Definition \ref{Hreg} becomes:
\begin{equation}\label{eq1}
\text{for all } p \in S \ \text{ there exists a neighbourhood }  U  \text{ and }  f : U \to \mathbb{R}, \ f \in C_{\mathbb{H}}^1 (U, \mathbb{R}),  \text{ so that } 
\end{equation}
$$
S \cap U = \{ f=0 \} \text{ and } \nabla_{\mathbb{H}} f \neq 0 \text{ on } U.
$$
On the other hand, in case $S$ is $C^1$-Euclidean then (see for instance the introduction of \cite{BAL})
\begin{equation}\label{eq2}
\text{for all } p \in S \ \text{ there exists a neighbourhood }  U \text{ and } g : U \to \mathbb{R}, \ g \in C^1(U, \mathbb{R}), \text{ so that } 
\end{equation}
$$
S \cap U = \{ g=0 \} \text{ and } \nabla g \neq 0 \text{ on } U.
$$
These two notions of regularity are obviously similar. First we will connect each of them to a definition of orientability; then we will compare these last two definitions.
\end{obs}



\subsection{$\mathbb{H}$-Orientability in $\mathbb{H}^n$}\label{Horient}

In this subsection we recall the notion of Euclidean orientability and define the Heisenberg orientability. We add some observations and an explicit representation of continuous global independent vector field tangent to an $\mathbb{H}$-orientable surface in the Heisenberg sense.

\begin{no}
Denote $ T S $ the space of  vector fields tangent to a surface $S$. A vector $v$ is \emph{normal to $S$} (and write $v \perp S$) if
$$
\langle v,w \rangle =0 \ \ \forall w \in TS,
$$
where $\langle , \rangle$ appears in Definition \ref{scal}.
\end{no}

\begin{defin}\label{Eorientable}
Consider a $1$-codimensional $C^1$-Euclidean surface $S$ in $\mathbb{H}^n$ with $C(S) = \varnothing$. We say that $S$ is \emph{Euclidean-orientable} (or \emph{orientable in the Euclidean sense}) if there exists a continuous global $1$-vector field 
$$
 n_E=\sum_{i=1}^{n} \left ( n_{E,i} \partial_{x_i} + n_{E,n+i} \partial_{y_i} \right ) + n_{E,2n+1} \partial_t  \neq 0 ,
$$
defined on $S$ and normal to $S$.\\
Such $n_{E}$ is called \emph{Euclidean normal vector field of} $S$.
\end{defin}

\begin{lem}
Consider a $1$-codimensional $C^1$-Euclidean surface $S$ in $\mathbb{H}^n$ with $C(S) = \varnothing$. The following are equivalent:
\begin{enumerate}[label=(\roman*)]
\item
$S$ is \emph{Euclidean-orientable}
\item
there exists a continuous global $2n$-vector field $t_E$ on $S$, so that $t_E$ is tangent to $S$ in the sense that $*t_E$ is normal to $S$.
\end{enumerate}
The symbol $*$ represents the Hodge operator (see Definition \ref{hodge}).
\end{lem}

\begin{obs}
It is straightforward that, up to a choice of sign,
$$t_E = * n_E.$$
\end{obs}

\noindent
It is possible to give an equivalent definition of orientability (both Euclidean or Heisenberg) using differential forms: a manifold is orientable (in some sense) if and only if there exists a proper volume form on it; where a volume form is a never-null form of maximal order. In particular:

\begin{obs}
Consider a $1$-codimensional $C^1$-Euclidean surface $S$ in $\mathbb{H}^n$ with $C(S) = \varnothing$. A fourth equivalent condition is that $S$ allows a volume form $\omega_E$, which can be chosen so that the following property holds: 
$$
\langle \omega_E \vert t_E \rangle=1.
$$
\end{obs}


\begin{obs}\label{Elocally}
Both vectors $t_E$ and $n_E$ can be written in local components. For example, if $n=1$, by condition \eqref{eq2} locally gives $n_{E}=\mu \nabla g$, with $\mu \in C^\infty (\mathbb{H}^1,\mathbb{R})$. 
So,
$$
n_E=  n_{E,1} \partial_x  + n_{E,2} \partial_y + n_{E,3}  \partial_t =    \mu \partial_x g \partial_x +\mu \partial_y g \partial_y + \mu \partial_t g \partial_t
$$
and then
$$
t_E = * n_E = \mu \partial_x g \partial_y \wedge \partial_t - \mu \partial_y g \partial_x \wedge \partial_t + \mu \partial_t g \partial_x \wedge \partial_y.
$$
In this case, the corresponding volume form is
$$
\omega_E= \frac{ n_{E,1} }{ n_{E,1}^2+n_{E,2}^2+n_{E,3}^2  } dy \wedge dt 
- \frac{ n_{E,2} }{ n_{E,1}^2+n_{E,2}^2+n_{E,3}^2 } dx \wedge dt
$$
$$
+ \frac{ n_{E,3} }{ n_{E,1}^2+n_{E,2}^2+n_{E,3}^2 } dx \wedge dy.
$$
\end{obs}


\begin{defin}
Consider two vectors $v, w \in \mathfrak{h}_1 
$ in $\mathbb{H}^n$. We say that $v$ and $w$ are \emph{ orthogonal in the Heisenberg sense } $(v \perp_H  w)$  if 
$$
 \langle v,w \rangle_H = 0,
$$
where $\langle \cdot , \cdot \rangle_H$ is the scalar product that makes $X_j$'s and $Y_j$'s orthonormal (see Observation  \ref{scal}).
\end{defin}

\begin{defin}
Consider a $1$-codimensional $C^1$-Euclidean surface $S$ in $\mathbb{H}^n$, with $C(S) = \varnothing$. Consider also a vector $v \in  HT\mathbb{H}^n$. We say that $v$ and $S$ are \emph{orthogonal in the Heisenberg sense } $(v \perp_H  S)$ if 
$$
 \langle v,w \rangle_H = 0, \quad \text{for all } \ w \in TS.
$$
In the same way one can say that a $k$-vector field $v$ on $S$ is \emph{tangent to} $S$ \emph{in the Heisenberg sense}  if  
$$
\langle *v,w\rangle_H = 0,  \quad \text{for all } \ w \in TS.
$$
\end{defin}

\begin{defin}
Consider a $1$-codimensional $C^1$-Euclidean surface $S$ in $\mathbb{H}^n$ with $C(S) = \varnothing$. We say that $S$ is $\mathbb{H}$\emph{-orientable} (or \emph{orientable in the Heisenberg sense})  if there exists a continuous global $1$-vector field 
$$
n_{\mathbb{H}}=\sum_{i=1}^{n} \left ( n_{\mathbb{H},i} X_i + n_{\mathbb{H},n+i} Y_i \right ) \neq 0,
$$
defined on $S$ so that $n_{\mathbb{H}}$ and $S$ are orthogonal in the Heisenberg sense ($n_{\mathbb{H}} \perp_H S$).\\
Such $n_{\mathbb{H}}$ will be called \emph{Heisenberg normal vector field of} $S$.
\end{defin}

\begin{lem}
Consider a $1$-codimensional $C^1$-Euclidean surface $S$ in $\mathbb{H}^n$ with $C(S) = \varnothing$. The following are equivalent:
\begin{enumerate}[label=(\roman*)]
\item
$S$ is $\mathbb{H}$\emph{-orientable}
\item
there exists a continuous global $2n$-vector field $t_\mathbb{H}$ on $S$ so that $t_\mathbb{H}$ is $\mathbb{H}$-tangent to S in the sense that $* t_\mathbb{H}$ is $\mathbb{H}$-normal to $S$.
\end{enumerate}
\end{lem}

\begin{obs}
Again, one can easily say that, up to a choice of sign,
$$t_\mathbb{H}= * n_\mathbb{H}.$$
\end{obs}

\noindent
As before, it is possible to give an equivalent definition of Heisenberg orientability using a volume form on the Rumin complex:

\begin{obs}
Consider a $1$-codimensional $C^1$-Euclidean surface $S$ in $\mathbb{H}^n$ with $C(S) = \varnothing$. A fourth equivalent condition is that $S$ allow a volume form $\omega_\mathbb{H}$ (again a never-null form of maximal order), which can be chosen so that the following property holds: 
$$
\langle \omega_\mathbb{H} \vert t_\mathbb{H} \rangle=1.
$$
\end{obs}

\begin{obs}\label{Hlocally}
As in Observation \ref{Elocally}, one can write both $n_\mathbb{H}$ and $t_\mathbb{H}$ locally in $\mathbb{H}^1$: by condition \eqref{eq1}, locally $n_\mathbb{H}=\lambda \nabla_{\mathbb{H}} f$, with $\lambda \in C^\infty (\mathbb{H}^1,\mathbb{R})$.  
So,
$$
n_{\mathbb{H}} = n_{\mathbb{H},1}  X  +n_{\mathbb{H},2}  Y  =\lambda X f X +  \lambda Y f Y
$$
and, since $t_{\mathbb{H}}  = * n_{\mathbb{H}} = n_{\mathbb{H},1} Y \wedge T - n_{\mathbb{H},2} X \wedge T $,
$$
t_{\mathbb{H}}  =    \lambda X f   Y \wedge T   - \lambda Y f  X \wedge T .
$$
In this case, the corresponding volume form is
$$
\omega_\mathbb{H}= \frac{ n_{\mathbb{H},1} }{ n_{\mathbb{H},1}^2+n_{\mathbb{H},2}^2} dy \wedge \theta 
- \frac{ n_{\mathbb{H},2} }{ n_{\mathbb{H},1}^2+n_{\mathbb{H},2}^2 } dx \wedge \theta.
$$
\end{obs}

\begin{ex}
Consider an $\mathbb{H}$-orientable 1-codimensional surface $S$ in $\mathbb{H}^1$. Then at each point $p \in S$ there exist two continuous global linearly independent vector fields $ \vec{r}$ and $  \vec{s}$ tangent on $S$, $ \vec{r},  \vec{s} \in T_p S$. With the previous notation, we can explicitly find such a pair by solving the following list of conditions:
    \begin{multicols}{2}
\begin{enumerate}[nosep]
\item
$\langle \vec{r},  \vec{s} \rangle_{H} =0$,
\item
$\langle \vec{r},  n_\mathbb{H} \rangle_{H} =0$,
\item
$\langle   \vec{s}, n_\mathbb{H} \rangle_{H} =0$,
\item
$\vert \vec{r} \vert_{H} =1$,
\item
$\vert \vec{s} \vert_{H} =1$,
\item
$\vec{r} \times  \vec{s} = n_\mathbb{H}$,
\item
$\vec{r} \wedge  \vec{s} = t_\mathbb{H}.$
\end{enumerate}
    \end{multicols}
\noindent
Furthermore, one can (but it is not necessary) choose $\vec{r} = T$ since $n_\mathbb{H} \in \spn \{X,Y\}$. Then one can take $\vec{s} = aX+bY$, so the first two conditions are satisfied.
The third condition is $\langle  \vec{s}, n_\mathbb{H} \rangle_{H} =0$, meaning
$$
 a n_{\mathbb{H},1} + b n_{\mathbb{H},2} =0,
$$
and the solution is
$$
\begin{cases}
a= c n_{\mathbb{H},2}\\
b = -c n_{\mathbb{H},1}.
\end{cases}
$$
with $c$ arbitrary. Then, by Observation \ref{Hlocally}, there is a local function $f$ so that $\vec{s} = c n_{\mathbb{H},2} X - c n_{\mathbb{H},1} Y$ becomes
$$
\vec{s} =  c  \lambda Y f X  -c \lambda X f Y.
$$
The fourth condition comes for free. The fifth one is $\vert \vec{s} \vert_{\mathbb{H}} =1$, which gives
$$
1= \sqrt{c^2  \lambda^2 (Y f)^2 + c^2 \lambda^2 (X f)^2} = \vert c \lambda\vert \sqrt{ (Y f)^2 +  (X f)^2} = \vert c \lambda\vert \cdot \vert   \nabla_\mathbb{H} f   \vert,
$$
meaning
$$
 c=\pm \frac{1}{ \vert   \lambda  \nabla_\mathbb{H} f   \vert}.
$$
So one has that
\begin{align*}
\vec{s} &= \pm \frac{1}{ \vert   \lambda  \nabla_\mathbb{H} f   \vert} \left (  \lambda Y f X  - \lambda X f Y \right )=
\pm \frac{\lambda}{ \vert   \lambda  \vert }    \frac{  Y f X  -  X f Y }{  \vert  \nabla_\mathbb{H} f   \vert}\\
&=
\pm \ sign (\lambda)   \left ( 
  \frac{  Y f  }{  \vert  \nabla_\mathbb{H} f   \vert}X  -     \frac{  X f  }{  \vert  \nabla_\mathbb{H} f   \vert}Y
\right )  .
\end{align*}
The sixth condition is $\vec{r} \times  \vec{s} = n_{\mathbb{H}}$, so:
$$
 n_{\mathbb{H}}= \vec{r} \times  \vec{s} =
\begin{vmatrix}
X & Y & T \\ 
0 & 0 &  1 \\ 
c  n_{\mathbb{H},2} & -c  n_{\mathbb{H},1}  & 0
\end{vmatrix}
=
c  n_{\mathbb{H},2} Y + c  n_{\mathbb{H},1} X =c  n_{\mathbb{H}}.
$$
Then it is necessary to take $c=1$ and one has $\vert \lambda \nabla_\mathbb{H} f   \vert = 1$, namely,
$$
 \lambda = \pm  \frac{1}{\vert \nabla_\mathbb{H} f   \vert },
$$
and
$$
\vec{s} =   \lambda Y f X  - \lambda X f Y   = \pm    \left ( 
  \frac{  Y f  }{  \vert  \nabla_\mathbb{H} f   \vert}X  -     \frac{  X f  }{  \vert  \nabla_\mathbb{H} f   \vert}Y
\right )  .
$$
Finally, we verify $\vec{r} \wedge  \vec{s} =  t_{\mathbb{H}}$ (seventh and last condition):
$$
\vec{r} \wedge \vec{s} =  T \wedge ( \lambda Y f X  - \lambda X f Y  ) =\lambda X f Y \wedge T - \lambda Y f X \wedge T= t_{\mathbb{H}}.
$$
\end{ex}

\subsection{Invariances}

In this subsection we prove that, for a $1$-codimensional surface, the $\mathbb{H}$-regularity is invariant under left translations and anisotropic dilations. Furthermore, for an $\mathbb{H}$-regular $1$-codimensional surface, the $\mathbb{H}$-orientability is invariant under the same two types of transformations.

\begin{prop}\label{tau_prop}
Consider the left translation map $\tau_{\bar{p}} : \mathbb{H}^n \to \mathbb{H}^n$
, $\bar{p} \in \mathbb{H}^n$, and an $\mathbb{H}$-regular $1$-codimensional surface $S$.\\
Then $\tau_{\bar{p}} S  := \{ \bar{p}*p ; \ p \in S   \}$ is again an $\mathbb{H}$-regular $1$-codimensional surface.
\end{prop}

\begin{proof}
Since $\tau_{\bar{p}} S = \{ \bar{p}*p ; p \in S   \}$, for all $ q \in \tau_{\bar{p}} S$ there exists a point $ p \in S$ so that $q=\bar{p}*p$. 
For such $p \in S $, there exists a neighbourhood $U_p$ and a function $f: U_p \to \mathbb{R}$ so that $S \cap U_p = \{ f=0 \}$ and $\nabla_\mathbb{H} f \neq 0$ on $U_p$. 
Define $U_q := \tau_{\bar{p}} U_p = \bar{p} * U_p $, which is a neighbourhood of $q=  \bar{p}*p$, and a function $\tilde{f}: = f \circ \tau_{\bar{p}}^{-1} : U_q \to \mathbb{R} $. 
Then, for all $ q' \in U_q$,
$$
\tilde{f}(q') = (f \circ \tau_{\bar{p}}^{-1})(q')  =  f (\bar{p}^{-1} * \bar{p}* p' ) = f(p')=0,
$$
where $q'= \bar{p}* p' $ and $p' \in U_p$. Then
$$
\tau_{\bar{p}} S \cap U_q = \{ \tilde{f}=0 \}.
$$
Furthermore, on $U_q$, and using Definition \ref{leftinv},
\begin{align*}
\nabla_\mathbb{H} \tilde{f} = &\nabla_\mathbb{H} (  f \circ \tau_{\bar{p}}^{-1} ) =  \nabla_\mathbb{H} (  f \circ \tau_{\bar{p}^{-1}}) =
\sum_{i=1}^n \bigg (
X_i(  f \circ \tau_{\bar{p}^{-1}}) X_i + Y_i (  f \circ \tau_{\bar{p}^{-1}}) Y_i
\bigg )\\
=&
\sum_{i=1}^n \bigg (
[X_i (  f )\circ \tau_{\bar{p}^{-1}} ] X_i + [ Y_i (  f ) \circ \tau_{\bar{p}^{-1}} ] Y_i
\bigg )
 \neq 0
\end{align*}
as $X_i (  f )\circ \tau_{\bar{p}^{-1}}$ and $ Y_i (  f ) \circ \tau_{\bar{p}^{-1}}$ are defined on $U_p$ and on $U_p$ one of the two is always non-negative by the hypothesis that $\nabla_\mathbb{H} f \neq 0$ on $U_p$.
\end{proof}

\begin{prop}\label{delta_prop}
Consider the usual anisotropic dilation $\delta_r : \mathbb{H}^n \to \mathbb{H}^n$
, $r>0$, and an $\mathbb{H}$-regular $1$-codimensional surface $S$.\\
Then $\delta_r S := \{ \delta_r(p) ; \ p \in S   \}$ is again an $\mathbb{H}$-regular $1$-codimensional surface.
\end{prop}

\begin{proof}
Since $\delta_r S = \{ \delta_r(p) ; \ p \in S   \}$, then for all $ q \in \delta_r S$ there exists a point $  p \in S$ so that $q=\delta_r(p)$. For such $p \in S $, there exists a neighbourhood $U_p$ and a function $f: U_p \to \mathbb{R}$ so that $S \cap U_p = \{ f=0 \}$ and $\nabla_\mathbb{H} f \neq 0$ on $U_p$. 
Define $U_q := \delta_r( U_p )  $, which is a neighbourhood of $q=  \delta_r ( p )$, and a function $\tilde{f}: = f \circ \delta_{1/r} : U_q \to \mathbb{R} $. 
Then, for all $ q' \in U_q$,
$$
\tilde{f}(q') = (f \circ \delta_{1/r}   )(q')  =  f ( \delta_{1/r}  \delta_r p' ) = f(p')=0,
$$
where $q'= \bar{p}* p' $ and $p' \in U_p$. Then
$$
 \delta_r S \cap U_q = \{ \tilde{f}=0 \}.
$$
Furthermore, on $U_q$, using the fact that $\delta_{1/r}$ is a contact map and Lemma \ref{nabla_comp1},
\begin{align*}
\nabla_{\mathbb{H}} \tilde{f} = \nabla_\mathbb{H} (  f \circ \delta_{1/r} ) = (\delta_{1/r})_*^T (\nabla_{\mathbb{H}}  f)_{\delta_{1/r}}  \neq 0.
\end{align*}
\end{proof}

\begin{prop}\label{letra}
Consider a left translation map $\tau_{\bar{p}} : \mathbb{H}^n \to \mathbb{H}^n$, $\bar{p} \in \mathbb{H}^n$, the anisotropic dilation $\delta_r : \mathbb{H}^n \to \mathbb{H}^n$, $r>0$ and an $\mathbb{H}$-regular $1$-codimensional surface $S$. 
Then the $\mathbb{H}$-regular $1$-codimensional surfaces $\tau_{\bar{p}} S$  and $\delta_r S $ are $\mathbb{H}$-orientable (respectively) if and only if $S$  $\mathbb{H}$-orientable.
\end{prop}

\begin{proof}
Remember that $\tau_{\bar{p}} S = \{ \bar{p}*p ; \ p \in S   \}$. From Proposition \ref{tau_prop}, one knows that for all $ q \in \tau_{\bar{p}} S$ there exists a point $ p \in S$ so that $q=\bar{p}*p$ and there exists a neighbourhood $U_p$ and a function $f: U_p \to \mathbb{R}$ so that $S \cap U_p = \{ f=0 \}$ and $\nabla_\mathbb{H} f \neq 0$ on $U_p$. 
Furthermore, $U_q = \tau_{\bar{p}} U_p = \bar{p} * U_p $ is a neighbourhood of $q=  \bar{p}*p$ and the function $\tilde{f}: = f \circ \tau_{\bar{p}}^{-1} : U_q \to \mathbb{R} $ is so that  $\tau_{\bar{p}} S \cap U_q = \{ \tilde{f}=0 \}$ and $\nabla_\mathbb{H} \tilde{f}  \neq 0$ on $U_q$.\\
Assume now that $S$ is $\mathbb{H}$-orientable, then there exists a global vector field
$$
n_{\mathbb{H}} = \sum_{j=1}^{n} \left (  n_{\mathbb{H},j} X_j + n_{\mathbb{H},n+j} Y_j  \right ),
$$
that, locally, takes the form of
$$
 \sum_{j=1}^{n}  \left ( 
\frac{ X_j f }{ \vert \nabla_\mathbb{H} f \vert }  X_j +  \frac{ Y_j f }{ \vert \nabla_\mathbb{H} f \vert }   Y_j \right ).
$$
Now we consider:
\begin{align*}
(\tau_{\bar{p}}^{-1})_*   n_\mathbb{H} = \sum_{j=1}^{n} \left (  
n_{\mathbb{H},j}  \circ \tau_{\bar{p}}^{-1} 
 {X_j}_{\tau_{\bar{p}}^{-1}} +
n_{\mathbb{H},n+j}  \circ \tau_{\bar{p}}^{-1} 
 {Y_j}_{\tau_{\bar{p}}^{-1}}  \right ).
\end{align*}
Locally, it becomes
\begin{align*}
(\tau_{\bar{p}}^{-1})_*   n_\mathbb{H} =
\sum_{j=1}^{n} \left (  
 \frac{ X_j f }{ \vert \nabla_\mathbb{H} f \vert }  \circ \tau_{\bar{p}}^{-1} 
 {X_j}_{\tau_{\bar{p}}^{-1}} +
\frac{ Y_j f }{ \vert \nabla_\mathbb{H} f \vert }  \circ \tau_{\bar{p}}^{-1} 
 {Y_j}_{\tau_{\bar{p}}^{-1}}  \right ).
\end{align*}
Note that this is still a global vector field and is defined on the whole $\tau_{\bar{p}} S$, therefore it gives an orientation to $\tau_{\bar{p}} S$. 
Since we can repeat the whole proof starting from $\tau_{\bar{p}} S$ to $S= \tau_{\bar{p}}^{-1} \tau_{\bar{p}} S$, this proves both directions.\\\\
For the dilation, remember that $\delta_r S = \{ \delta_r(p) ; \  p \in S   \}$. From Proposition \ref{delta_prop}, for all $ q \in \delta_r S$ there exists a point $ p \in S$ so that $q=\delta_r(p)$ and there exists a neighbourhood $U_p$ and a function $f: U_p \to \mathbb{R}$ so that $S \cap U_p = \{ f=0 \}$ and $\nabla_\mathbb{H} f \neq 0$ on $U_p$. 
In the same way, $U_q = \delta_r( U_p ) $ is a neighbourhood of $q=  \delta_r(p)$ and the function $\tilde{f}: = f \circ \delta_{1/r} : U_q \to \mathbb{R} $ is so that  $\delta_r S \cap U_q = \{ \tilde{f}=0 \}$ and $\nabla_\mathbb{H} \tilde{f}  \neq 0$ on $U_q$.\\
Assume now that $S$ is $\mathbb{H}$-orientable Then there exists a global vector field
$$
n_\mathbb{H} = \sum_{j=1}^{n} \left (  n_{\mathbb{H},j}  X_j + n_{\mathbb{H},n+j}  Y_j  \right ),
$$
that, locally is written as
$$
 \sum_{j=1}^{n}  \left ( 
\frac{ X_j f }{ \vert \nabla_\mathbb{H} f \vert }  X_j +  \frac{ Y_j f }{ \vert \nabla_\mathbb{H} f \vert }   Y_j \right ).
$$
Now
\begin{align*}
( \delta_{1/r} )_*   n_\mathbb{H} =  \sum_{j=1}^{n} \left (  
 n_{\mathbb{H},j}  \circ \delta_{1/r} 
 {X_j}_{\delta_{1/r}} +
n_{\mathbb{H},n+j}  \circ \delta_{1/r} 
 {Y_j}_{\delta_{1/r}}  \right ).
\end{align*}
Locally, it becomes
\begin{align*}
( \delta_{1/r} )_*   n_\mathbb{H}  =
 \sum_{j=1}^{n} \left (  
 \frac{ X_j f }{ \vert \nabla_\mathbb{H} f \vert }  \circ \delta_{1/r} 
 {X_j}_{\delta_{1/r}} +
\frac{ Y_j f }{ \vert \nabla_\mathbb{H} f \vert }  \circ \delta_{1/r}
 {Y_j}_{\delta_{1/r}}  \right ).
\end{align*}
Note that this is still a global vector field and is defined on the whole $\delta_r S$, therefore it gives an orientation to $\delta_r S$. 
Since we can repeat the whole proof starting from $\delta_r S$ to $S=\delta_{1/r} \delta_r S$, this proves both directions.
\end{proof}


\subsection{Comparison}

In this subsection we show how the two notions of orientability are related, concluding that non-$\mathbb{H}$-orientable $\mathbb{H}$-regular surfaces exist, at least when $n=1$.\\

\begin{no}
Consider a $1$-codimensional $C^1$-Euclidean surface $S$ in $\mathbb{H}^n$ with $C(S)\neq 0$. We say that a surface is $C^2_\mathbb{H}$-regular if its Heisenberg normal vector fields $n_\mathbb{H} \in C^1_\mathbb{H}$.
\end{no}

\begin{prop}\label{finalmente4}
Let $S$ be a  $1$-codimensional $C^1$-Euclidean surface in $\mathbb{H}^{n}$, with $C(S) = \varnothing$. Then the following holds: 
\begin{enumerate}
\item Suppose $S$ is Euclidean-orientable. Recall from condition \eqref{eq2} that $C^1$-Euclidean means that for all $ p \in S$ there exists $ U \in \mathcal{U}_p$ and $g : U \to \mathbb{R}$, $g \in C^1$, so that $S \cap U = \{ g=0 \}$ and $\nabla g \neq 0$  on $U$.
If, for any such $g$, no point of $S$ belongs to the set
$$
\left \{ 
\left ( 
  - \frac{2    ( \partial_{y_1}   g )_p  }{ ( \partial_t   g )_p  } 
, \dots, 
  - \frac{2    ( \partial_{y_n}   g )_p  }{ ( \partial_t   g )_p  } 
,
  \frac{2    ( \partial_{x_1}   g )_p  }{ ( \partial_t   g )_p  }
,\dots,
  \frac{2    ( \partial_{x_n}   g )_p  }{ ( \partial_t   g )_p  }
,t
 \right )
 , \text{ with }   ( \partial_t   g )_p \neq 0  \right \},
$$
then
\begin{equation}\label{.1}
 S \text{ is } \mathbb{H}\text{-orientable}.
\end{equation}
\item
If $S$ is  $C^2_\mathbb{H}$-regular,
\begin{equation}\label{.2}
S \text{ is } \mathbb{H}\text{-orientable implies }   S \text{ is Euclidean-orientable } .
\end{equation}
\end{enumerate}
\end{prop}

\noindent
The proof will follow at the end of this chapter. A question arises naturally:

\begin{que}
About the extra conditions for the first implication in Proposition \ref{finalmente4}, what can we say about that set? Is it possible to do better?
\end{que}

\begin{rem}
Remember that, by Proposition \ref{Mobius}, $ \widebar{\mathcal{M}}$ is an $\mathbb{H}$-regular surface not Euclidean-orientable. 
Observe also that $ \widebar{\mathcal{M}}$ satisfies the hypotheses of Proposition \ref{finalmente4}.\\
Then implication \eqref{.2} in Proposition \ref{finalmente4} implies that $\widebar{\mathcal{M}}$ is not an $\mathbb{H}$-orientable $\mathbb{H}$-regular surface. 
Then we can say that there exist $\mathbb{H}$-regular surfaces which are not $\mathbb{H}$-orientable, at least when $n=1$.\\
This opens the possibility, among others, to study surfaces that are, in the Heisenberg sense, regular but not orientable; for example as supports of Sub-Riemannian currents.
\end{rem}

\noindent
Before proving the proposition, we state a weaker result as a lemma.

\begin{lem}\label{question}
Let $S$ be a $1$-codimensional $C^1$-Euclidean surface in $\mathbb{H}^n$, with $C(S) = \varnothing$.\\
Recall that, by Observation \ref{Cpoints}, $S$ is $\mathbb{H}$-regular and, by condition \eqref{eq1}, this means that for all $p \in S$ there exists $ U  \in \mathcal{U}_p$  and  $ f : U \to \mathbb{R}$, $f \in C_{\mathbb{H}}^1$,  so that $S \cap U = \{ f=0 \}$ and $ \nabla_{\mathbb{H}} f \neq 0$  on $U$. Likewise, by condition \eqref{eq2}, $C^1$-Euclidean means that for all $ p \in S$ there exists $ U' \in \mathcal{U}_p$ and $g : U' \to \mathbb{R}$, $g \in C^1$, so that $S \cap U' = \{ g=0 \}$ and $\nabla g \neq 0 $ on $ U'$.\\
Assume that for each $p \in S$ we can consider $f=g$ on $U \cap U'$.  Then the following holds:
$$
S \text{ is Euclidean-orientable implies }   S \text{ is } \mathbb{H}\text{-orientable}.
$$
If $S$ is  $C^2_\mathbb{H}$-regular,
$$
S \text{ is } \mathbb{H}\text{-orientable implies }   S \text{ is Euclidean-orientable } .
$$
\end{lem}

\begin{proof}
The first implication of Lemma \ref{question} says that there exists a global continuous vector field 
$$
 n_E= \sum_{i=1}^{n} \left ( n_{E,i} \partial_{x_i} + n_{E,n+i} \partial_{y_i} \right ) + n_{E,2n+1} \partial_t  \neq 0 
$$ 
so that for all  $p \in S$ there exist $U \in \mathcal{U}_p $ and a function $ g : U \to \mathbb{R}$, $ g \in C^1$,  so that $S \cap U = \{ g=0 \}$ and $ \nabla g \neq 0$  on $ U$, with 
$$
\begin{cases}
n_{E,i} = \mu \partial_{x_i} g,\\
n_{E,n+i} = \mu \partial_{y_i} g,  \\
n_{E,2n+1} = \mu \partial_t g,
\end{cases}
\quad  i=1,\dots,n,
$$
where $\mu$ is simply a normalising factor so that $\vert n_E \vert =1$ that, from now on, can be ignored. By hypothesis, we also know that for all $ p \in S$ there exist $ \exists U' \in \mathcal{U}_p$ and a function $  f : U' \to \mathbb{R}$, $ f \in C_{\mathbb{H}}^1$,  so that $S \cap U = \{ f=0 \} $ and $ \nabla_{\mathbb{H}} f \neq 0 \text{ on } U'$.\\
The extra-hypothesis of this lemma is exactly that $g=f$ and so one can also take $U=U'$. The goal is to find a global $n_{\mathbb{H}}$. A natural choice is
$$
\begin{cases}
n_{\mathbb{H},i} :=n_{E,i} -\frac{1}{2}y_i \cdot n_{E,2n+1},\\
n_{\mathbb{H},n+i} := n_{E,n+i} +\frac{1}{2}x_i \cdot n_{E,2n+1},
\end{cases}
\quad  i=1,\dots,n.
$$
Then
\begin{align*}
\sum_{i=1}^{n} \left ( n_{\mathbb{H},i} ^2 + n_{\mathbb{H},n+i}^2 \right ) =&  \sum_{i=1}^{n} \left [ \left (    n_{E,i} -\frac{1}{2}y_i \cdot n_{E,2n+1}  \right )^2 + \left (    n_{E,n+i} +\frac{1}{2}x_i \cdot n_{E,2n+1}  \right )^2 \right ] .
\end{align*}
Locally, and remembering $g=f$, this means
\begin{align*}
\sum_{i=1}^{n} \left ( n_{\mathbb{H},i} ^2 + n_{\mathbb{H},n+i}^2 \right ) =& \mu^2   \sum_{i=1}^{n} \left [ \left (    \partial_{x_i} g     -\frac{1}{2}y_i     \partial_t g \right )^2 +  \left (   \partial_{y_i} g  +\frac{1}{2}x_i \partial_t g  \right )^2 \right ] \\
= & \mu^2  \sum_{i=1}^{n} \left [ \left (    \partial_{x_i} f     -\frac{1}{2}y_i     \partial_t f \right )^2 +  \mu^2 \left (   \partial_{y_i} f  +\frac{1}{2}x_i \partial_t f  \right )^2 \right ] \\
=& \mu^2 \sum_{i=1}^{n}  \left [ \left (    X_i f \right )^2 + \left (   Y_if  \right )^2 \right ] \neq 0.
\end{align*}
So there exists a proper vector 
$$
n_{\mathbb{H}}:=\sum_{i=1}^{n} \left ( n_{\mathbb{H},i}  X_i+n_{\mathbb{H},n+i} Y_i \right )
$$ 
and this proves the first implication.\\\\
To prove the second implication of Lemma \ref{question}, first note that there exists a global continuous vector field 
$$
 n_{\mathbb{H}}=\sum_{i=1}^{n} \left (  n_{\mathbb{H},i} X_i + n_{\mathbb{H},n+i} Y_i \right ) \neq 0 
$$
 so that for all $ p \in S$ there exist $ U \in \mathcal{U}_p $ and $ f : U \to \mathbb{R}$, $ f \in C_{\mathbb{H}}^1$,  so that $ S \cap U = \{ f=0 \}$ and $ \nabla_{\mathbb{H}} f \neq 0$  on $ U$, with 
$$
\begin{cases}
n_{\mathbb{H},i} = \lambda X_i f,\\
n_{\mathbb{H},n+i} = \lambda Y_i f,
\end{cases}
\quad  i=1,\dots,n,
$$
where $\lambda$ is just a normalising factor so that $\vert n_{\mathbb{H}} \vert =1$ that, from now on, can be ignored.\\ 
By hypothesis, we also know that for all $ p \in S$ there exist $ U' \in \mathcal{U}_p$ and $ g : U' \to \mathbb{R}$, $ g \in C^1$,  so that  $S \cap U' = \{ g=0 \} $ and $\nabla g \neq 0 $ on $U'$.\\
As before, since we am in the case $g=f$, we can also take $U=U'$. This time the goal is to find a global vector field $n_E$. A natural choice is
$$
\begin{cases}
n_{E,2n+1}:=\frac{1}{n} \sum_{j=1}^{n} \left ( X_j n_{\mathbb{H},n+j} - Y_j n_{\mathbb{H},j} \right ) ,\\
n_{E,i} := n_{\mathbb{H},i} +\frac{1}{2}y_i \cdot n_{E,2n+1} ,\\
n_{E,n+i} := n_{\mathbb{H},n+i} -\frac{1}{2}x_i \cdot n_{E,2n+1},
\end{cases}
\quad  i=1,\dots,n.
$$
Note that we need 
 $ n_{\mathbb{H},i},  n_{\mathbb{H},n+i} \in C_{\mathbb{H}}^1 \left (  n_{\mathbb{H}} \in C_{\mathbb{H}}^1 \right )$ for this definition to be useful. Indeed, this is the same as having $S$ $ C^2_\mathbb{H}$-regular.\\
Now note that
$$
n_{E,2n+1}=\frac{1}{n}  \sum_{j=1}^{n} \left ( X_j n_{\mathbb{H},n+j} - Y_j n_{\mathbb{H},j} \right ) .
$$
Moreover, remembering that locally $g=f$ and that in each $U$ we have one such function, we get
$$
n_{E,2n+1} = \frac{1}{n}  \sum_{j=1}^{n} \left (   X_j  Y_j f - Y_j  X_j f \right )  =\frac{1}{n}   \sum_{j=1}^{n} \left[        X_j  Y_j g - Y_j X_j g         \right  ] =\frac{1}{n} n  Tg=  \partial_t g.
$$
Now
\begin{align*}
& \sum_{i=1}^{n}     \left ( n_{E,i}^2 + n_{E,n+i}^2 \right ) + n_{E,2n+1}^2 = \\
&=
 \sum_{i=1}^{n} \left [ \left (    n_{\mathbb{H},i} +\frac{1}{2}y_i   n_{E,2n+1}  \right )^2 
+ \left (    n_{\mathbb{H},n+i}  -\frac{1}{2}x_i  n_{E,2n+1}   \right )^2  \right ]
+ n_{E,2n+1}^2.
\end{align*}
This locally is
\begin{align*}
& \sum_{i=1}^{n}     \left ( n_{E,i}^2 + n_{E,n+i}^2 \right ) + n_{E,2n+1}^2 = \\
& =
 \sum_{i=1}^{n} \left [   \left (      X_i g +\frac{1}{2}y_i (   \partial_t g )  \right )^2 
+ \left (     Y_i g  -\frac{1}{2}x_i (  \partial_t g )   \right )^2   \right ]
+ \left (       \partial_t g  \right )^2\\
&=
 \sum_{i=1}^{n} \left [  \left (            \partial_{x_i} g     -\frac{1}{2}{y_i}     \partial_t g   +\frac{1}{2}{y_i}  \partial_t g  \right )^2 
+ \left (    \partial_{y_i} g     +\frac{1}{2}{x_i}     \partial_t g   -\frac{1}{2}{x_i} \partial_t g   \right )^2 \right ]
+ \left (      \partial_t g  \right )^2\\
&=
 \sum_{i=1}^{n} \left (  \left (            \partial_{x_i} g     \right )^2 
+ \left (    \partial_{y_i} g       \right )^2   \right )
+ \left (      \partial_t g  \right )^2  \neq 0.
\end{align*}
So there is a proper vector
$$
n_{E}:=\sum_{i=1}^{n} \left (  n_{E,i} \partial_{x_i}+n_{E,n+i} \partial_{y_i} \right ) +n_{E,2n+1} \partial_t
$$
 and this proves the second implication and concludes the proof.
\end{proof}

\noindent
Differently from the proof of Lemma \ref{question}, for Proposition \ref{finalmente4} we cannot assume $g=f$. We can still construct the vector field as before but cannot prove straight away that the new-built global vector field is never zero.

\begin{proof}[Proof of implication \eqref{.1} in Proposition \ref{finalmente4}]
We know there exists a global vector field 
$$ 
n_{E}=\sum_{i=1}^{n} \left ( n_{E,i} \partial_{x_i} +n_{E,n+i} \partial_{y_i} \right ) + n_{E,2n+1} \partial_t  \neq 0 
$$ 
that can be written locally as 
$$
 n_{E}=\mu \sum_{i=1}^{n} \left (  \partial_{x_i} g  \partial_{x_i} +  \partial_{y_i} g  \partial_{y_i} \right ) + \mu \partial_t  g  \partial_t
$$
 so that $\nabla g  \neq 0 ,  \ g \in C^1(U,\mathbb{R})$, $U\subseteq \mathbb{H}^n$ open.   
Define
$$
\begin{cases}
n_{\mathbb{H},i} :=n_{E,i} -\frac{1}{2}y_i \cdot n_{E,2n+1},\\
n_{\mathbb{H},n+i} := n_{E,n+i} +\frac{1}{2}x_i \cdot n_{E,2n+1},
\end{cases}
\quad  i=1,\dots,n.
$$
Locally (for each point $p$ there exists a neighbourhood $U$ where such $g$ is defined) this becomes
$$
\begin{cases}
n_{\mathbb{H},i} =\mu \partial_{x_i}   g -\frac{1}{2} y_i \mu \partial_t   g= \mu  X_i g,\\
n_{\mathbb{H},n+i} = \mu \partial_{y_i}   g +\frac{1}{2} x_i \mu \partial_t   g= \mu  Y_i g ,
\end{cases}
\quad  i=1,\dots,n,
$$
where $\mu$ is simply a normalising factor that, from now on, we ignore.\\
In order to verify the $\mathbb{H}$-orientability, we have to show that $\nabla_{\mathbb{H}} g \neq 0$ .  
Note here that $ C^1(U,\mathbb{R}) \subsetneq C_{\mathbb{H}}^1(U,\mathbb{R}), $ so $g$ is regular enough.\\ 
Consider first the case in which $( \partial_t   g )_p =0$. We still have that $\nabla_{p} g \neq 0$, so at least one of the derivatives $ ( \partial_{x_i}   g )_p , \ ( \partial_{y_i}   g )_p $ must be different from zero in $p$. But, when $( \partial_t   g )_p =0$, then  $(X_i g)_p= ( \partial_{x_i}   g )_p$ and $(Y_i g)_p= ( \partial_{y_i}   g )_p $, so
$$
\norm{ \nabla_{\mathbb{H},p} g }^2=(X_i g)_p^2 + (Y_i g)_p^2  \neq 0.
$$
\noindent
Second, consider the case when $( \partial_t   g )_p \neq 0  $. In this case:
$$
\norm{ \nabla_{\mathbb{H},p} g }^2 =\sum_{i=1}^{n} (X_i g)_p^2 + (Y_i g)_p^2 = \sum_{i=1}^{n}  \left  (  \partial_{x_i} g  -\frac{1}{2}y_{i,p}  \partial_t   g \right )_p^2 + \left (  \partial_{y_i} g + \frac{1}{2}x_{i,p}  \partial_t   g  \right )_p^2  \neq 0
$$
is equivalent to the fact that there exists $i \in \{1,\dots, n\}$ such that
$$
 y_{i,p}   \neq  \frac{2    ( \partial_{x_i}   g )_p  }{ ( \partial_t   g )_p  } \  \text{ or } \   x_{i,p}   \neq  - \frac{2    ( \partial_{y_i} g )_p  }{ ( \partial_t   g )_p  }.
$$
So the Heisenberg gradient of $g$ in $p$ is zero at the points\\
$$
\left ( 
  - \frac{2    ( \partial_{y_1}   g )_p  }{ ( \partial_t   g )_p  } 
, \dots, 
  - \frac{2    ( \partial_{y_n}   g )_p  }{ ( \partial_t   g )_p  } 
,
  \frac{2    ( \partial_{x_1}   g )_p  }{ ( \partial_t   g )_p  }
,\dots,
  \frac{2    ( \partial_{x_n}   g )_p  }{ ( \partial_t   g )_p  }
,t
 \right )
$$
and the first implication of the proposition is true.
\end{proof}


\begin{proof}[Proof of implication \eqref{.2} in Proposition \ref{finalmente4}]
In the second case \eqref{.2}, we know that there exists a global vector 
$$
 n_{\mathbb{H}}= \sum_{i=1}^{n} n_{\mathbb{H},i} X_i + n_{\mathbb{H},n+i} Y_i  \neq 0 
$$
 that can be written locally as 
$$
 n_{\mathbb{H}}= \sum_{i=1}^{n} \lambda X_i f   X_i + \lambda Y_i f  Y_i 
$$
 so that $\nabla_{\mathbb{H}} f  \neq 0 , \ f \in C_{\mathbb{H}}^1(U,\mathbb{R})$, with $U\subseteq \mathbb{H}^n$ open. As before, $\lambda$ is simply a normalising factor that, from now on, we ignore.\\
Note that $n_{\mathbb{H}} \in C_{\mathbb{H}}^1(U,\mathbb{R})$ (that is the same as asking $S$ to be $C_\mathbb{H}^2$-regular). Then define
$$
\begin{cases}
n_{E,2n+1}:=\frac{1}{n} \sum_{j=1}^{n} \left ( X_j n_{\mathbb{H},n+j} - Y_j n_{\mathbb{H},j} \right ), \\
n_{E,i} := n_{\mathbb{H},i} +\frac{1}{2}y_i \cdot n_{E,2n+1}, \\
n_{E,n+i} := n_{\mathbb{H},n+i} -\frac{1}{2}x_i \cdot n_{E,2n+1} ,
\end{cases}
\quad  i=1,\dots,n.
$$
Locally (for each point $p$ there exists a neighbourhood $U$ where such $f$ is defined) we can write the above as:
$$
n_{E,2n+1}
= \frac{1}{n}  \sum_{j=1}^{n} \left (   X_j Y_j f - Y_j   X_j f \right ) 
=\frac{1}{n} n  Tf=  \partial_t f.
$$
So now we have that
$$
\begin{cases}
n_{E,2n+1}=   \partial_t f,\\
n_{E,i} =  \partial_{x_i}  f -\frac{1}{2} {y_i}  \partial_t  f  +\frac{1}{2} {y_i}  \partial_t  f=  \partial_{x_i}  f ,\\
n_{E,n+i} =  \partial_{y_i}  f +\frac{1}{2} {x_i}  \partial_t  f  -\frac{1}{2} {x_i} \partial_t  f =  \partial_{y_i}  f ,
\end{cases}
\quad  i=1,\dots,n.
$$
In order to verify the Euclidean-orientability, we have to show that $\nabla f \neq 0$ . \\
Note that $f \in C_{\mathbb{H}}^1(U,\mathbb{R})$ and, a priori, we do not know whether $f \in C^1(U,\mathbb{R})$. However, asking $n_{\mathbb{H}} \in C_{\mathbb{H}}^1(U,\mathbb{R})$ allows us to write $\partial_{x_i}, \partial_{y_i} $ and $ \partial_t$ using only $X_i, Y_i, n_{\mathbb{H},i} $  and  $ n_{\mathbb{H},n+i}$, which guarantees that  $\partial_{x_i} f, \partial_{y_i} f$ and $ \partial_t f$ are well defined.\\
Now, $\nabla f \neq 0$ if and only if
\begin{align*}
\sum_{i=1}^{n} \left ( (\partial_{x_i} f)^2 + (\partial_{y_i} f)^2 \right ) + (\partial_t f)^2 \neq 0 ,
\end{align*}
which is the same as
\begin{align*}
& \sum_{i=1}^{n} \left [ \left  ( X_i f + \frac{1}{2}y_i  T f \right  )^2 + \left ( Y_i f - \frac{1}{2}x_i T f \right   )^2 + \left ( T f \right )^2 \right ] \neq 0 .
\end{align*}
In the case  $Tf \neq 0$, we have that $\nabla f \neq 0$ immediately. In the case $Tf=0$, instead, we have that $\nabla f \neq 0$ if and only if
$$
 \sum_{i=1}^{n} \left [ \left ( X_i f \right )^2+ \left ( Y_i f  \right   )^2 \right ] \neq 0,
$$
which is true because $\nabla_{\mathbb{H}} f \neq 0$. This completes the cases and shows that there actually is a global vector field $n_E$ that is continuous (by hypotheses) and never zero. So the second implication of the proposition is true.
\end{proof}






\chapter{Appendices}
\renewcommand{\thesection}{\Alph{section}}


\section{Proof  of the Explicit Rumin Complex in $\mathbb{H}^2$}\label{computationH2}

In this section we show a simple but useful lemma and then prove Proposition \ref{exH2}.

\begin{obs}[Change of basis]\label{change}
While dealing with the equivalence classes of the complex, there are the circumstances in which we need to make a change of basis. In particular, in the case of $\frac{\Omega^2}{I^2}$, we already used in Observation \ref{obsH2} that $\{ dx_1 \wedge dy_1,  dx_2 \wedge dy_2 \}$ and $\{  dx_1 \wedge dy_1 + dx_2 \wedge dy_2,  dx_1 \wedge dy_1 - dx_2 \wedge dy_2 \}$ span the same subspace.
%
%
Consider
 $$
\begin{cases}
\xi=dx_1 \wedge d y_1,\\
\eta=dx_2 \wedge d y_2,
\end{cases}
\quad \text{ and } \quad
\begin{cases}
u=\frac{\xi+\eta}{\sqrt{2}} = \frac{dx_1 \wedge d y_1 +dx_2 \wedge d y_2}{\sqrt{2}},\\
v=\frac{\xi-\eta}{\sqrt{2}} = \frac{dx_1 \wedge d y_1 -dx_2 \wedge d y_2}{\sqrt{2}}.
\end{cases}
$$
Then denote $\mathcal{B}_{\xi, \eta}^{u,v}$ the linear transformation that sends forms with respect to the basis $\{\xi, \eta\}$ to forms with respect to the basis $\{u,v\}$. Then
$$
\mathcal{B}_{\xi, \eta}^{u,v}:=\left (
\begin{matrix}
\frac{\partial u}{\partial \xi } &  \frac{\partial u}{\partial \eta } \\
\frac{\partial v}{\partial \xi } & \frac{\partial v}{\partial \eta }
\end{matrix} \right )=
\left (
\begin{matrix}
\frac{1}{\sqrt{2} } &  \frac{1}{\sqrt{2} } \\
\frac{1}{\sqrt{2} } & \frac{-1}{\sqrt{2} }
\end{matrix} \right ).
$$
Consider now a form $\omega_{ \{\xi, \eta\} }$ with respect to the basis $\{\xi, \eta\}$:
$$
\omega_{ \{\xi, \eta\} } = f   dx_1 \wedge d y_1 + g dx_2 \wedge d y_2 = f \xi + g \eta = 
\left (   \begin{matrix}   f \\   g   \end{matrix} \right )_{  \{\xi, \eta\}  }.
$$
To write the same form with respect to the basis $\{u,v\}$ (call it  $\omega_{ \{u, v\} }$), we have
\begin{align*}
\omega_{ \{u, v\} } & =\mathcal{B}_{\xi, \eta}^{u,v} \omega_{ \{\xi, \eta\} }=
\mathcal{B}_{\xi, \eta}^{u,v}
 \left (   \begin{matrix}   f \\    g   \end{matrix} \right )_{  \{\xi, \eta\}  }=
\left (
\begin{matrix}
\frac{1}{\sqrt{2} } &  \frac{1}{\sqrt{2} } \\
\frac{1}{\sqrt{2} } & \frac{-1}{\sqrt{2} }
\end{matrix}
 \right ) \cdot
\left (  \begin{matrix}   f \\    g    \end{matrix} \right )_{  \{\xi, \eta\}  }\\
&=
\left (
\begin{matrix}
\frac{f+g}{\sqrt{2}}\\
\frac{f-g}{\sqrt{2}}
\end{matrix} 
\right )_{  \{ u,v \}  }
 =\frac{f+g}{\sqrt{2}}u + \frac{f-g}{\sqrt{2}} v\\
&=\frac{f+g}{2} (dx_1 \wedge d y_1 +dx_2 \wedge d y_2) + \frac{f-g}{2} (dx_1 \wedge d y_1 -dx_2 \wedge d y_2).
\end{align*}
Since the two bases span forms on the same space, we trivially have that $\omega_{ \{\xi, \eta\} } = \omega_{ \{u, v\} }  $.
\end{obs}

\begin{ex}\label{variables}
The following computation, with the same notations for $ \xi,\eta $ and for $u,v$ as above, will be used later. Consider
$$
\omega_{\xi,\eta} = (X_1 \alpha_3  - Y_1 \alpha_1 ) dx_1 \wedge dy_1 +( X_2 \alpha_4  -  Y_2 \alpha_2 ) dx_2 \wedge dy_2 .
$$
Then the same form can be rewritten as
\begin{align*}
 \omega_{ \{u, v\} } =& \frac{   X_1 \alpha_3  - Y_1 \alpha_1 +  X_2 \alpha_4  -  Y_2 \alpha_2   }{2} (dx_1 \wedge d y_1 +dx_2 \wedge d y_2) \\
&+   \frac{ X_1 \alpha_3 - Y_1 \alpha_1 -  X_2 \alpha_4  + Y_2 \alpha_2    }{2}  ( dx_1 \wedge dy_1 - dx_2 \wedge dy_2).
\end{align*}
\end{ex}


\begin{proof}[Proof of Proposition \ref{exH2}]
 First we will show the cases regarding $d_Q^{(1)}$, $d_Q^{(2)}$, $d_Q^{(3)}$ and $d_Q^{(4)}$. Finally we will discuss $D$.\\
By Observation \ref{df}, we immediately get that:
$$
d_Q^{(1)} f = [ X_1f dx_1 + X_2f dx_2 + Y_1f dy_1+ Y_2f dy_2]_{I^1}.
$$
Consider now $[\alpha]_{I^1}= [\alpha_1 dx_1 + \alpha_2 dx_2 + \alpha_3 dy_1+ \alpha_4 dy_2]_{I^1} \in \frac{\Omega^1}{I^1}$. In this case we first work apply the definition and obtain:
\begin{align*}
d_Q^{(2)} \left ( \left [\alpha \right ]_{I^1} \right ) = & d_Q \left ( \left [ \alpha_1 dx_1 + \alpha_2 dx_2 + \alpha_3 dy_1+ \alpha_4 dy_2 \right ]_{I^1} \right ) \\
= &\big [- X_2 \alpha_1 dx_1 \wedge dx_2 - Y_1 \alpha_1 dx_1 \wedge dy_1 -  Y_2 \alpha_1 dx_1 \wedge dy_2                  \\
&+  X_1 \alpha_2 dx_1 \wedge dx_2 - Y_1 \alpha_2 dx_2 \wedge dy_1 -  Y_2 \alpha_2 dx_2 \wedge dy_2                    \\
&+      X_1 \alpha_3 dx_1 \wedge dy_1 + X_2 \alpha_3 dx_2 \wedge dy_1 -  Y_2 \alpha_3 dy_1 \wedge dy_2                \\
&+   X_1 \alpha_4 dx_1 \wedge dy_2 + X_2 \alpha_4 dx_2 \wedge dy_2 +  Y_1 \alpha_4 dy_1 \wedge dy_2 \big  ]_{I^2} \\
= & \big [ ( X_1 \alpha_2  - X_2 \alpha_1 ) dx_1 \wedge dx_2 + ( X_2 \alpha_3   - Y_1 \alpha_2 ) dx_2 \wedge dy_1\\
&+ ( Y_1 \alpha_4   -  Y_2 \alpha_3  ) dy_1 \wedge dy_2  + (    X_1 \alpha_4  -  Y_2 \alpha_1 ) dx_1 \wedge dy_2 \\
&  + (X_1 \alpha_3      - Y_1 \alpha_1 ) dx_1 \wedge dy_1 +( X_2 \alpha_4  -  Y_2 \alpha_2 ) dx_2 \wedge dy_2 \big ]_{I^2} ,
\end{align*}
but this is not enough. Thanks to the equivalence class, the last line can be written differently using Example \ref{variables}, so we get
\begin{align*}
d_Q^{(2)} ([\alpha]_{I^1} ) 
= & \bigg [ ( X_1 \alpha_2  - X_2 \alpha_1 ) dx_1 \wedge dx_2 + ( X_2 \alpha_3   - Y_1 \alpha_2 ) dx_2 \wedge dy_1\\
&+ ( Y_1 \alpha_4   -  Y_2 \alpha_3  ) dy_1 \wedge dy_2  + (    X_1 \alpha_4  -  Y_2 \alpha_1 ) dx_1 \wedge dy_2 \\
& + \left ( \frac{ X_1 \alpha_3 - Y_1 \alpha_1 -  X_2 \alpha_4  + Y_2 \alpha_2    }{2} \right ) ( dx_1 \wedge dy_1 - dx_2 \wedge dy_2) \bigg ]_{I^2} .
\end{align*}
We proceed on considering $d_Q^{(4)}$. Consider $\alpha= \alpha_1 dx_1 \wedge  dx_2 \wedge \theta +\alpha_2 dx_1 \wedge dy_2 \wedge \theta + \alpha_3  dx_2 \wedge dy_1 \wedge \theta + \alpha_4  dy_1 \wedge dy_2 \wedge \theta + \alpha_5( dx_1 \wedge dy_1 \wedge \theta - dx_2 \wedge dy_2 \wedge \theta) \in J^3$.\\
Then  $d_Q^{(4)}$ becomes
\begin{align*}
d_Q^{(4)} \alpha =& d_Q^{(4)} \big (   \alpha_1 dx_1 \wedge  dx_2 \wedge \theta +\alpha_2 dx_1 \wedge dy_2 \wedge \theta + \alpha_3  dx_2 \wedge dy_1 \wedge \theta + \alpha_4  dy_1 \wedge dy_2 \wedge \theta  \\
&+\alpha_5( dx_1 \wedge dy_1 \wedge \theta - dx_2 \wedge dy_2 \wedge \theta) \big )\\
=& Y_1 \alpha_1 dx_1 \wedge  dx_2 \wedge dy_1 \wedge \theta +Y_2 \alpha_1 dx_1 \wedge  dx_2 \wedge dy_2 \wedge \theta\\
& -X_2 \alpha_2 dx_1 \wedge dx_2 \wedge dy_2 \wedge \theta -Y_1    \alpha_2 dx_1 \wedge dy_1 \wedge dy_2 \wedge \theta\\
&+ X_1 \alpha_3 dx_1 \wedge  dx_2 \wedge dy_1 \wedge \theta  +Y_2       \alpha_3  dx_2 \wedge dy_1 \wedge dy_2 \wedge \theta\\ 
& + X_1 \alpha_4  dx_1 \wedge dy_1 \wedge dy_2 \wedge \theta+ X_2 \alpha_4  dx_2 \wedge dy_1 \wedge dy_2 \wedge \theta \\
&- X_1\alpha_5 dx_1 \wedge dx_2 \wedge dy_2 \wedge \theta - X_2  \alpha_5  dx_1 \wedge dx_2 \wedge dy_1 \wedge \theta\\
&+Y_1 \alpha_5 dx_2 \wedge dy_1 \wedge dy_2 \wedge \theta + Y_2  \alpha_5  dx_1 \wedge dy_1 \wedge dy_2 \wedge \theta\\
&+ \alpha_5( - dx_1 \wedge dy_1 \wedge dx_2 \wedge dy_2 + dx_2 \wedge dy_2 \wedge dx_1 \wedge dy_1) \\
=& ( Y_1 \alpha_1  + X_1 \alpha_3 - X_2  \alpha_5)   dx_1 \wedge  dx_2 \wedge dy_1 \wedge \theta\\
&+( Y_2 \alpha_1  -X_2 \alpha_2 - X_1\alpha_5      )dx_1 \wedge  dx_2 \wedge dy_2 \wedge \theta\\
&+( -Y_1 \alpha_2 + X_1 \alpha_4   + Y_2  \alpha_5       )     dx_1 \wedge dy_1 \wedge dy_2 \wedge \theta\\
&+(  Y_2  \alpha_3  +  X_2    \alpha_4 +Y_1 \alpha_5 ) dx_2 \wedge dy_1 \wedge dy_2 \wedge \theta.
\end{align*}
The last of the easy cases is $d_Q^{(5)}$. Consider $\alpha = \alpha_1 dx_1 \wedge  dx_2 \wedge dy_1 \wedge \theta + \alpha_2 dx_1 \wedge dx_2 \wedge dy_2 \wedge \theta + \alpha_3 dx_1 \wedge  dy_1 \wedge dy_2 \wedge \theta + \alpha_4  dx_2 \wedge  dy_1 \wedge dy_2 \wedge \theta \in J^4.$
Then $d_Q^{(5)}$ gives us
$$
d_Q^{(5)} \alpha= ( -Y_2  \alpha_1 + Y_1  \alpha_2 - X_2  \alpha_3 + X_1  \alpha_4 )  dx_1 \wedge  dx_2 \wedge dy_1 \wedge dy_2 \wedge \theta.
$$
The final case is the study of the behaviour of $D$ and it is the one that requires more effort. Remember that by Observation \ref{df}:
$$
d g = X_1g dx_1 + Y_1 g dy_1 +X_2g dx_2 + Y_2 g dy_2 + Tg \theta.
$$
Consider a form $\alpha' \in {\prescript{}{}\bigwedge}^{2} \mathfrak{h}_1 $: 
\begin{align*}
\alpha'=& \alpha_1 dx_1 \wedge  dx_2 +\alpha_2 dx_1 \wedge dy_1 + \alpha_3  dx_1 \wedge dy_2 + \alpha_4  dx_2 \wedge dy_1  + \alpha_5 dx_2 \wedge dy_2 \\
&+\alpha_6 dy_1 \wedge dy_2.
\end{align*}
Now, proceding as in the case of Observation \ref{change}, we have two bases $\{ \xi=dx_1 \wedge d y_1,  \eta=dx_2 \wedge d y_2 \}$ and $\{  u=\frac{\xi+\eta}{\sqrt{2}},  v=\frac{\xi-\eta}{\sqrt{2}}  \}$ that span both the same space.\\
Consider $\omega_{ \{ \xi,\eta \} }= \alpha_2 dx_1 \wedge dy_1  + \alpha_5 dx_2 \wedge dy_2$. In the second basis the same form is 
\begin{align*}
 \omega_{ \{u, v\} } = \frac{    \alpha_2+  \alpha_5  }{2} (dx_1 \wedge d y_1 +dx_2 \wedge d y_2) +   \frac{ \alpha_2-  \alpha_5   }{2}  ( dx_1 \wedge dy_1 - dx_2 \wedge dy_2).
\end{align*}
Denote $\gamma = \frac{    \alpha_2+  \alpha_5  }{2}$ and $\beta=  \frac{ \alpha_2-  \alpha_5   }{2} $. So we have
\begin{align*}
 \omega_{ \{u, v\} } = \gamma (dx_1 \wedge d y_1 +dx_2 \wedge d y_2) +  \beta  ( dx_1 \wedge dy_1 - dx_2 \wedge dy_2).
\end{align*}
So, by the definition of $\frac{\Omega^2}{I^2}$, one gets
\begin{align*}
[\alpha']_{I^2}=& [\alpha_1 dx_1 \wedge  dx_2 + \alpha_3  dx_1 \wedge dy_2 + \alpha_4  dx_2 \wedge dy_1  +\alpha_6 dy_1 \wedge dy_2\\
& 
+ \beta (dx_1 \wedge dy_1 - dx_2 \wedge dy_2)    ]_{I^2}.
\end{align*}
Then call $\alpha$ this new representative of the class $[\alpha']_{I^2}$,
\begin{align*}
\alpha :=& \alpha_1 dx_1 \wedge  dx_2 + \alpha_3  dx_1 \wedge dy_2 + \alpha_4  dx_2 \wedge dy_1  +\alpha_6 dy_1 \wedge dy_2\\
& 
+ \beta (dx_1 \wedge dy_1 - dx_2 \wedge dy_2)  ,
\end{align*}
and one obviously has that $[\alpha']_{I^2}=[\alpha]_{I^2}.$ Then also $D([\alpha']_{I^2})=D([\alpha]_{I^2})$ and from now on we will compute the latter.\\
 Remember that $D([\alpha]_{I^2}) =d \left ( \alpha +  L^{-1} ( - (d\alpha)_{\vert_{{\prescript{}{}\bigwedge}^{k} \mathfrak{h}_1 }} ) \wedge \theta  \right )$. The full derivative of $\alpha$ is:
\begin{align*}
d \alpha  = & d \big (  \alpha_1 dx_1 \wedge  dx_2  + \alpha_3  dx_1 \wedge dy_2 + \alpha_4  dx_2 \wedge dy_1  +\alpha_6 dy_1 \wedge dy_2\\
&
 + \beta (dx_1 \wedge dy_1 - dx_2 \wedge dy_2) \big ) \\
=&  Y_1   \alpha_1 dy_1 \wedge dx_1 \wedge  dx_2 + Y_2  \alpha_1 dy_2 \wedge dx_1 \wedge  dx_2 + T \alpha_1 \theta \wedge dx_1 \wedge  dx_2 \\
&+ Y_1   \alpha_3 dy_1 \wedge  dx_1 \wedge dy_2 +X_2 \alpha_3 dx_2 \wedge  dx_1 \wedge dy_2  + T \alpha_3 \theta \wedge  dx_1 \wedge dy_2\\
&+ X_1 \alpha_4 dx_1  \wedge  dx_2 \wedge dy_1 + Y_2  \alpha_4 dy_2 \wedge  dx_2 \wedge dy_1 + T \alpha_4 \theta \wedge  dx_2 \wedge dy_1\\
&+ X_1 \alpha_6 dx_1  \wedge  dy_1 \wedge dy_2+X_2 \alpha_6 dx_2 \wedge  dy_1 \wedge dy_2 +  T \alpha_6 \theta  \wedge  dy_1 \wedge dy_2\\
&- X_1 \beta dx_1 \wedge  dx_2 \wedge dy_2 - Y_1 \beta dy_1  \wedge  dx_2 \wedge dy_2 \\
&+X_2 \beta dx_2  \wedge  dx_1 \wedge dy_1 + Y_2 \beta dy_2  \wedge  dx_1 \wedge dy_1 \\
&+T \beta (dx_1 \wedge dy_1  \wedge \theta - dx_2 \wedge dy_2 \wedge \theta )\\
=&(Y_1 \alpha_1 - X_2 \beta + X_1 \alpha_4 )   dx_1 \wedge  dx_2  \wedge dy_1 +(Y_2 \alpha_1 - X_2 \alpha_3 - X_1 \beta )  dx_1 \wedge  dx_2  \wedge dy_2 \\
&+(Y_2 \beta - Y_1 \alpha_3 + X_1 \alpha_6 )  dx_1 \wedge  dy_1  \wedge dy_2 +(Y_2 \alpha_4 + Y_1 \beta + X_2 \alpha_6 )  dx_2 \wedge  dy_1  \wedge dy_2 \\
&+ T \alpha_1 \theta \wedge dx_1 \wedge  dx_2 +  T \alpha_3 \theta \wedge  dx_1 \wedge dy_2 +  T \alpha_4 \theta \wedge  dx_2 \wedge dy_1+  T \alpha_6 \theta  \wedge  dy_1 \wedge dy_2\\
&+T \beta (dx_1 \wedge dy_1  \wedge \theta - dx_2 \wedge dy_2 \wedge \theta ).
\end{align*}
Then we need to isolate the part belonging to ${\prescript{}{}\bigwedge}^{k} \mathfrak{h}_1$, meaning:
\begin{align*}
(d\alpha)_{\vert_{{\prescript{}{}\bigwedge}^{k} \mathfrak{h}_1}} =&(Y_1 \alpha_1 - X_2 \beta + X_1 \alpha_4 )   dx_1 \wedge  dx_2  \wedge dy_1 +(Y_2 \alpha_1 - X_2 \alpha_3 - X_1 \beta )  dx_1 \wedge  dx_2  \wedge dy_2 \\
&+(Y_2 \beta - Y_1 \alpha_3 + X_1 \alpha_6 )  dx_1 \wedge  dy_1  \wedge dy_2 +(Y_2 \alpha_4 + Y_1 \beta + X_2 \alpha_6 )  dx_2 \wedge  dy_1  \wedge dy_2 \\
=&(Y_1 \alpha_1 - X_2 \beta + X_1 \alpha_4 )   dx_2 \wedge  d \theta -(Y_2 \alpha_1 - X_2 \alpha_3 - X_1 \beta )  dx_1 \wedge  d \theta \\
&-(Y_2 \beta - Y_1 \alpha_3 + X_1 \alpha_6 )  dy_2 \wedge d \theta +(Y_2 \alpha_4+ Y_1 \beta + X_2 \alpha_6 )    dy_1  \wedge d \theta.
\end{align*}
Next we have that
\begin{align*}
L^{-1} ( - (d\alpha)_{\vert_{{\prescript{}{}\bigwedge}^{k} \mathfrak{h}_1 }} ) =&- (Y_1 \alpha_1 - X_2 \beta + X_1 \alpha_4 )   dx_2 +(Y_2 \alpha_1 - X_2 \alpha_3 - X_1 \beta )  dx_1  \\
&+(Y_2 \beta - Y_1 \alpha_3 + X_1 \alpha_6 )  dy_2 -(Y_2 \alpha_4+ Y_1 \beta + X_2 \alpha_6 )    dy_1,
\end{align*}
and so
\begin{align*}
d& (L^{-1} ( - (d\alpha)_{\vert_{{\prescript{}{}\bigwedge}^{k} \mathfrak{h}_1 }}  ) ) \wedge \theta=\\ 
=&
- X_1 (Y_1 \alpha_1 - X_2 \beta + X_1 \alpha_4 ) dx_1 \wedge  dx_2 \wedge \theta 
+Y_1 (Y_1 \alpha_1 - X_2 \beta + X_1 \alpha_4 )  dx_2 \wedge  dy_1 \wedge \theta\\
&
 +Y_2(Y_1 \alpha_1 - X_2 \beta + X_1 \alpha_4 )   dx_2 \wedge  dy_2 \wedge \theta
-X_2 (Y_2 \alpha_1 - X_2 \alpha_3 - X_1 \beta )  dx_1  \wedge dx_2 \wedge \theta\\
&
-Y_1 (Y_2 \alpha_1 - X_2 \alpha_3 - X_1 \beta )  dx_1  \wedge dy_1 \wedge \theta
-Y_2(Y_2 \alpha_1 - X_2 \alpha_3 - X_1 \beta )  dx_1  \wedge dy_2 \wedge \theta\\
&
+X_1 (Y_2 \beta - Y_1 \alpha_3 + X_1 \alpha_6 ) dx_1 \wedge  dy_2\wedge \theta
+X_2 (Y_2 \beta - Y_1 \alpha_3 + X_1 \alpha_6 )  dx_2 \wedge dy_2\wedge \theta\\
&
+Y_1 (Y_2 \beta - Y_1 \alpha_3 + X_1 \alpha_6 )  dy_1 \wedge dy_2\wedge \theta
-X_1 (Y_2 \alpha_4+ Y_1 \beta + X_2 \alpha_6 )  dx_1 \wedge  dy_1 \wedge \theta\\
&
-X_2 (Y_2 \alpha_4+ Y_1 \beta + X_2 \alpha_6 )  dx_2 \wedge  dy_1 \wedge \theta
+Y_2 (Y_2 \alpha_4+ Y_1 \beta + X_2 \alpha_6 )  dy_1 \wedge dy_2 \wedge \theta.
\end{align*}
Finally we can put all the pieces together and compute $D([\alpha]_{I^2}) =d \left ( \alpha +  L^{-1} ( - (d\alpha)_{\vert_{{\prescript{}{}\bigwedge}^{k} \mathfrak{h}_1 }} ) \wedge \theta  \right ) =
d  \alpha + d \left ( L^{-1} ( - (d\alpha)_{\vert_{{\prescript{}{}\bigwedge}^{k} \mathfrak{h}_1 }} ) \right ) \wedge \theta  + L^{-1} ( - (d\alpha)_{\vert_{{\prescript{}{}\bigwedge}^{k} \mathfrak{h}_1 }} )   \wedge d \theta $. \\
First notice that part of $d \alpha$ totally cancels $L^{-1} ( - (d\alpha)_{\vert_{{\prescript{}{}\bigwedge}^{k} \mathfrak{h}_1 }} )   \wedge d \theta $. Then we have
\begin{align*}
D&([\alpha]_{I^2}) =\\
=& T \alpha_1  dx_1 \wedge  dx_2 \wedge \theta+  T \alpha_3  dx_1 \wedge dy_2 \wedge \theta +  T \alpha_4   dx_2 \wedge dy_1 \wedge \theta + T \alpha_6  dy_1 \wedge dy_2 \wedge \theta \\
&+T \beta (dx_1 \wedge dy_1  \wedge \theta - dx_2 \wedge dy_2 \wedge \theta ) \\
& - X_1 (Y_1 \alpha_1 - X_2 \beta + X_1 \alpha_4 ) dx_1 \wedge  dx_2 \wedge \theta +Y_1 (Y_1 \alpha_1 - X_2 \beta + X_1 \alpha_4 )  dx_2 \wedge  dy_1 \wedge \theta\\
&+ Y_2(Y_1 \alpha_1 - X_2 \beta + X_1 \alpha_4 )   dx_2 \wedge  dy_2 \wedge \theta -X_2 (Y_2 \alpha_1 - X_2 \alpha_3 - X_1 \beta )  dx_1  \wedge dx_2 \wedge \theta\\
&-Y_1 (Y_2 \alpha_1 - X_2 \alpha_3 - X_1 \beta )  dx_1  \wedge dy_1 \wedge \theta -Y_2(Y_2 \alpha_1 - X_2 \alpha_3 - X_1 \beta )  dx_1  \wedge dy_2 \wedge \theta\\
&+X_1 (Y_2 \beta - Y_1 \alpha_3 + X_1 \alpha_6 ) dx_1 \wedge  dy_2\wedge \theta +X_2 (Y_2 \beta - Y_1 \alpha_3 + X_1 \alpha_6 )  dx_2 \wedge dy_2\wedge \theta\\
&+Y_1 (Y_2 \beta - Y_1 \alpha_3 + X_1 \alpha_6 )  dy_1 \wedge dy_2\wedge \theta -X_1 (Y_2 \alpha_4+ Y_1 \beta + X_2 \alpha_6 )  dx_1 \wedge  dy_1 \wedge \theta\\
&-X_2 (Y_2 \alpha_4+ Y_1 \beta + X_2 \alpha_6 )  dx_2 \wedge  dy_1 \wedge \theta +Y_2 (Y_2 \alpha_4+ Y_1 \beta + X_2 \alpha_6 )  dy_1 \wedge dy_2 \wedge \theta\\\\
=&( T \alpha_1 - X_1 Y_1 \alpha_1 +X_1 X_2 \beta -X_1 X_1 \alpha_4   -X_2 Y_2 \alpha_1 +X_2 X_2 \alpha_3 + X_2 X_1 \beta ) dx_1 \wedge  dx_2 \wedge \theta \\
&+( T \alpha_3   -Y_2 Y_2 \alpha_1+ Y_2 X_2 \alpha_3 + Y_2 X_1 \beta  +X_1 Y_2 \beta -X_1 Y_1 \alpha_3 +X_1 X_1 \alpha_6 )   dx_1 \wedge dy_2 \wedge \theta\\
&+ (T \alpha_4   +Y_1 Y_1 \alpha_1 - Y_1 X_2 \beta +Y_1 X_1 \alpha_4 -X_2 Y_2 \alpha_4-X_2 Y_1 \beta -X_2 X_2 \alpha_6 ) dx_2 \wedge dy_1 \wedge \theta \\
&+(T \alpha_6 + Y_1 Y_2 \beta -Y_1 Y_1 \alpha_3 +Y_1 X_1 \alpha_6 +Y_2 Y_2 \alpha_4+Y_2 Y_1 \beta + Y_2 X_2 \alpha_6 ) dy_1 \wedge dy_2 \wedge \theta\\
&+T \beta (dx_1 \wedge dy_1  \wedge \theta - dx_2 \wedge dy_2 \wedge \theta )\\
&+ (Y_2Y_1 \alpha_1 - Y_2 X_2 \beta +  Y_2X_1 \alpha_4  +X_2 Y_2 \beta -X_2 Y_1 \alpha_3 +X_2 X_1 \alpha_6 )  
dx_2 \wedge dy_2\wedge \theta\\
&+(-Y_1 Y_2 \alpha_1 +Y_1 X_2 \alpha_3 +Y_1 X_1 \beta -X_1 Y_2 \alpha_4 -X_1 Y_1 \beta  -X_1 X_2 \alpha_6 )
  dx_1 \wedge  dy_1 \wedge \theta\\\\
=&\left [ (  - X_1 Y_1 -Y_2 X_2) \alpha_1  +X_2 X_2 \alpha_3  -X_1 X_1 \alpha_4   +2X_1 X_2 \beta  \right  ]
 dx_1 \wedge  dx_2 \wedge \theta \\
&
+\left [
 -Y_2 Y_2 \alpha_1
+(X_2 Y_2 \alpha_3 -X_1 Y_1 )\alpha_3
 +X_1 X_1 \alpha_6  
+2X_1 Y_2 \beta 
 \right  ]   
dx_1 \wedge dy_2 \wedge \theta\\
&
+\left [
+Y_1 Y_1 \alpha_1    
 +(Y_1 X_1  -Y_2 X_2) \alpha_4
 -X_2 X_2 \alpha_6 
-2 X_2 Y_1 \beta
 \right  ]  
dx_2 \wedge dy_1 \wedge \theta \\
&
+\left [
 -Y_1 Y_1 \alpha_3 
+Y_2 Y_2 \alpha_4
+(Y_1 X_1 + X_2 Y_2) \alpha_6 
 + 2Y_1 Y_2 \beta
 \right  ]  
  dy_1 \wedge dy_2 \wedge \theta\\
&
+\left [
-Y_1 Y_2 \alpha_1 
+Y_1 X_2 \alpha_3 
-X_1 Y_2 \alpha_4 
 -X_1 X_2 \alpha_6 
 \right  ]  
 ( dx_1 \wedge  dy_1 \wedge \theta
-dx_2 \wedge dy_2\wedge \theta).
\end{align*}
\end{proof}


\section{General Rumin Differential Operator $d_c$ in $\mathbb{H}^1$ and $\mathbb{H}^2$}\label{A}

In the second chapter we introduced the Rumin complex by defining three different types of operators ($d_Q^{(k)}$ for $k<n$ and $k>n$, $D$ for $k=n$). It turns out that such complex 
 can be written as one general Rumin operator $d_c$. Here we show $d_c$ specifically for the first and second Heisenberg groups. We follow the presentation of \cite{TRIP} but giving here the minimum necessary amount of details.

\begin{defin}[See 11.22 in \cite{TRIP}]
Recall Remark \ref{Hcarnot}, Definition \ref{dual_basis} and Notation \ref{dfnota}. Let $\alpha \in \Omega^1, \ \alpha \neq 0$, have \emph{pure weight} $p$ if its dual vector field belongs to the $p$-layer $\mathfrak{h}_p$ of the algebra $\mathfrak{h}$. In this case write $w(\alpha)=p$.\\
Let $\beta \in \Omega^k, \ \beta \neq 0$, have \emph{pure weight} $p$ if $\beta$ can be expressed as a linear combination of $k$-forms $\theta_{i_1} \wedge \dots \wedge \theta_{i_k}$ such that $w(\theta_{i_1}) + \dots + w(\theta_{i_k})=p$ for all such forms.\\
Denote $\Omega^{k,p}$ the span of $k$-forms of pure weight $p$.
\end{defin}

\begin{ex}
In the Heisenberg group $\mathbb{H}^n$
\begin{align*}
&w(f)=0, \quad  f \in \Omega^0,\\
&w(dx_j)=w(dy_j)=1,   \quad j=1,\dots,n,\\
&w(\theta)=2,\\
&w(dx_1 \wedge dx_2)=2,\\
&w(d \theta)= w \left (  - \sum_{j=1}^n   dx_j \wedge dx_{n+j} \right )=2,\\
&w(dx_1 \wedge \theta)=3,\\
&w(dx_1 \wedge \dots \wedge dx_n \wedge dy_1 \wedge \dots \wedge dy_n \wedge \theta)=2n+2.
\end{align*}
\end{ex}

\noindent
Recall Notation \ref{dfnota} and notice that all differential $k$-forms  $  \{  \theta_{i_1} \wedge \dots \wedge \theta_{i_k} \}_{1\leq i_1 \leq \dots \leq i_k \leq 2n+1}$ are left-invariant. Denote such $k$-forms as $\theta_i^k $, $i \in \mathbb{N}$.\\
Furthermore, any other differential form with constant coefficients is left-invariant.

\begin{obs}[See 11.25 in \cite{TRIP}]\label{samew}
Let $\alpha \in \Omega^{k,p}$ be a left-invariant $k$-form of pure weight $p$ such that $d \alpha \neq 0$, then $w(d\alpha)=w(\alpha)$.
\end{obs}

\begin{obs}[See page 90 in \cite{TRIP}]
A general form $\alpha\in \Omega^{k,p}$ can be expressed as
$$
\alpha =\sum_{i} f_i \theta_i^k, 
$$
where $\spn_{i}  \{ \theta_i^k \} = \Omega^{k,p}$. Then its exterior differential is
$$
d\alpha =\sum_{i} d \left ( f_i \theta_i^k \right ) =  \sum_{i} \left ( \sum_{j=1}^{2n+1} W_j f_i \theta_j \wedge     \theta_i^k     +  f_i d \theta_i^k     \right )  .
$$
This shows that, whenever the terms are not zero,
\begin{align*}
 w \left ( \sum_{j=1}^{2n} W_j f_i \theta_j \wedge     \theta_i^k \right ) &=p+1,\\
 w \left (   W_{2n+1} f_i \theta_{2n+1} \wedge     \theta_i^k \right ) &=p+2,\\
w(f_i d \theta_i^k) &=p.
\end{align*}
In particular, the last equality holds by Observation \ref{samew}.
\end{obs}

\noindent
This provides the well-posedness for the following definition.

\begin{defin}[See definition 11.26 in \cite{TRIP}]\label{defind0}
Consider $\alpha \in \Omega^{k,p}$, then we can write:
$$
d \alpha = d_0 \alpha + d_1 \alpha + d_2 \alpha,
$$
where $d_i \alpha$ denotes the part of $d \alpha$ which increases the weight by $i$.\\
In particular $d_0 \alpha$ denotes the part of $d \alpha$ which does not increase the weight of the form. Then we can define
\begin{align*}
d_0 : \Omega^{k,p} \to \Omega^{k+1,p}, \   \sum_{i} f_i \theta_i^k \mapsto  \sum_{i}  f_i d \theta_i^k,
\end{align*}
whenever $d \theta_i^k \neq 0$.\\
One can easily check that the only form that, after a differentiation, does not change its weight is $\theta$. Therefore, considering a simple differential $k$-form $\alpha$ of the kind 
$$
\alpha = f \theta_{i_1} \wedge \dots \wedge \theta_{i_{k}},
$$
for $1 \leq i_1 \leq \dots \leq i_{k} \leq 2n$, one always has
$$
d_0 (\alpha)=0.
$$
On the other hand, considering a simple differential $k$-form 
$$
\alpha = f \theta_{i_1} \wedge \dots \wedge \theta_{i_{k-1}} \wedge \theta,
$$
for $1 \leq i_1 \leq \dots \leq i_{k-1} \leq 2n$, one has that
$$
d_0(\alpha) = f \theta_{i_1} \wedge \dots \wedge \theta_{i_{k-1}} \wedge d \theta.
$$
\end{defin}

\begin{defin}[11.32 in \cite{TRIP}]
Thanks to the isomorphism in definition 11.32 in \cite{TRIP}
\begin{equation}\label{isom}
d_{0_{\vert_{ \frac{\Omega^k}{\Ker d_0}}}} :   \frac{\Omega^k}{\Ker d_0}  \stackrel{\cong}{\to} \Ima d_0,
\end{equation}
one has that, for all $\beta  \in \Omega^k $, there exists  a unique $\alpha \in \Omega^k$, $\alpha \bot \Ker d_0$, such that $d_0=\beta+\xi$, with $\xi \in (\Ima d_0)^\bot$. 
So one can define $d_0^{-1}$ as
\begin{align*}
d_0^{-1}:  \Omega^{k+1} \to (\Ker d_0)^\bot, \ \beta \mapsto d_0^{-1} \beta := \alpha
\end{align*}
\end{defin}
and another operator
$$
D := d_0^{-1} (d-d_0).
$$
which is a nilpotent operator not to be confused with the second-order differential operator in the Rumin complex. Remark 11.33.1 in \cite{TRIP} says that  $D$ is nilpotent, meaning that there exists $N\in \mathbb{N}$ such that  $D^N\equiv 0$. Call $k_{max}$ the maximum non trivial exponent for $D$. Then the following operators are well-defined:
\begin{align*}
P &:=\sum_{k=0}^{k_{max}} (-D)^k,\\
Q &:=P d_0^{-1},\\ 
{\prescript{}{}\prod}_{E} &:= \text{Id} -Qd- dQ,\\
{\prescript{}{}\prod}_{\mathcal{E}_0} &:= \text{Id} - d_0^{-1} d_0- d_0 d_0^{-1}.
\end{align*}

\begin{prop}[See theorem 11.40 of \cite{TRIP}]
The Rumin complex of the Heisenberg group $\mathbb{H}^n$ can be written as:
$$
0 \to \mathbb{R} \to C^\infty  \stackrel{d_c^{(0)}}{\to} \mathcal{E}_0^1  \stackrel{d_c^{(1)}}{\to}  \dots \stackrel{d_c^{(n-1)}}{\to} \mathcal{E}_0^n \stackrel{d_c^{(n)}}{\to}\mathcal{E}_0^{n+1}    \stackrel{d_c^{(n+1)}}{\to} \dots   \stackrel{d_c^{(2n)}}{\to} \mathcal{E}_0^{2n+1} \to 0,
$$
$$
\text{i.e.,} \quad  ( \mathcal{E}_0^*, d_c),
$$
where
$$
\mathcal{E}_0^* = \bigoplus_{k=1}^{k_{max}} \mathcal{E}_0^k ,\quad  \mathcal{E}_0^k := \Ker d_0^{(k)} \cap \left ( \Ima d_0^{(k-1)} \right )^\bot  \quad  \text{and} \quad  d_c:= {\prescript{}{}\prod}_{\mathcal{E}_0} d {\prescript{}{}\prod}_E. 
$$
\end{prop}

\noindent
We show how these construction acts in practice, and when the operators just defined play an active role. In particular:

\begin{obs}\label{nullll}
Note that if, at some degree $k=0,\dots,2n$, $d_0 \equiv 0$, this means that all the other operators we just defined collapse to:
$$
\begin{cases}
D = 0,\\
P =0,\\
Q = 0,\\
{\prescript{}{}\prod}_{E} = \text{Id},\\
{\prescript{}{}\prod}_{\mathcal{E}_0} = \text{Id}.
\end{cases}
$$
In such case $d_c$, on its own defining domain, acts as $d$.
\end{obs}

\begin{obs}
In the case of the first Heisenberg group, $n=1$, we have three operators $d_c^{(k)}$, for $k\in \{0,1,2\}$, and so three corresponding operators $d_0^{(k)}$ with, by \eqref{isom},  
\begin{align*}
d_{0_{\vert_{ \frac{\Omega^k}{\Ker d_0^{(k)} }}}}^{(k)} :   \frac{\Omega^k}{\Ker d_0^{(k)} }  \stackrel{\cong}{\to} \Ima d_0^{(k)} .
\end{align*}
$\blacktriangleright$ If $k=0$, by Definition \ref{defind0} of $d_0$, it follows immediately that $\Ker d_0^{(0)}=\Omega^0$ and so $\frac{\Omega^0}{\Ker d_0^{(0)}}=\{ 0 \} $. Then $ d_0^{(0)} \equiv 0$.\\\\
$\blacktriangleright$ If $k=1$, again by Definition \ref{defind0}, we have that
$$
\Ker d_0^{(1)}=\spn \{dx,dy\}, \quad \text{and so} \quad \frac{\Omega^1}{\Ker d_0^{(1)}}=\spn \{ \theta \} \quad \text{and} \quad \Ima d_0^{(1)} = \spn  \{ dx \wedge dy \} .
$$
So we get
\begin{align*}
d_{0_{\vert_{ \frac{\Omega^1}{\Ker d_0^{(1)} }}}}^{(1)} :   \frac{\Omega^1}{\Ker d_0^{(1)} } & \stackrel{\cong}{\to} \Ima d_0^{(1)} , \\
\theta & \mapsto - dx \wedge dy.
\end{align*}
If we were to write $d_c^{(1)}$ explicitly with the operators introduced here, the final result would give back $D$ as in Proposition \ref{exH1}.\\\\
$\blacktriangleright$ If $k=2$ it is easy to check that $\Ker d_0^{(2)}=\Omega^2$ and so $\frac{\Omega^2}{\Ker d_0^{(2)}}=\{ 0 \}$. Then, again, $ d_0^{(2)} \equiv 0$.
\end{obs}


\begin{obs}
In the case of the second Heisenberg group, $n=2$, we have five operators $d_c^{(k)}$, for $k\in \{0,\dots,4\}$, with five corresponding operators $d_0^{(k)}$. Again, by \eqref{isom},  
\begin{align*}
d_{0_{\vert_{ \frac{\Omega^k}{\Ker d_0^{(k)} }}}}^{(k)} :   \frac{\Omega^k}{\Ker d_0^{(k)} }  \stackrel{\cong}{\to} \Ima d_0^{(k)} .
\end{align*}
$\blacktriangleright$ If $k=0$, as before, by Definition \ref{defind0} it follows that, $\Ker d_0^{(0)}=\Omega^0$ and so $\frac{\Omega^0}{\Ker d_0^{(0)}}=\{ 0 \} $. Then $ d_0^{(0)} \equiv 0$.\\\\
$\blacktriangleright$ If $k=1$,  take $\alpha = f_1 dx_1 + f_2 dx_2 + f_3 dy_1 + f_4 dy_2 + f_5 \theta \in \Omega^1$ and compute
$$
d_0 \alpha= - f_5 (dx_1 \wedge dy_1 + dx_2 \wedge dy_2).
$$
Consequently we have that $\Ker d_0^{(1)}=\spn \{dx_1, dx_2,dy_1,dy_2\}$, and so $ \frac{\Omega^1}{\Ker d_0^{(1)}}=\spn \{ \theta \} $ and $\Ima d_0^{(1)} = \spn  \{ dx_1 \wedge dy_1 + dx_2\wedge dy_2 \} .$ \\
The action of $d_0^{(1)}$ is then
\begin{align*}
d_{0_{\vert_{ \frac{\Omega^1}{\Ker d_0^{(1)} }}}}^{(1)} : \frac{\Omega^1}{\Ker d_0^{(1)}} & \to \Ima d_0^{(1)},\\
\theta & \mapsto -dx_1 \wedge dy_1 - dx_2 \wedge dy_2.
\end{align*}
Writing $d_c^{(1)}$ explicitly would give $d_Q^{(1)}$ as in Proposition \ref{exH2}.\\\\
$\blacktriangleright$ If $k=2$, one can compute that  $\Ker d_0=\spn \{ dx_1 \wedge dx_2,  dx_1 \wedge dy_1,  dx_1 \wedge dy_2,  dx_2 \wedge dy_1,  dx_2 \wedge dy_2, dy_1 \wedge dy_2   \}$, $ \frac{\Omega^2}{\Ker d_0}=\spn \{ dx_1 \wedge \theta, dy_1 \wedge \theta, dx_2 \wedge \theta, dy_2 \wedge \theta \}$ and $ \Ima d_0 = \spn  \{  dx_1 \wedge dx_2 \wedge dy_2, dy_1 \wedge dx_2 \wedge dy_2,  dx_2 \wedge dx_1 \wedge dy_1,  dy_2 \wedge dx_1 \wedge dy_1 \} \subseteq \Omega^3.$\\
The action of $d_0^{(2)}$ is then
\begin{align*}
d_{0_{\vert_{ \frac{\Omega^2}{\Ker d_0^{(2)} }}}}^{(2)} : \frac{\Omega^2}{\Ker d_0^{(2)}} & \to \Ima d_0^{(2)},\\
 dx_1 \wedge \theta  & \mapsto - dx_1 \wedge dx_2 \wedge dy_2,\\
 dy_1 \wedge \theta  & \mapsto - dy_1 \wedge dx_2 \wedge dy_2,\\ 
 dx_2 \wedge \theta  & \mapsto - dx_2 \wedge dx_1 \wedge dy_1,\\
 dy_2 \wedge \theta  & \mapsto - dy_2 \wedge dx_1 \wedge dy_1.
\end{align*}
Writing $d_c^{(2)}$ explicitly would give $D$ as in Proposition \ref{exH2}.\\\\
$\blacktriangleright$  If $k=3$, by computation one can see that $\Ker d_0=\spn \{ dx_1 \wedge dx_2  \wedge dy_1,  dx_1 \wedge dx_2 \wedge dy_2,  dx_1 \wedge dx_2 \wedge \theta,  dx_1 \wedge dy_1 \wedge  dy_2,  dx_1 \wedge dy_2 \wedge \theta, dx_2 \wedge  dy_1 \wedge dy_2,  dx_2 \wedge  dy_1 \wedge \theta, dy_1 \wedge  dy_2 \wedge \theta,   dx_1 \wedge  dy_1 \wedge \theta  dx_2 \wedge  dy_2 \wedge \theta  \}$,    $\frac{\Omega^3}{\Ker d_0}=\spn \{ dx_1 \wedge dy_1 \wedge \theta +  dx_2 \wedge dy_2 \wedge \theta \}$ and $\Ima d_0 = \spn  \{  dx_1 \wedge  dy_1 \wedge dx_2 \wedge dy_2 \} \subseteq \Omega^4.$\\
The action of $d_0^{(3)}$ is then
\begin{align*}
d_{0_{\vert_{ \frac{\Omega^3}{\Ker d_0^{(3)} }}}}^{(3)} : \frac{\Omega^3}{\Ker d_0^{(3)}} & \to \Ima d_0^{(3)},\\
 dx_1 \wedge dy_1 \wedge \theta +  dx_2 \wedge dy_2 \wedge \theta  & \mapsto -dx_1 \wedge dy_1 \wedge  dx_2 \wedge dy_2.
\end{align*}
Once again, also in this case writing $d_c^{(3)}$ explicitly would give $d_Q^{(3)}$  in Proposition \ref{exH2}.\\\\
$\blacktriangleright$  If $k=4$, we are in the last case. Take $\alpha = f_1 \widehat{dx_1} + f_2 \widehat{dx_2} + f_3 \widehat{dy_1}+ f_4 \widehat{dy_2}  + f_5 \widehat{\theta}$, where the hat indicates the presence of all other basis elements except the one written below the hat. Then 
$$
d_0 \alpha=0,
$$
and so $ \Ker d_0=\Omega^4$, $ \frac{\Omega^4}{\Ker d_0}=\{ 0 \}$ and   $ \Ima d_0 =  \{ 0 \} \subseteq \Omega^5$. Then, again, $ d_0^{(4)} \equiv 0$.
\end{obs}

\section{Dimension of the Rumin Complex in $\mathbb{H}^n$}\label{B}
We show formulas to compute the dimensions for the spaces in the Rumin complex. The intent is to see how fast they grow with respect to $n$.

\begin{rem}
Recall from Definition \ref{def_forms} that:
\begin{itemize}
\item $\Omega^k = k$-differential forms in $\mathbb{H}^{n}$
\item $I^k= \{ \beta \wedge \theta + \gamma \wedge d \theta \ / \  \beta \in \Omega^{k-1}, \ \gamma \in \Omega^{k-2}  \}$
\item $J^k=\{ \alpha \in \Omega^{k} \ / \  \alpha \wedge \theta =0, \  \alpha \wedge d\theta=0   \}$.
\end{itemize}
\end{rem}


\begin{obs}\label{C?}
\begin{itemize}
\item
It is an immediate observation that $\dim \Omega^k =  \binom{2n+1}{k}$.
\item
Modulo the coefficient, the possibilities for a simple differential form $\beta \in \Omega^{k-1}$ such that $\beta \wedge \theta \neq 0$ are $C_{k-1}^{2n} =  \binom{2n}{k-1}$.
\item
Modulo the coefficient, the possibilities for a simple differential form $\gamma  \in \Omega^{k-2}$ such that $\gamma \wedge d \theta \neq 0$ are $C_{k-2}^{2n+1} =  \binom{2n+1}{k-2}$. 
\end{itemize}
\end{obs}

\begin{proof}
The first point of the observation is straightforward. One has that $\beta \in \Omega^{k-1}$ and $\beta$ can't contain $\theta$, otherwise $\beta \wedge \theta$ would be null.\\
$ \gamma \in \Omega^{k-2}$ and $ \gamma$ can contain anything. Indeed, to have  $d \theta \wedge \gamma =0$, with $\gamma$ simple, one needs $\gamma \in \Omega^{2n}$ at  least, but $k-2$ is always lower than $2n$ for any $k$.
\end{proof}

\begin{obs}
To compute the dimension of $I^k$, it is important to know how many times and element of its base can be written both in the form $\beta \wedge \theta$ and $\gamma \wedge d \theta$, with $\beta \in \Omega^{k-1}$, $\beta \wedge \theta \neq 0$, $\gamma  \in \Omega^{k-2}$ and $\gamma \wedge d \theta \neq 0$.\\
With these hypotheses, $\beta$ can be written as $\beta = d \theta \wedge \tau$, with $\tau \in \Omega^3$; while $\gamma$ can be written as $\gamma = \theta \wedge \tau'$, with $\tau' \in \Omega^3$.\\
Posing
$$
\beta \wedge \theta=  \gamma \wedge d \theta \neq 0
$$
one gets $\tau = \tau' \in \Omega^3$, with $\tau \wedge d\theta \wedge \theta \neq 0$. This means that $\tau$ cannot contain $\theta$. Furthermore, this is the only condition since it is impossible for $\tau$ to annihilate $d \theta$ as $k-3$ is always lower than $2n$, as before.\\
Then, modulo coefficients, the possibilities for a simple $\tau$ are $C_{k-3}^{2n} =  \binom{2n}{k-3}$ . 
\end{obs}

\begin{obs} [See 2.3 and 2.5 in \cite{FSSC}] For $k \leq n$,
$$
 \dim \frac{\Omega^k}{I^k} = \dim J^{2n+1-k} .
$$
\end{obs}

\begin{prop}\label{22star}
Summing up the previous considerations, for $k \leq n$, we obtain
\begin{equation}\label{star}
\dim I^k =\binom{2n}{k-1} + \binom{2n+1}{k-2} -  \binom{2n}{k-3}  = \binom{2n+1}{k-1}
\end{equation}
\begin{equation}\label{2star}
 \dim J^{2n+1-k} = \dim \frac{\Omega^k}{I^k} = \binom{2n+1}{k} - \binom{2n+1}{k-1} = \binom{2n+1}{k}     \cdot  \frac{2n+2-2k}{2n+2-k}
\end{equation}
The proof follows after the example.
\end{prop}

\begin{ex}
Here some cases as an example.
\begin{center}
\begin{tabular}{c|c|c|c|c|}
 &  & $\dim \Omega^k$ & $\dim I^k$ & $     \dim \frac{\Omega^k}{ I^k} =   \dim J^{2n+1-k}$ \\ \hline
$\mathbb{H}^1 $ & $k=1$ & 3 & 1 & 2 \\ \arrayrulecolor{red}\specialrule{.1em}{.05em}{.05em} 
$\mathbb{H}^2 $ & $k=1$ & 5 & 1 & 4 \\ \arrayrulecolor{black}\hline
& $k=2$ & 10 & 5 & 5 \\ \arrayrulecolor{red}\specialrule{.1em}{.05em}{.05em} 
$\mathbb{H}^3$ & $k=1$ & 7 & 1 & 6 \\ \arrayrulecolor{black}\hline
 & $k=2$ & 21 & 7 & 14 \\ \arrayrulecolor{black}\hline
 & $k=3$ & 35 & 21 & 14 \\ \arrayrulecolor{red}\specialrule{.1em}{.05em}{.05em} 
$\mathbb{H}^4$ & $k=1$ & 9 & 1 & 8 \\ \arrayrulecolor{black}\hline
 & $k=2$ & 36 & 9 & 27 \\ \arrayrulecolor{black}\hline
 & $k=3$ & 84 & 36 & 48 \\ \arrayrulecolor{black}\hline
 & $k=4$ & 126 & 84 & 42 \\ \arrayrulecolor{red}\specialrule{.1em}{.05em}{.05em} 
$\mathbb{H}^5$ & $k=1$ & 11 & 1 & 10 \\ \arrayrulecolor{black}\hline
 & $k=2$ & 55 & 11 & 44 \\ \arrayrulecolor{black}\hline
 & $k=3$ & 165 & 55 & 110 \\ \arrayrulecolor{black}\hline
 & $k=4$ & 330 & 165 & 165 \\ \arrayrulecolor{black}\hline
 & $k=5$ & 462 & 330 & 132 \\ \arrayrulecolor{black}\hline
\end{tabular}
\end{center}
\end{ex}

\begin{proof}[Proof of Proposition \ref{22star}]
First we prove equation \eqref{star}. To do so, we rewrite the three terms as follows:
\begin{align*}
 \binom{2n}{k-1}   &=\frac{(2n)!}{(k-1)!(2n-(k-1))!} \cdot  \frac{2n+1}{2n+1} \cdot  \frac{2n-(k-1)+1}{2n-(k-1)+1}  \\
&
=\frac{(2n+1)!}{(k-1)!(2n+1-(k-1))!} \cdot  \frac{1}{2n+1} \cdot  (2n-(k-1)+1) \\
&
= \binom{2n+1}{k-1} \cdot  \frac{2n-k+2}{2n+1}.
\end{align*}
The second term:
\begin{align*}
 \binom{2n+1}{k-2}  &=\frac{(2n+1)!}{(k-2)!(2n+1-(k-2))!} \cdot  \frac{k-1}{k-1} \cdot  \frac{2n+1-(k-2)}{2n+1-(k-2)} \\
&
  =\frac{(2n+1)!}{(k-1)!(2n+1-(k-1))!} \cdot  (k-1) \cdot    \frac{1}{2n+1-(k-2)}  \\
&
= \binom{2n+1}{k-1}  \cdot  \frac{k-1}{2n-k+3}.
\end{align*}
And the last term:
\begin{align*}
 \binom{2n}{k-3}&= \frac{(2n)!}{(k-3)!(2n-(k-3))!} \cdot  \frac{2n+1}{2n+1}  \cdot  \frac{(k-2)(k-1)}{(k-2)(k-1)} \cdot  \frac{2n-(k-3)}{2n-(k-3)}  \\
&
 = \frac{(2n+1)!}{(k-1)!(2n-(k-3)-1)!} \cdot  \frac{1}{2n+1}  \cdot  (k-2)(k-1) \cdot  \frac{1}{2n-(k-3)}  \\
&
 = \frac{(2n+1)!}{(k-1)!(2n+1-(k-1))!} \cdot    \frac{(k-2)(k-1)}{(2n+1)(2n-(k-3))} \\
&
 =  \binom{2n+1}{k-1} \cdot    \frac{(k-2)(k-1)}{(2n+1)(2n-k+3)} .
\end{align*}
Now we can add and subtract them together and we obtain
$$
\binom{2n}{k-1} + \binom{2n+1}{k-2} -  \binom{2n}{k-3}
$$
$$
=\binom{2n+1}{k-1} \cdot \left [  
\frac{2n-k+2}{2n+1} +   \frac{k-1}{2n-k+3}  -   \frac{(k-2)(k-1)}{(2n+1)(2n-k+3)}
  \right ]
$$
$$
=\binom{2n+1}{k-1} \cdot  \left [  
\frac{     (2n-k+2) (2n-k+3) +  (k-1) (2n+1)      -(k-2)(k-1)}{(2n+1)(2n-k+3)}
  \right ]
$$
\centerline{
  \begin{minipage}{\linewidth}
    \begin{align*}
=\binom{2n+1}{k-1}  \cdot \left [  
\frac{    4n^2 -2n(k-3) + 2n-(k-3) - 2n(k-1) + k^2  -4k  +3          +2nk+k-2n-1               -k^2 +3k-2
}{(2n+1)(2n-k+3)}
  \right ]
    \end{align*}
\end{minipage}
}
$$
=
   \binom{2n+1}{k-1} \cdot  \left [  
\frac{    4n^2   +n( 8 -2k    )    +(3-k) 
}{(2n+1)(2n-k+3)}
  \right ]=\binom{2n+1}{k-1},
$$
whete the last equality comes from
\[
  \begin{array}{r|r}
 4n^2   +n( 8 -2k    )     +(3-k)  & 2n+1 \\ \cline{2-2}
    4n^2   +2n  \phantom{  8 -2k    )     +(3-k)  }   & 2n +3 -k \\ \cline{1-1} \\[\dimexpr-\normalbaselineskip+\jot]
 2n( 3 -k    )      +(3-k) \\
  2n( 3 -k    )      +(3-k) \\ \cline{1-1} \\[\dimexpr-\normalbaselineskip+\jot]
                     / /
  \end{array}
\]
\noindent
Now we prove  equation \eqref{2star}.
\begin{align*}
\binom{2n+1}{k} - \binom{2n+1}{k-1} &= \binom{2n+1}{k} -  \frac{(2n+1)!}{(k-1)!(2n+1-(k-1))!}=\\
&
= \binom{2n+1}{k} -  \frac{(2n+1)!}{(k-1)!(2n+1-(k-1))!} \cdot  \frac{k}{k} \cdot  \frac{2n + 1 - (k-1)}{2n + 1 - (k-1)}   \\
&
= \binom{2n+1}{k} -  \frac{(2n+1)!}{k!(2n+1-k)!} \cdot  k \cdot  \frac{1}{2n + 1 - (k-1)}   \\
&
= \binom{2n+1}{k}  \left [ 1 -   \frac{k}{2n+2-k} \right ] =   \binom{2n+1}{k} \cdot\frac{2n+2-2k}{2n+2-k} .
\end{align*}
\end{proof}



\section{Explicit Computation of the Commutation between Pullback and Rumin Complex}\label{explicitcommutation}

Let $f : \mathbb{H}^1 \to \mathbb{H}^1$ be a smooth contact map and consider the Rumin complex. The following Proposition is already contained in Theorem \ref{Fdc=dcF}, but here we show the result with an explicit computation.\\\\
We recall the result in $\mathbb{H}^1$:

\begin{prop}\label{commuH1}
Consider a contact map $f: U \subset \mathbb{H}^1 \to \mathbb{H}^1$. The following hold
$\blacktriangleright$ For all $ \omega = g \in \mathcal{D}_{\mathbb{H}}^0( U )= C^\infty ( U )$,
\begin{equation}\label{k=1}
(\Lambda^1 df) d_Q \omega = d_Q (\Lambda^0 df \omega), \quad  \text{ i.e.,} \quad  f^* d_Q \omega = d_Q f^* \omega.
\end{equation}
$\blacktriangleright$  For all $  \omega \in \mathcal{D}_{\mathbb{H}}^1(U) =\frac{\Omega^1}{I^1} \cong \spn  \{ dx , dy \}$.
\begin{equation}\label{k=2}
(\Lambda^1 df) D \omega = D (\Lambda^0 df \omega), \quad  \text{ i.e.,} \quad    f^* D \omega = D f^* \omega.
\end{equation}
$\blacktriangleright$  For all $ \omega \in \mathcal{D}_{\mathbb{H}}^2(U) = J^2 \cong  \spn  \{ dx \wedge \theta, dy \wedge \theta \}$,
\begin{equation}\label{k=3}
(\Lambda^3 df) d_Q \omega = d_Q (\Lambda^2 df \omega) , \quad  \text{ i.e.,} \quad    f^* d_Q \omega = d_Q f^* \omega.
\end{equation}
\noindent
Namely, the pullback by a contact map $f$ commutes with the differental operators of the Rumin complex.
\end{prop}


\begin{no}
Recalling Notation \ref{callL} and Lemma \ref{T3=XY12}, we have
$$
\lambda (1,f):= Xf^1  Yf^2  -  Xf^2 Yf^1 =  T f^3   - \frac{1}{2}f^1  Tf^2 + \frac{1}{2}f^2  Tf^1  .
$$
In this section, for brevity we will denote this simply as
$$
\lambda (f):=\lambda (1,f).
$$
\end{no}

\begin{proof}[ Proof of equation \eqref{k=1} in Proposition \ref{commuH1}]
From lemma \eqref{composition_contact}
\begin{align*}
d_Q (\Lambda^0 df (g)) =& d_Q f^* g = d_Q (g \circ f)  = X(g \circ f) dx + Y(g \circ f) dy \\
=& ( Xg   X f^1  + Yg X f^2 )dx +( Xg Y f^1     + Yg Y f^2 )dy.
\end{align*}

On the other hand
\begin{align*}
(\Lambda^1 df) d_Q g &= f^* d_Q g =  f^* (Xg dx + Yg dy) = Xg  f^*  dx + Yg f^* dy \\
&
= Xg ( Xf^1 dx + Y f^1 dy ) + Yg( Xf^2 dx + Y f^2 dy)\\
&
= ( Xg   X f^1  + Yg X f^2 )dx +( Xg Y f^1  + Yg Y f^2 )dy.
\end{align*}
So the first equality is verified.
\end{proof}

\noindent
We move on to the third equality.

\begin{proof}[Proof of equation \eqref{k=3} in Proposition \ref{commuH1}]
Consider $\omega \in \mathcal{D}_{\mathbb{H}}^2(\mathbb{H}^1)$, namely, $\omega = \omega_1 dx \wedge \theta + \omega_2 dy \wedge \theta$.
\begin{align*}
&(\Lambda^3 df) d_Q \omega = f^* d_Q \omega =f^* (  (X\omega_2-Y\omega_1)  dx \wedge dy \wedge \theta)=\\
&=(X\omega_2-Y\omega_1)_f  ( Xf^1  Yf^2 -  Yf^1  Xf^2) \left ( T f^3   - \frac{1}{2}f^1  Tf^2 + \frac{1}{2}f^2  Tf^1   \right ) dx \wedge dy \wedge \theta  \\
&=(X\omega_2-Y\omega_1)_f   \lambda^2 (f) dx \wedge dy \wedge \theta,
\end{align*}
where we use equation \eqref{3dim} for the second line.\\\\
On the other hand, thanks to equations \eqref{2dimx} and \eqref{2dimy},
\begin{align*}
 d_Q (\Lambda^2 df \omega) =& d_Q f^* \omega =d_Q f^*( \omega_1 dx \wedge \theta + \omega_2 dy \wedge \theta)\\
=&d_Q \left (  (\omega_1)_f  f^*(  dx \wedge \theta )   +  (\omega_2)_f  f^*(  dy \wedge \theta ) \right ) \\
=&d_Q \Big ( (\omega_1)_f  \left ( Xf^1 \lambda (f) dx \wedge \theta +  Y f^1 \lambda (f) dy \wedge \theta  \right )  \\
& +  (\omega_2)_f   \left ( Xf^2 \lambda (f) dx \wedge \theta +  Y f^2 \lambda (f) dy \wedge \theta \right )  \Big  ) \\
=&d_Q \Big ( \lambda (f) \left ( (\omega_1)_f  Xf^1 + (\omega_2)_f  Xf^2  \right ) dx \wedge \theta \\
&+\lambda (f) \left ( (\omega_1)_f  Yf^1 + (\omega_2)_f  Yf^2  \right ) dy \wedge \theta      \Big  )  \\
=&\Big ( X \left (\lambda (f) \left ( (\omega_1)_f  Yf^1 + (\omega_2)_f  Yf^2  \right ) \right) \\
&- Y \left (   \lambda (f) \left ( (\omega_1)_f  Xf^1 + (\omega_2)_f  Xf^2  \right )  \right )
    \Big  )  dx \wedge dy \wedge \theta\\
=& \left ( X \left (\lambda (f) \cdot \mathcal{E}_{\omega,f} \right) - Y \left (   \lambda (f) \cdot \mathcal{F}_{\omega,f} \right )   \right )   dx \wedge dy \wedge \theta\\
=& \Big ( X (\lambda (f) ) \cdot \mathcal{E}_{\omega,f} + \lambda (f) \cdot X(\mathcal{E}_{\omega,f})\\
& - Y (   \lambda (f) ) \cdot \mathcal{F}_{\omega,f} - \lambda (f) \cdot Y(\mathcal{F}_{\omega,f})    \Big )   dx \wedge dy \wedge \theta,
\end{align*}
where one denotes
$$
\begin{cases}
  \mathcal{E}_{\omega,f}:=\mathcal{E}(\omega,f):=(\omega_1)_f  Yf^1 + (\omega_2)_f  Yf^2,\\
\mathcal{F}_{\omega,f}:=\mathcal{F}(\omega,f):=(\omega_1)_f  Xf^1 + (\omega_2)_f  Xf^2.
\end{cases}
$$
\noindent
Consider these terms one by one. Recall that we already computed $X(\lambda (f))$ and $Y(\lambda (f))$ in Lemma  \ref{XCYC}. Then compute $X (\mathcal{E}_{\omega,f})$:
\begin{align*}
X (\mathcal{E}_{\omega,f})=&X \left (  (\omega_1)_f  Yf^1 + (\omega_2)_f  Yf^2  \right )\\
=&X  ((\omega_1)_f )  Yf^1 + (\omega_1)_f  XYf^1 + X ( (\omega_2)_f )  Yf^2  +  (\omega_2)_f  XYf^2\\
=&X  (\omega_1 \circ f )  Yf^1 + (\omega_1)_f  XYf^1 + X ( \omega_2 \circ f )  Yf^2  +  (\omega_2)_f  XYf^2\\
=& \left ( ( X\omega_1)_f   X f^1  + (Y\omega_1)_f  X f^2 \right )  Yf^1 + \left ( ( X\omega_2)_f   X f^1  + (Y\omega_2)_f  X f^2 \right ) Yf^2 \\
&
+ (\omega_1)_f  XYf^1 +   (\omega_2)_f  XYf^2,
\end{align*}
where the last line follows by Lemma \ref{composition_contact}.\\\\
The last term $Y ( \mathcal{F}_{\omega,f} ) $ is similar:
\begin{align*}
Y ( \mathcal{F}_{\omega,f} ) =& Y  \left (  (\omega_1)_f  Xf^1 + (\omega_2)_f  Xf^2 \right )\\
=&Y  ((\omega_1)_f ) Xf^1 + (\omega_1)_f  YXf^1 + Y ( (\omega_2)_f )  Xf^2  +  (\omega_2)_f  YXf^2\\
=&Y  (\omega_1 \circ f )  Xf^1 + (\omega_1)_f  YXf^1 + Y ( \omega_2 \circ f )  Xf^2  +  (\omega_2)_f  YXf^2\\
=& \left ( ( X \omega_1)_f  Y f^1  + (Y\omega_1)_f  Y f^2 \right )  Xf^1 + \left ( ( X\omega_2)_f   Y f^1  + (Y\omega_2)_f  Y f^2 \right ) Xf^2 \\
&
+ (\omega_1)_f  YXf^1 +   (\omega_2)_f  YX f^2,
\end{align*}
where the last line follows again by Lemma \ref{composition_contact}.\\\\
Now we can get $X (\mathcal{E}_{\omega,f}) - Y ( \mathcal{F}_{\omega,f} ) $. Remember that 
\begin{align*}
X (\mathcal{E}_{\omega,f}) - Y ( \mathcal{F}_{\omega,f} )  =&\left ( ( X\omega_1)_f   X f^1  + (Y\omega_1)_f  X f^2 \right )  Yf^1 + \left ( ( X\omega_2)_f   X f^1  + (Y\omega_2)_f  X f^2 \right ) Yf^2 \\
&+ (\omega_1)_f  XYf^1 +   (\omega_2)_f  XYf^2  -\left ( ( X \omega_1)_f  Y f^1  + (Y\omega_1)_f  Y f^2 \right )  Xf^1\\
& - \left ( ( X\omega_2)_f   Y f^1  + (Y\omega_2)_f  Y f^2 \right ) Xf^2 - (\omega_1)_f  YXf^1 - (\omega_2)_f  YX f^2\\
=& (\omega_1)_f  (XYf^1 -  YXf^1) + (\omega_2)_f  (XYf^2- YX f^2)\\
&+ ( X\omega_1)_f  ( X f^1  Yf^1 -Y f^1  Xf^1 ) + (Y\omega_1)_f  ( X f^2   Yf^1 -  Y f^2   Xf^1   )\\
&+( X\omega_2)_f  (  X f^1 Yf^2 -  Y f^1 Xf^2 )+ (Y\omega_2)_f  (  X f^2  Yf^2   -    Y f^2  Xf^2 )\\
=& (\omega_1)_f  Tf^1 + (\omega_2)_f T f^2+ \lambda (f) \left ( ( X\omega_2)_f - (Y\omega_1)_f \right  ).
\end{align*}
Next is $X (\lambda (f) ) \cdot \mathcal{E}_{\omega,f} - Y (   \lambda (f) ) \cdot \mathcal{F}_{\omega,f}$:
\begin{align*}
X (\lambda (f) ) \cdot \mathcal{E}_{\omega,f} & - Y (   \lambda (f) ) \cdot \mathcal{F}_{\omega,f}=\\
=& ( Xf^2  Tf^1 -Tf^2Xf^1 ) \left ( (\omega_1)_f  Yf^1 + (\omega_2)_f  Yf^2 \right )\\
&
 -( Yf^2  Tf^1 -Tf^2 Yf^1 ) \left (  (\omega_1)_f  Xf^1 + (\omega_2)_f  Xf^2 \right )\\
=&(\omega_1)_f \left (  Xf^2  Tf^1 Yf^1 - Tf^2 Xf^1 Yf^1  -    Yf^2  Tf^1  Xf^1+ Tf^2 Yf^1 Xf^1    \right ) \\
&
+(\omega_2)_f \left (  Xf^2  Tf^1 Yf^2 - Tf^2 Xf^1 Yf^2    -Yf^2  Tf^1 Xf^2+ Tf^2 Yf^1  Xf^2   \right )\\
=& (\omega_1)_f Tf^1 \left (   Xf^2  Yf^1 -    Yf^2    Xf^1    \right ) +(\omega_2)_f  Tf^2  \left (    Yf^1  Xf^2  -  Xf^1 Yf^2    \right )\\
=& - \lambda (f) \left (   (\omega_1)_f Tf^1 + (\omega_2)_f  Tf^2    \right ).
\end{align*}
\noindent
Finally we can put all the terms together and calculate the coefficient of $dx \wedge dy \wedge \theta$:
\begin{align*}
X (\lambda (f) ) \cdot \mathcal{E}_{\omega,f} - Y  &( \lambda (f) ) \cdot \mathcal{F}_{\omega,f} + \lambda (f) ( X(\mathcal{E}_{\omega,f})-  Y(\mathcal{F}_{\omega,f}) )=\\
&  - \lambda (f) \left [   (\omega_1)_f Tf^1 + (\omega_2)_f  Tf^2    \right ] \\
&+  \lambda (f) \left [  (\omega_1)_f  Tf^1 + (\omega_2)_f T f^2+ \lambda (f) ( ( X\omega_2)_f - (Y\omega_1)_f  )  \right ]\\
 =&\lambda ^2(f) \left ( ( X\omega_2)_f - (Y\omega_1)_f \right  ) .
\end{align*}
This completes the proof.
\end{proof}


\begin{proof}[Proof of equation \eqref{k=2} in Proposition \ref{commuH1}]
Let $\omega= \omega_1 dx + \omega_2 dy \in \mathcal{D}_{\mathbb{H},1}(U)$. Then
\begin{align*}
f^* D \omega=& f^* \left [  (XX \omega_2 -XY\omega_1 -T\omega_1 ) dx \wedge \theta+ (YX \omega_2 - YY\omega_1 - T \omega_2) dy \wedge \theta   \right ]\\
=&(XX \omega_2 -XY\omega_1 -T\omega_1 )_f \left (  X f^1 \lambda (f) dx \wedge \theta + Y f^1 \lambda (f) dy \wedge \theta \right  )\\
&
+(YX \omega_2 - YY\omega_1 - T \omega_2)_f \left (  X f^2 \lambda (f) dx \wedge \theta + Y f^2 \lambda (f) dy \wedge \theta  \right  )\\
=& \left [    
(XX \omega_2 -XY\omega_1 -T\omega_1 )_f   X f^1 + (YX \omega_2 - YY\omega_1 - T \omega_2)_f   X f^2 \right ] \lambda (f)  dx \wedge \theta\\
&
+ \left [ (XX \omega_2 -XY\omega_1 -T\omega_1 )_f   Y f^1 + (YX \omega_2 - YY\omega_1 - T \omega_2)_f   Y f^2 \right ] \lambda (f)  dy \wedge \theta\\
=& \Big [   - X f^1 ( XY\omega_1 +T\omega_1 )_f    - X f^2  (YY\omega_1 )_f  +  X f^1 (XX \omega_2  )_f   \\
& +  X f^2 (YX \omega_2  - T \omega_2)_f   \Big ] \lambda (f)  dx \wedge \theta  + \Big [   - Y f^1 ( XY\omega_1 +T\omega_1 )_f   \\
& - Y f^2  (YY\omega_1 )_f  +  Y f^1 (XX \omega_2  )_f    +  Y f^2 (YX \omega_2  - T \omega_2)_f   \Big ] \lambda (f)  dy \wedge \theta.
\end{align*}

On the other hand, $D f^* \omega$ will be much more complicated
\begin{align*}
D f^* \omega =& D \left ( (\omega_1)_f (  Xf^1 dx + Y f^1 dy )  +(\omega_2)_f (  Xf^2 dx + Y f^2 dy )   \right )\\
=&  D \left ( ( (\omega_1)_f   Xf^1  +(\omega_2)_f  Xf^2  )dx + ( (\omega_1)_f   Yf^1  +(\omega_2)_f  Yf^2  ) dy \right )\\
=& \Big [  XX \left (  (\omega_1)_f   Yf^1  +(\omega_2)_f  Yf^2  \right ) - XY  \left ( (\omega_1)_f   Xf^1  +(\omega_2)_f  Xf^2 \right  )\\
& - T  \left ( (\omega_1)_f   Xf^1  +(\omega_2)_f  Xf^2 \right  )  \Big ] dx \wedge \theta +\Big [  YX \left (  (\omega_1)_f   Yf^1  +(\omega_2)_f  Yf^2  \right ) \\
& - YY  \left ( (\omega_1)_f   Xf^1  +(\omega_2)_f  Xf^2 \right  )  - T  \left ( (\omega_1)_f   Yf^1  +(\omega_2)_f  Yf^2 \right  )  \Big ] dy \wedge \theta.
\end{align*}
\noindent
Given the length of the coefficients, we consider them piece by piece. Consider only the coefficient of $dx \wedge \theta$ and divide it accordindly to the number $k$ of derivatives that hit, respectively, the functions $\omega_1$ and $\omega_2$ (the $T$ derivative counts as two): we denote such parts as $\mathcal{Z}^{(k)} \omega_i$, with $k=0,1,2$ and $i=1,2$.\\
Then the coefficient of $dx \wedge \theta$ becomes
\begin{align*}
(\dots) dx \wedge \theta &= \left ( \mathcal{Z}^{(0)} \omega_1 + \mathcal{Z}^{(0)} \omega_2 + \mathcal{Z}^{(1)} \omega_1 +\mathcal{Z}^{(1)} \omega_2 + \mathcal{Z}^{(2)} \omega_1 + \mathcal{Z}^{(2)} \omega_2 \right ) dx \wedge \theta
\end{align*}
with
\begin{align*}
\mathcal{Z}^{(0)} \omega_1 &:=   (\omega_1)_f  XXYf^1 +  -   (\omega_1)_f   XYXf^1 -  (\omega_1)_f  TXf^1    \\
\mathcal{Z}^{(0)} \omega_2  &:=   (\omega_2)_f  XXYf^2  -   (\omega_2)_f   XYXf^2  -  (\omega_2)_f  TXf^2 \\
\mathcal{Z}^{(1)} \omega_1   &:= 2  X \left (  (\omega_1)_f  \right  ) XYf^1 - Y  \left ( (\omega_1)_f \right )   XXf^1  - X  \left ( (\omega_1)_f \right )  YXf^1    \\
\mathcal{Z}^{(1)} \omega_2   &:= 2  X \left (  (\omega_2)_f  \right  ) XYf^2  - Y  \left ( (\omega_2)_f \right )   XXf^2  - X  \left ( (\omega_2)_f \right )  YXf^2   \\
\mathcal{Z}^{(2)} \omega_1  &:= XX \left (  (\omega_1)_f  \right  ) Yf^1  - XY  \left ( (\omega_1)_f  \right  )  Xf^1 - T  \left ( (\omega_1)_f  \right  )  Xf^1      \\
\mathcal{Z}^{(2)} \omega_2   &:=   XX \left (  (\omega_2)_f  \right  ) Yf^2  - XY  \left ( (\omega_2)_f  \right  )  Xf^2 - T  \left ( (\omega_2)_f  \right  )  Xf^2.
\end{align*}
Compute:
\begin{align*}
\mathcal{Z}^{(0)} \omega_1 &=   (\omega_1)_f  XXYf^1  -   (\omega_1)_f   XYXf^1 -  (\omega_1)_f TXf^1\\
&
=(\omega_1)_f   X (   XY-YX )f^1-  (\omega_1)_f TXf^1\\
&
=  (\omega_1)_f (XT-TX) f^1 = (\omega_1)_f [X,T] f^1 =0.
\end{align*}
Notice that, thanks of the simmetry, in the same way $\mathcal{Z}^{(0)} \omega_2=0$.\\
Next, by Lemma \ref{doublederivativecomposition}, one has the following equations in the case $n=1$:
 \begin{align*}
  \begin{aligned}
XX (g \circ f) =&  \left [ (XXg)_f X f^1 + ( YXg)_f  X f^2  \right ]  X f^1 +  (Xg)_f  X X f^1 \\    
&+ \left [ (XYg)_f X f^1 + ( YYg)_f X f^2  \right ]  X f^2 +  (Yg)_f  X X f^2,
  \end{aligned}
 \end{align*}
 \begin{align*}
  \begin{aligned}
XY(g \circ f)=&\left [ (XXg)_f X f^1 + ( YXg)_f X f^2  \right ]  Y f^1 +  (Xg)_f  X Y f^1\\
& + \left [ (XYg)_f X f^1 + ( YYg)_f  X f^2  \right ]  Y f^2 +  (Yg)_f  X Y f^2,
  \end{aligned}
 \end{align*}
 \begin{align*}
  \begin{aligned}
T(g \circ f) =&XgTf^1 + YgTf^2 +\lambda (1,f) Tg,\quad \quad \quad \quad \quad \quad \quad 
  \end{aligned}
 \end{align*}
where $\lambda (1,f):=Xf^1 Yf^2-Yf^1Xf^2$. Using them, one gets
\begin{align*}
\mathcal{Z}^{(2)} \omega_1 =&  XX \left (  (\omega_1)_f  \right  ) Yf^1  - XY  \left ( (\omega_1)_f  \right  )  Xf^1  - T  \left ( (\omega_1)_f  \right  )  Xf^1  \\
=& XX \left (  \omega_1 \circ f  \right  ) Yf^1  - XY  \left ( \omega_1 \circ f  \right  )  Xf^1  - T  \left ( \omega_1 \circ f  \right  )  Xf^1\\
 =& \bigg \{ \left [  (XX\omega_1)_f X f^1 +( YX\omega_1)_f  X f^2  \right ]  X f^1 +  (X\omega_1)_f  X X f^1\\
&+ \left [  (XY\omega_1)_f X f^1 + ( YY\omega_1)_f  X f^2  \right ]  X f^2 +  (Y\omega_1)_f  X X f^2 \bigg \}  Yf^1 \\
&- \bigg \{ \left [  (XX\omega_1)_f X f^1 +  ( YX\omega_1)_f  X f^2  \right ]  Y f^1 +  (X\omega_1)_f  X Y f^1\\
&+ \left [ (XY\omega_1)_f X f^1 + ( YY\omega_1)_f  X f^2  \right ]  Y f^2 +  (Y\omega_1)_f  X Y f^2 \bigg \}  Xf^1 \\
&- \bigg \{  (X\omega_1)_f Tf^1 + (Y\omega_1)_f Tf^2 +\lambda (f) (T\omega_1)_f \bigg \}  Xf^1\\
 =&  (XY\omega_1)_f X f^1 (    Xf^2  Yf^1 -    Yf^2    Xf^1  ) + ( YY\omega_1)_f  X f^2  (    Xf^2  Yf^1 -    Yf^2    Xf^1 )\\
&+   (X\omega_1)_f \left  [ X X f^1  Yf^1 -( X Y f^1 +  Tf^1 )  Xf^1 \right ]  \\
&+ (Y\omega_1)_f \left [  X X f^2 Yf^1  -( X Y f^2 + Tf^2)  Xf^1   \right ] - (T\omega_1)_f \lambda (f)  Xf^1\\
 =&  - \lambda (f) \left  [ (XY\omega_1)_f X f^1  + ( YY\omega_1)_f  X f^2  + (T\omega_1)_f   Xf^1 \right ] \\
&+   (X\omega_1)_f \left  [ X X f^1  Yf^1 -( X Y f^1 +  Tf^1 )  Xf^1 \right ]\\
& +  (Y\omega_1)_f \left [  X X f^2 Yf^1  -( X Y f^2 + Tf^2)  Xf^1   \right ] \\
 =&  - \lambda (f) \left  [  ( YY\omega_1)_f  X f^2  +\left ( (XY\omega_1)_f + (T\omega_1)_f \right )  Xf^1 \right ]  \\
&+   (X\omega_1)_f \left  [ X X f^1  Yf^1 -( X Y f^1 +  Tf^1 )  Xf^1 \right ] \\
&+  (Y\omega_1)_f \left [  X X f^2 Yf^1  -( X Y f^2 + Tf^2)  Xf^1   \right ].
\end{align*}
Notice that $ - \lambda (f) \left  [  ( YY\omega_1)_f  X f^2  +\left ( (XY\omega_1)_f + (T\omega_1)_f \right )  Xf^1 \right ]$    is exactly the $ \omega_1$-part of the coefficient of $ dx \wedge \theta$  in $f^*D \omega$.\\
Since we found the  $\omega_1$-part of the coefficient  of $dx \wedge \theta$ that we wish as result, now we show that the rest is zero. Specifically that
\begin{align}\label{II+=0}
  \begin{aligned}
\mathcal{Z}^{(1)} \omega_1 &+  (X\omega_1)_f \left  [ X X f^1  Yf^1 -( X Y f^1 +  Tf^1 )  Xf^1 \right ] \\
&+  (Y\omega_1)_f \left [  X X f^2 Yf^1  -( X Y f^2 + Tf^2)  Xf^1   \right ] \stackrel{Th}{=}0.
  \end{aligned}
\end{align}
Indeed
\begin{align*}
\mathcal{Z}^{(1)} \omega_1 =& 2  X \left (  (\omega_1)_f  \right  ) XYf^1 - Y  \left ( (\omega_1)_f \right )   XXf^1  - X  \left ( (\omega_1)_f \right )  YXf^1 \\
=&  X \left (  (\omega_1)_f  \right  ) ( XY-YX) f^1 + X \left (  (\omega_1)_f  \right  ) XYf^1 - Y  \left ( (\omega_1)_f \right )   XXf^1 \\ 
=&  X \left (  \omega_1 \circ f  \right  ) \left ( Tf^1 + XYf^1 \right ) - Y  \left ( \omega_1 \circ f  \right )   XXf^1\\
=&   \left ( (X\omega_1)_f    X f^1  + (Y\omega_1)_f   X f^2  \right  ) \left ( Tf^1 + XYf^1 \right ) \\
&-   \left (   (X\omega_1)_f  Y f^1    + (Y\omega_1)_f  Y f^2 \right  )   XXf^1\\
=& (X\omega_1)_f  [  X f^1 \left ( Tf^1 + XYf^1 \right ) - Y f^1 XXf^1 ] \\
&+ (Y\omega_1)_f [  X f^2  \left ( Tf^1 + XYf^1 \right ) - Y f^2 XXf^1 ],
\end{align*}
where we used Lemma \ref{composition_contact} in the second to last line.\\\\
With the appropriate cancellations, equation \eqref{II+=0} is reduced to
\begin{equation}
  X X f^2 Yf^1  -( X Y f^2 + Tf^2)  Xf^1   +   X f^2  \left ( Tf^1 + XYf^1 \right ) - Y f^2 XXf^1  =0
\end{equation}
Rearranging the first hand side, notice that
\begin{align*}
&  X X f^2 Yf^1 + X f^2  XYf^1 - X Y f^2  Xf^1- Y f^2 XXf^1    - Tf^2  Xf^1 + X f^2   Tf^1=\\
& =   X (X f^2 Yf^1) - X (Y f^2  Xf^1)    - Tf^2  Xf^1 + X f^2   Tf^1\\
& =  X (X f^2 Yf^1 - Y f^2  Xf^1)   - Tf^2  Xf^1 + X f^2   Tf^1\\
& = -  X (\lambda (f))   - Tf^2  Xf^1 + X f^2   Tf^1\\
&= - ( X f^2   Tf^1 - Tf^2  Xf^1  )   - Tf^2  Xf^1 + X f^2   Tf^1=0,
\end{align*}
where we used a previous calculation of $ X (\lambda (f))$ (see Lemma \ref{XCYC}).\\\\
This indeed shows that the $\omega_1$-parts of the coefficient of $dx \wedge \theta$ are the same in $f^* D \omega $ and $D f^* \omega$. To complete the proof of $f^* D \omega= D f^* \omega$, one should prove the same also for the $\omega_2$-part of the coefficient of $dx \wedge \theta$ and then for the whole $dy \wedge \theta$. This can be done in a similar way.
\end{proof}

\section{Calculations for the Möbius Strip}\label{appMo}

\begin{obs}\label{Hcoordinates1}
We write $\vec\gamma_r (r,s)$ in Heisenberg coordinates:
\begin{align*}
\vec\gamma_r (r,s) =& 
\bigg (   - \frac{s}{2} \sin \left  ( \frac{r}{2} \right )     \cos r -  \left [ R+s \cos \left  ( \frac{r}{2} \right ) \right ]     \sin r   , \\ 
& - \frac{s}{2} \sin \left  ( \frac{r}{2} \right )     \sin r +       \left [ R+s \cos \left  ( \frac{r}{2} \right ) \right ]     \cos r    ,  \
  \frac{s}{2} \cos \left  ( \frac{r}{2} \right ) 
\bigg  ) \\
=& \left (   
- \frac{s}{2} \sin \left  ( \frac{r}{2} \right )     \cos r -  
\left [ R+s \cos \left  ( \frac{r}{2} \right ) \right ]
\sin r  \right ) \partial_x  \\
& +
\left ( 
- \frac{s}{2} \sin \left  ( \frac{r}{2} \right )     \sin r +  
\left [ R+s \cos \left  ( \frac{r}{2} \right ) \right ]
 \cos r \right ) \partial_y
+
 \left ( \frac{s}{2} \cos \left  ( \frac{r}{2} \right ) \right  )\partial_t \\
=&  \left (   
- \frac{s}{2} \sin \left  ( \frac{r}{2} \right )     \cos r -  
\left [ R+s \cos \left  ( \frac{r}{2} \right ) \right ]
\sin r  \right ) X \\
& +
\left(
- 
\frac{s}{2} \sin \left  ( \frac{r}{2} \right )     \sin r +  
\left [ R+s \cos \left  ( \frac{r}{2} \right ) \right ]
 \cos r \right ) Y  \\
&  +
\bigg [
\frac{s}{2} \cos \left  ( \frac{r}{2} \right )
+
\left (   
- \frac{s}{2} \sin \left  ( \frac{r}{2} \right )     \cos r -  
\left [ R+s \cos \left  ( \frac{r}{2} \right ) \right ]
\sin r  \right )
\frac{1}{2}y\\
&  -
\left (   
- \frac{s}{2} \sin \left  ( \frac{r}{2} \right )     \sin r +  
\left [ R+s \cos \left  ( \frac{r}{2} \right ) \right ]
 \cos r \right )
\frac{1}{2}x
\bigg ] T .
\end{align*}
The third component can be written better as
\begin{align*}
  \bigg [  & \frac{s}{2} \cos \left  ( \frac{r}{2} \right )  +  \left (   - \frac{s}{2} \sin \left  ( \frac{r}{2} \right )     \cos r -  \left [ R+s \cos \left  ( \frac{r}{2} \right ) \right ]  \sin r  \right )  \frac{1}{2}y\\
&  -  \left (   - \frac{s}{2} \sin \left  ( \frac{r}{2} \right )     \sin r +   \left [ R+s \cos \left  ( \frac{r}{2} \right ) \right ]   \cos r \right )  \frac{1}{2}x  \bigg ] \\
=&  \bigg [  \frac{s}{2} \cos \left  ( \frac{r}{2} \right )  +  \left (   - \frac{s}{2} \sin \left  ( \frac{r}{2} \right )     \cos r -  \left [ R+s \cos \left  ( \frac{r}{2} \right ) \right ]  \sin r  \right )   \frac{1}{2}  \left [R+s \cos \left ( \frac{r}{2} \right ) \right ] \sin r \\
& -  \left (   - \frac{s}{2} \sin \left  ( \frac{r}{2} \right )     \sin r +  \left [ R+s \cos \left  ( \frac{r}{2} \right ) \right ]   \cos r \right )  \frac{1}{2}     \left [ R+s \cos \left  ( \frac{r}{2} \right ) \right ] \cos r    \bigg  ]  \\
=&   \left [ \frac{s}{2} \cos \left  ( \frac{r}{2} \right )  -   \left [ R+s \cos \left  ( \frac{r}{2} \right ) \right ]^2 \sin^2 r   \frac{1}{2}       -  \left [ R+s \cos \left  ( \frac{r}{2} \right ) \right ]^2   \cos^2 r  \frac{1}{2}    \right  ]  \\
=&  \left  [\frac{s}{2} \cos \left ( \frac{r}{2} \right ) -\left [ R+s \cos \left ( \frac{r}{2} \right ) \right ]^2\frac{1}{2} \right  ] .\\
\end{align*}
\end{obs}


\begin{obs}\label{Hcoordinates2}
We write $\vec\gamma_s (r,s)$ in Heisenberg coordinates, remembering that $x(r,s)=  \left [ R+s \cos \left  ( \frac{r}{2} \right ) \right ] \cos r  $ and $  y(r,s) =  \left [R+s \cos \left ( \frac{r}{2} \right ) \right ] \sin r $:
\begin{align*}
\vec\gamma_s (r,s)=&    \cos \left  ( \frac{r}{2} \right )  \cos r    \partial_x
  +
  \cos \left ( \frac{r}{2} \right )  \sin r \partial_y
 +
 \sin \left  ( \frac{r}{2} \right )  \partial_t \\
=&    \cos \left  ( \frac{r}{2} \right )  \cos r X
  +
 \cos \left ( \frac{r}{2} \right )  \sin r Y  \\
&
 +
\left (  \sin \left  ( \frac{r}{2} \right )
+ \cos \left  ( \frac{r}{2} \right )  \cos r \frac{1}{2}y
-  \cos \left ( \frac{r}{2} \right )  \sin r   \frac{1}{2}x
 \right  )  T \\
=&    \cos \left  ( \frac{r}{2} \right )  \cos r X
  +
 \cos \left ( \frac{r}{2} \right )  \sin r Y\\
&
 +
\bigg (  \sin \left  ( \frac{r}{2} \right )
+ \cos \left  ( \frac{r}{2} \right )  \frac{ \cos r}{2}  \left [R+s \cos \left ( \frac{r}{2} \right ) \right ] \sin r \\
&-  \cos \left ( \frac{r}{2} \right )    \frac{ \sin r}{2}  \left [ R+s \cos \left  ( \frac{r}{2} \right ) \right ] \cos r
 \bigg  )  T \\
=&    \cos \left  ( \frac{r}{2} \right )  \cos r X
  +
 \cos \left ( \frac{r}{2} \right )  \sin r Y
 +
 \sin \left  ( \frac{r}{2} \right )  T .
\end{align*}
\end{obs}


\begin{obs}\label{TvecN}
We compute the $T$ component of $\vec{N}= \vec\gamma_r \times_\mathbb{H}  \vec\gamma_s= \vec{N}_1 X + \vec{N}_2 Y + \vec{N}_3 T$.
\begin{align*}
 \vec{N}_3 =&  \left (    - \frac{1}{2}s \sin \left  ( \frac{r}{2} \right )    \cos \left ( \frac{r}{2} \right )  \sin r  \cos r
 -   \left [ R+s \cos \left  ( \frac{r}{2} \right ) \right ]   \cos \left ( \frac{r}{2} \right )   \sin^2 r  \right )\\
&
-
\left( - \frac{1}{2}s \sin \left  ( \frac{r}{2} \right )    \cos \left  ( \frac{r}{2} \right )  \sin r  \cos r 
+  \left [ R+s \cos \left  ( \frac{r}{2} \right ) \right ]  \cos \left  ( \frac{r}{2} \right )  \cos^2 r \right ) \\
=& -   \left [ R+s \cos \left  ( \frac{r}{2} \right ) \right ]   \cos \left ( \frac{r}{2} \right ) .
\end{align*}
\end{obs}


\begin{obs}\label{XvecN}
We compute the $X$ component of $\vec{N}$.
\begin{align*}
 \vec{N}_1 =& \left( - \frac{1}{2}s \sin \left  ( \frac{r}{2} \right )   \sin r +  \left [ R+s \cos \left  ( \frac{r}{2} \right ) \right ]  \cos r \right )   \sin \left  ( \frac{r}{2} \right ) \\
&
-
\left (  s \frac{1}{2} \cos \left  ( \frac{r}{2} \right )  - \left [ R+s \cos \left  ( \frac{r}{2} \right ) \right ]^2 \frac{1}{2} \right  )
 \cos \left ( \frac{r}{2} \right )  \sin r\\
=&
 - \frac{1}{2}s \sin^2 \left  ( \frac{r}{2} \right )   \sin r +  \left [ R+s \cos \left  ( \frac{r}{2} \right ) \right ]  \cos r  \sin \left  ( \frac{r}{2} \right )     \\
&
-    s \frac{1}{2} \cos^2 \left  ( \frac{r}{2} \right )  \sin r   +  \left [ R+s \cos \left  ( \frac{r}{2} \right ) \right ]^2 \frac{1}{2}  \cos \left ( \frac{r}{2} \right )  \sin r  \\
=&
 - \frac{1}{2}s    \sin r 
+  \left [ R+s \cos \left  ( \frac{r}{2} \right ) \right ]  \cos r  \sin \left  ( \frac{r}{2} \right )     
 +  \left [ R+s \cos \left  ( \frac{r}{2} \right ) \right ]^2 \frac{1}{2}  \cos \left ( \frac{r}{2} \right )  \sin r  .
\end{align*}
\end{obs}


\begin{obs}\label{YvecN}
We compute the $Y$ component of $\vec{N}$, remembering again $\sin (2 \alpha)=2 \sin \alpha \cos \alpha$ and $\cos (2 \alpha) = 2 \cos^2 \alpha -1$, and naming $\alpha= \frac{r}{2}$.
\begin{align*}
 \vec{N}_2 =& \left (
\frac{s}{2} \sin \left  ( \frac{r}{2} \right )  \cos r +   \left [ R+s \cos \left  ( \frac{r}{2} \right ) \right ] \sin r
\right )
\sin \left  ( \frac{r}{2} \right )  \\
&
+
\left (
  \frac{s}{2} \cos \left  ( \frac{r}{2} \right )  - \left [ R+s \cos \left  ( \frac{r}{2} \right ) \right ]^2 \frac{1}{2}
\right )
\cos \left  ( \frac{r}{2} \right )  \cos r \\
=& \frac{s}{2} \sin^2 \left  ( \frac{r}{2} \right )  \cos r +   \left [ R+s \cos \left  ( \frac{r}{2} \right ) \right ] \sin r \sin \left  ( \frac{r}{2} \right )  \\
&
+
  \frac{s}{2} \cos^2 \left  ( \frac{r}{2} \right )   \cos r  - \left [ R+s \cos \left  ( \frac{r}{2} \right ) \right ]^2 \frac{1}{2}
\cos \left  ( \frac{r}{2} \right )  \cos r \\
=&  \frac{s}{2}   \cos r 
+   \left [ R+s \cos \left  ( \frac{r}{2} \right ) \right ] \sin r \sin \left  ( \frac{r}{2} \right )
  - \left [ R+s \cos \left  ( \frac{r}{2} \right ) \right ]^2 \frac{1}{2}
\cos \left  ( \frac{r}{2} \right )  \cos r\\
=& \frac{s}{2} (  2 \cos^2 \alpha -1 ) 
+   \left ( R+s \cos \alpha \right ) ( 2 \sin \alpha \cos \alpha ) \sin \alpha \\
&
  - \left ( R+s \cos\alpha \right )^2 \frac{1}{2}  \cos \alpha  (  2 \cos^2 \alpha -1 )\\
=&  s \cos^2 \alpha   - \frac{s}{2} 
+  2  ( R+s \cos \alpha )   \sin^2 \alpha \cos \alpha  \\
&
  -     ( R+s \cos\alpha  )^2   \cos^3 \alpha 
+   \frac{1}{2}  [ R+s \cos\alpha  ]^2  \cos \alpha\\
=& s \cos^2 \alpha   - \frac{s}{2} 
+  2  ( R+s \cos \alpha  )  ( 1-  \cos^2 \alpha ) \cos \alpha\\
&
  -     ( R+s \cos\alpha  )^2   \cos^3 \alpha 
+   \frac{1}{2}  [ R+s \cos\alpha  ]^2  \cos \alpha \\
=&  sz^2    - \frac{s}{2}  +  2 ( R+s z  )  ( 1- z^2  ) z    -     ( R+sz  )^2   z^3 +   \frac{1}{2}  ( R+s z  )^2  z\\
=& sz^2   
 - \frac{s}{2}  
+  2  ( R+s z  )  ( z- z^3  )   
-    ( R^2+ 2sRz +  s^2 z^2  )                 z^3 
+   \frac{1}{2}  ( R^2+ 2sRz +  s^2 z^2  )  z\\
=& 
sz^2   
 - \frac{1}{2}   s
+  2   R  ( z- z^3  )   
+  2  s ( z^2- z^4  )   
-     R^2    z^3 - 2sRz^4 - s^2 z^5
+   \frac{1}{2} R^2 z+  sRz^2 +  \frac{1}{2}  s^2 z^3  \\
=& 
\left (   - z^5     +  \frac{1}{2}  z^3   \right )   s^2   
+   \left (    - 2 (R+1)z^4       +    ( R+3) z^2     - \frac{1}{2}   \right )   s   
   - ( R^2  + 2   R) z^3  \\
& + \left ( \frac{1}{2} R^2    +  2   R  \right  )   z  ,
\end{align*}
where we used the facts that $ \sin^2 \alpha  = 1-  \cos^2 \alpha $ and $\cos \alpha =z $.
\end{obs}


\begin{obs}\label{N1=0N2=0}
We solve the system $\{\vec{N}_1 =0, \ \vec{N}_2 =0 \}$.\\
First impose $\vec{N}_1 =0$ and, using again $\sin (2 \alpha)=2 \sin \alpha \cos \alpha$ and $\cos (2 \alpha) = 2 \cos^2 \alpha -1$, with $\alpha= \frac{r}{2}$, and recalling that $z=\cos \alpha $, we find that:
\end{obs}
\begin{align*}
 - \frac{1}{2}s    \sin r 
+  \Big [ R &+s \cos \left  ( \frac{r}{2} \right ) \Big ]  \cos r  \sin \left  ( \frac{r}{2} \right )     
 +  \left [ R+s \cos \left  ( \frac{r}{2} \right ) \right ]^2 \frac{1}{2}  \cos \left ( \frac{r}{2} \right )  \sin r  =0,
\\
 - s   \sin \alpha \cos \alpha &
+  \left [ R+s \cos \alpha \right ]   (   2 \cos^2 \alpha -1   )  \sin \alpha     
 +  \left [ R+s \cos \alpha \right  ]^2  \cos^2 \alpha   \sin \alpha  =0,
\\
\sin \alpha =0
\quad &\vee \quad
 - s    \cos \alpha 
+  \left [ R+s \cos \alpha \right ]   (   2 \cos^2 \alpha -1   )  
 +  \left [ R+s \cos \alpha \right  ]^2  \cos^2 \alpha     =0,
\\
 \frac{r}{2}= \alpha = k\pi
\quad &\vee \quad
 - s   z
+   [ R+s z  ]   (   2 z^2  -1   )  
 +   [ R+s z   ]^2  z^2     =0,
\\
 r = 2k\pi
\quad &\vee \quad
 - s z  +  2R z^2+2s z^3 - R - s z   +    R^2  z^2 +  2sRz^3  +s^2 z^4     =0,
 \\
 r = 2k\pi
\quad &\vee \quad
s^2 z^4   +  2s(R+1) z^3  +   ( R^2 +2R  )   z^2 - 2s z  - R       =0 .
\end{align*}
\noindent
The second condition is an equation of fourth degree in $z$ and of second in $s$, so we solve it in $s$:
$$
 r = 2k\pi
\quad \vee \quad
s^2 z^4  
 + (  2(R+1) z^3      - 2 z  )     s
+   ( R^2 +2R  )   z^2 - R       =0
$$
Now consider only the second condition. If $z=0$, one has that 
 $R=0$, 
which is impossible.\\
If $z\neq 0$, to solve the second equation as an equation of $2^{nd}$ order in $s$, compute its discriminant:
\begin{align*}
\Delta &= (  2(R+1) z^3      - 2 z  )^2                      -4 z^4   (   ( R^2 +2R  )   z^2 - R   )  \\
&=4(R+1)^2 z^6 + 4z^2 - 8(R+1)z^4 -4  (  R^2 +2R  )   z^6  +4 Rz^4\\
&=4z^6 -4(    R +2   )   z^4 + 4z^2\\
&=4z^2 (  z^4   -(    R +2   )   z^2  +1).
\end{align*}
Then
\begin{align*}
s& = \frac{    - (  2(R+1) z^3      - 2 z  )   \pm \sqrt{        4z^2 (  z^4   -(    R +2   )   z^2  +1)        }       }{2 z^4}\\
&= \frac{          -  (R+1) z^2     +  1   \pm    \sqrt{          z^4   -(    R +2   )   z^2  +1       }        }{ z^3}.
\end{align*}
Since $z=\cos \frac{r}{2}$, $z \in [0,2\pi)$, then $z   \neq 0 $ if and only if $ r \neq \pi +2k\pi.$\\
As a summary, the $X$ component of the normal vector $\vec{N}$ is zero at the points\\
$(x(r,s),y(r,s),t(r,s))$ with
\begin{align}\label{1solution}
(r,s)=(0,s), \quad  s \in [-w,w], \quad \text{or}
\end{align}
\begin{align}\label{2solution}
 (r,s)= \left (r, \frac{ -(R+1) z^2 + 1 \pm \sqrt{ z^4  -( R+2 ) z^2 +1 } }{ z^3} \right  ),
 \quad r \in [0, 2 \pi), \ r \neq \pi, \ z=\cos \frac{r}{2}.
\end{align}
Now we check whether the  parameters \eqref{1solution} and \eqref{2solution} also force $ \vec{N}_2 (r,s) $, the coefficient of the $Y$ component of $\vec{N}$, to be zero.
We will show that $ \vec{N}_2 (r,s) $ is zero only in at most a finite number of the  parameters (actually at most one)  \eqref{1solution} and \eqref{2solution}.\\\\
$\blacktriangleright$ 
the case $(r,s)=(0,s)$ 
gives $z 
= \cos \frac{r}{2} = \cos 0=1$ and we obtain that 
\begin{align*}
\vec{N}_2 (r,s) =&  \left (   - 1     +  \frac{1}{2}    \right )   s^2     +   \left (    - 2 (R+1)       +    ( R+3)      - \frac{1}{2}   \right )   s      - R^2  - 2   R  +  \frac{1}{2} R^2    +  2   R   \\
=&  -   \frac{1}{2}    s^2   +   \left (    - R        + \frac{1}{2}   \right )   s      -  \frac{1}{2} R^2  .   
\end{align*}
So we get that $\vec{N}_2 (r,s) =0$ if and only if   $ s^2   +   \left (     2R  -1   \right )   s      +  R^2 =0$. Then
$$
\Delta =  (     2R  -1  )^2 -4R^2=
 -4R+1
$$
and so
$$
s_1,s_2= \frac{ -2R+1 \pm \sqrt{-4R+1}  }{2}.
$$
This proves that $\vec{N}_1 (r,s)$ and $\vec{N}_2 (r,s)$, the component in $X$ and $Y$ of the normal vector $\vec{N}$ are both zero at least in the two cases $(0,s_1)$ and $(0, s_2)$.\\\\
$\blacktriangleright$ The second case gives
$$
s =    \frac{          -  (R+1) z^2     +  1   \pm    \sqrt{          z^4   -(    R +2   )   z^2  +1       }        }{ z^3}  =
 \frac{          -  (R+1) z^2     +  1  +A(z)       }{ z^3},
$$
where $A(z):= \pm    \sqrt{          z^4   -(    R +2   )   z^2  +1       } $.  
Then
\begin{align*}
\vec{N}_2& (r,s)=\\
 =&  \left (   - z^5     +  \frac{1}{2}  z^3   \right )   s^2   +   \left (    - 2 (R+1)z^4       +    ( R+3) z^2     - \frac{1}{2}   \right )   s     - ( R^2  + 2   R) z^3     + \left ( \frac{1}{2} R^2    +  2   R  \right  )   z  \\
=& \left (   - z^5     +  \frac{1}{2}  z^3   \right )   \left  (      \frac{          -  (R+1) z^2     +  1  +A(z)       }{ z^3}   \right    )^2     \\
&+ \left (    - 2 (R+1)z^4       +    ( R+3) z^2     - \frac{1}{2}   \right )   \frac{          -  (R+1) z^2     +  1  +A(z)       }{ z^3}   \\
&   - ( R^2  + 2   R) z^3   + \left ( \frac{1}{2} R^2    +  2   R  \right  )   z  \\
=&\left (   - \frac{1}{z}     +  \frac{1}{2  z^3}   \right )            \left  (      (R+1)^2   z^4     +  1  +A^2  (z) -2 (R+1) z^2 +2A(z)  -  2(R+1) A(z)z^2   \right    )        \\ 
&+   \left (    - 2 (R+1)z       +     \frac{( R+3)}{z}     - \frac{1}{2 z^3}   \right )   (         -  (R+1) z^2     +  1  +A(z)    )   - ( R^2  + 2   R) z^3  \\
&   + \left ( \frac{1}{2} R^2    +  2   R  \right  )   z \\
=&    -(R+1)^2   z^3      - \frac{1}{z}   - \frac{A^2  (z)}{z}  
 +  \frac{(R+1)^2   z}{2 }     +  \frac{A^2  (z)}{2  z^3}   - \frac{ (R+1)}{  z}  
+ \frac{A(z)}{  z^3}  +     2 (R+1)^2 z^3       \\
&  -  (R+1) ( R+3)z    + \frac{ R+3}{z}  
+ \frac{(R+1)}{2 z}      - \frac{A(z) }{2 z^3}    
   - ( R^2  + 2   R) z^3   + \left ( \frac{1}{2} R^2    +  2   R  \right  )   z  \\
=& \left (   -(R+1)^2     +     2 (R+1)^2       -   R^2 - 2   R    \right  ) z^3   \\
&   +\left (  \frac{(R+1)^2  }{2 }   -  (R+1) ( R+3)      +  \frac{1}{2} R^2    +  2   R  \right  )   z \\
&   + \left (  - 1   - R -1   +  R+3  + \frac{R}{2 } + \frac{1}{2 }   \right )  \frac{1}{z}
  - \frac{A^2  (z)}{z}      +  \frac{A^2  (z)}{2  z^3} 
+ \frac{A(z)}{  z^3}         - \frac{A(z) }{2 z^3}   .
\end{align*}
Here I use that $A^2 (z):=  z^4   -(    R +2   )   z^2  +1  $ and thus
\begin{align*}
\vec{N}_2 (r,s) =&   z^3   +\left (   -R   -\frac{5  }{2 }    \right  )   z + \frac{R+3}{2z}    -   z^3 + (R+2)z  - \frac{1  }{z}   +  \frac{  1 }{2 } z   - \frac{  (R+2)  }{2  z}   + \frac{1  }{2  z^3}         + \frac{A(z) }{2 z^3}  \\  
=&- \frac{ 1 }{2z}   + \frac{1  }{2  z^3}        + \frac{A(z) }{2 z^3}   .\\
\end{align*}
Since $z\neq 0$, $\vec{N}_2 (r,s) =0$ if and only if
\begin{align*}
 A(z)  & = z^2 -1,\\
 A^2(z)  & = ( z^2 -1 )^2,\\
z^4 -(R+2)z^2 +1 &=  z^4 - 2z^2 + 1,\\
-R z^2   &=0,\\
z   &=0,
\end{align*}
which is impossible because this is the case $z \neq 0$.\\\\
So the second solution of the equation $\vec{N}_1 (r,s)=0$ is never a solution for the equation $\vec{N}_2(r,s)=0$.

\begin{obs}\label{partialconclusion}
So far we have obtained that the Möbius strip has at most only two critical points and, in this parametrization, they are those obtained by
$$
(r,s_1)= \left (0,\frac{ -2R+1 - \sqrt{-4R+1}  }{2} \right )  \text{ and }   (r,s_2)=\left (0,\frac{ -2R+1 + \sqrt{-4R+1}  }{2} \right ) .
$$
In particular, there are two critical points if $R < \frac{1}{4}$, one if $R = \frac{1}{4}$ and none if $R > \frac{1}{4}$, which could be the most frequent case. Changing the parametrization will change the coefficients, but not the topological properties.\\\\ 
Now, it's easy to see that, if $R = \frac{1}{4}$, then 
$$
s_1=s_2=\frac{ -2\frac{1}{4}+1   }{2} = \frac{1}{4},
$$
but we know that $s \in[-w,w]$ with $w <R= \frac{1}{4}$, so this solution is actually impossible.\\
On the other hand, if $R < \frac{1}{4}$, then $ -2R+1 > - \frac{1}{2} +1=\frac{1}{2}$. So
$$
s_2=\frac{ -2R+1 + \sqrt{-4R+1}  }{2} >   \frac{  \frac{1}{2} + 0  }{2} =\frac{1}{4},
$$
but again $s \in[-w,w]$ with $w <R< \frac{1}{4}$, so this solution is impossible.\\
Finally, there are no limitations for $s_1$ given by $R< \frac{1}{4}$; $s_1$ can be any real number in $[-w,w]$ with $w <R< \frac{1}{4}$, so the only condition one can ask is:
\begin{align*}
-\frac{1}{4} < s_1 < \frac{1}{4} ,& \\
-\frac{1}{4} < \frac{ -2R+1 - \sqrt{-4R+1}  }{2} < \frac{1}{4},&
\\
- 1 <  -4R+2 - 2\sqrt{-4R+1} <1,&
\\
- 1 <  -4R+2 - 2\sqrt{-4R+1} 
\quad &\land	\quad
-4R+2 - 2\sqrt{-4R+1} <1,
\\
 2\sqrt{-4R+1}  <  -4R+3
\quad &\land	\quad
 2\sqrt{-4R+1} > - 4R +1,
\\
 4(-4R+1)  <  (-4R+3)^2
\quad &\land	\quad
 4 (  -4R+1 ) > (- 4R +1)^2,
\\
-16R+4  < 16R^2 -24R  +9
\quad &\land	\quad
 -16R+ 4  > 16R^2  -8R  +1  ,
\\
 16R^2 -8R  +5 >0
\quad &\land	\quad
 16R^2  +8R  -3 <0.
\end{align*}
Now we apply some basic rules for quadratic inequalities. The first inequality is solved by:
$$
\frac{\Delta}{4}= 16 - 16\cdot 5=-16 \cdot 4<0,
$$
that says that the inequality is always true.\\
The second inequality is solved, instead, by computing
$$
\frac{\Delta}{4}= 16 + 16 \cdot 3=16 \cdot 4 =8^2  , 
$$
and observing the two solutions of the associated equation are
$$
R_{1,2}= \frac{   -4 \pm 8            }{16} . \quad \text{Then}
$$
$$
R_{1}=- \frac{12}{16}=-\frac{3}{4} \quad  \text{ and } \quad  R_{2}= \frac{4}{16}=\frac{1}{4} .
$$
So the second inequality (and hence the whole system) has solutions:
$$
-\frac{3}{4}  <  R <  \frac{1}{4},
$$
which does not give more limitations than what was known already, that is: $0 < R < \frac{1}{4}$.\\\\
Final conclusion: the Möbius strip has at most only one critical point and, in this parametrization, it is the one obtained by
$$
(r,s_1)= \left (0,\frac{ -2R+1 - \sqrt{-4R+1}  }{2} \right ),
$$
when $0< R < \frac{1}{4}$.
\end{obs}




%
\chapter*{Acknowledgments}            
\addcontentsline{toc}{chapter}{Acknowledgments} 
\markboth{\MakeUppercase{Acknowledgments}}{\MakeUppercase{Acknowledgments}}

I would like to thank my adviser, university lecturer Ilkka Holopainen, for the work done together and the time dedicated to me. I also want to thank professors Bruno Franchi, Raul Serapioni and Pierre Pansu for the stimulating discussions. Finally I thank professor Pekka Pankka, and my adviser again, for their careful revision of this work.\\\\
This work was supported by MAnET, a Marie Curie Initial Training Network (FP7-PEOPLE-2013-ITN, no. 607643).


\begin{thebibliography}{40}
\addcontentsline{toc}{chapter}{Bibliography} 

\bibitem{BAL} Balogh Z.M., \textit{Size of characteristic sets and functions with prescribed gradient}, J. Reine Angew. Math. 564, 63–83, 2003.



\bibitem{BRSC} Balogh Z.M., Rickly M., Serra Cassano F., \textit{Comparison of Hausdorff Measures with Respect to the Euclidean and the Heisenberg Metric}, Publ. Mat. 47, 237-259, 2003.

\bibitem{BTW} Balogh Z.M., Tyson J., Warhurst B., \textit{Sub-Riemannian vs. Euclidean dimension comparison and fractal geometry on Carnot groups}, Adv. Math. 220, 560-619, 2009.

\bibitem{GCmaster} Canarecci, G. \textit{Analysis of the Kohn Laplacian on the Heisenberg Group and on Cauchy--Riemann Manifolds}, MSc Thesis, 2014.

\bibitem{CDPT} Capogna, L., Danielli, D., Pauls, S.D., Tyson, J., \textit{An Introduction to the Heisenberg Group and the Sub-Riemannian Isoperimetric Problem}, Birkhäuser Verlag AG, Basel - Boston - Berlin, 2007.

\bibitem{FED} Federer H., \textit{Geometric Measure Theory}, Grundlehren math. Wiss., Band 153, Springer-Verlag, Berlin and New York, 1969. 


\bibitem{FSSC2001} Franchi B., Serapioni R., Serra Cassano F., \textit{Rectifiability and perimeter in the Heisenberg group}, Math. Ann. 321, 479–531, 2001.

\bibitem{FSSC} Franchi B., Serapioni R., Serra Cassano F., \textit{Regular submanifolds, graphs and area formula in Heisenberg groups}, Adv. Math. 211(1), 152–203, 2007.

\bibitem{WAVE} Franchi B., Tesi M.C., \textit{Wave and Maxwell’s equations in Carnot groups}, Commun. Contemp. Math. 14(5):1250032, 62, 2012.

\bibitem{TRIP} Franchi B., Tripaldi F., \textit{Differential Forms in Carnot Groups After M. Rumin : an Introduction}, Quaderni dell'Unione Matematica Italiana, Topics in mathematics, Bologna, Pitagora, 75 - 122, 2015.

\bibitem{GROMOV} Gromov M., \textit{Carnot-Carathéodory spaces seen from within}, in: Sub-Riemannian Geometry (eds. A. Bel- la\"{i}che, J.-J. Risler), Progr. Math. 144, Birkh\"{a}user, Basel, 79–323, 1996.

\bibitem{KSC} Kirchheim B., Serra Cassano F., \textit{Rectifiability and parameterizations of intrinsically regular surfaces in the Heisenberg group}, Ann. Sc. Norm. Sup. Pisa Cl. Sci. (5) III, 871–896, 2004.

\bibitem{KORREI} Korányi A., Reimann H.M., \textit{Foundations for the Theory of Quasiconformal Mappings on the Heisenberg Group}, Adv. Math. 111, 1-87, 1995.

\bibitem{KORR} Korányi A., Reimann H.M., \textit{Quasiconformal mappings on the Heisenberg group}, Invent. math. 80, 309-338, 1985.

\bibitem{MAG} Magnani V., \textit{Characteristic points, rectifiability and perimeter measure on stratified groups}, J. Eur. Math. Soc. 8, 585-609, 2006.


\bibitem{MORGAN1} Morgan F., \textit{Geometric Measure Theory, a beginner's guide}, Elsevier/Academic Press, Amsterdam, fourth edition, 2009.

\bibitem{MORGAN2} Morgan F., \textit{Harnack type mass bounds and Bernstein theorems for area-minimizing flat chains modulo $v$}, Communications in Partial Differential Equations 11 (12), 1257-1283, 1986.

\bibitem{PANSU} Pansu P., \textit{Métriques de Carnot--Carathéodory et quasiisométries des espaces symétriques de rang un}, Ann. of Math. 129, 1-60, 1989.

\bibitem{RUMIN} Rumin M., \textit{Formes differentielles sur les varietes de contact}, J. Diff. Geom. 39, 281-330, 1994.

\bibitem{SIMON} Simon L., \textit{Introduction to Geometric Measure Theory}, previously \textit{Lectures on geometric measure theory}, Australian National Univ., Canberra, 1983.








%
%
%
%
%
%
%
%
%
%
%
%
%




%
%
\end{thebibliography}
\end{document}